\documentclass[final]{amsart}

\usepackage{booktabs}
\usepackage{placeins}
\usepackage{fullpage}
\usepackage{tristan}
\usepackage{algorithm}
\usepackage{algpseudocode}

\usepackage[textsize=tiny,textwidth=2cm]{todonotes}
\usepackage{csquotes}

\usepackage{pgfplots}
\usepackage{pgfplotstable}
\usepackage{subcaption}
\usepackage{tcolorbox}
\usepackage{tikz}
\usepackage{tikz-cd}
\usepackage[giveninits=true,eprint=false,url=false,style=alphabetic]{biblatex}

\renewcommand{\cS}[0]{{\ensuremath{\mathbb{S}}^{2}}}

\newcommand{\vertiii}[1]{{\left\vert\kern-0.25ex\left\vert\kern-0.25ex\left\vert #1
    \right\vert\kern-0.25ex\right\vert\kern-0.25ex\right\vert}}

\usepackage{xcolor,cancel}

\author[1,2]{Tristan Pryer}

\author[3]{Nikolaos Rekatsinas}

\address{$^1$ Institute for Mathematical Innovation\\ University of
  Bath, Bath, UK. $^2$ Department of Mathematical Sciences
  \\ University of Bath, Bath, UK. $^3$ Institute of Computational and
  Applied Mathematics, Foundation for Research and Technology (FORTH),
  GR 70013, Crete, Greece}

\pgfplotsset{compat=1.18}
\usepackage[normalem]{ulem}

\renewcommand{\new}[1]{{\color{black}{#1}}}

\newcommand{\rev}[1]{{\color{black}#1}}
\hypersetup{hypertexnames=false}
\makeatletter
\renewcommand{\subsection}{\@ifstar{\oldsubsection*}{\stepcounter{custom}\oldsubsection}}
\makeatother

\begin{document}

\title{Residual Minimisation for Transport-Based Optimal Control}

\begin{abstract}
We consider an optimal control problem governed by a stationary linear
transport equation in two or three spatial dimensions, motivated by
photon transport in radiotherapy treatment planning. The control is an
inflow boundary source, and the observation operator may track either
the angular flux or its angular integral, the latter serving as a
simplified dose surrogate. The associated optimality system consists
of a transport--adjoint pair with angular dependence, coupled through
a boundary optimality condition.

We formulate a residual least-squares functional whose minimisers
coincide with solutions of the Karush--Kuhn--Tucker system and show
that  \new{the } functional controls the $L^2$ \new{errors in the state and adjoint
and the flux-weighted $L^2$ error in the control} . This yields
a continuum-level a posteriori estimate without recourse to
discretisation. \new{An additional unweighted optimality residual certifies the control in the norm of the objective.} \new{The estimate holds for every admissible approximation and accounts for the different boundary measures in the control penalty and transport traces, without a positive lower bound on $|\omega\cdot n|$.} We then parameterise the state, adjoint and control
using neural networks and minimise over a neural network space using
appropriate quadrature. \new{The Monte Carlo formulation replaces each
residual integral by a weighted sample mean, whose quadrature error
enters the a posteriori estimate.}

We prove convergence under \new{explicit assumptions on approximation, quadrature and
optimisation and derive } bounds on the total error in terms of
approximation, quadrature and optimisation error. Numerical examples in two and
three dimensions \new{illustrate the residual--error relation and the
flexibility of the residual framework for field- and dose-tracking control
problems, including regular manufactured solutions}.
\end{abstract}

\keywords{Transport equations, boundary control, optimality systems,
residual stability, a posteriori error estimates, neural networks}
\subjclass[2020]{49K40, 49K20, 35F15, 65N15, 68T07}

\maketitle

\section{Introduction}

Optimal control problems governed by transport equations arise in a
range of applied settings, from radiation therapy and shielding design
to optical engineering and neutron flux regulation. In such problems,
the control is typically applied at the inflow boundary of the domain,
and the system response propagates along characteristics with angular
structure and attenuation \cite{revFrankHertySandjo2010}.

This work is motivated by photon-based radiotherapy treatment
planning, where an incident fluence is chosen to shape the radiation
field within the patient
\cite{revBortfeld2006,revFrankHertySandjo2010}. Related
transport-control questions arise in proton therapy, where energy loss
and biological response also enter the model
\cite{ashby2025efficient}. Here we isolate the boundary-control
structure using a linear Boltzmann-type equation with pure
attenuation.  The control enters through the inflow boundary, and we
consider either angular-flux tracking or tracking of its angular
integral, the latter providing a simplified dose surrogate. The target-shape computations in Section~\ref{sec:numerics} compare angular-flux tracking with direct dose tracking. In the field-tracking
examples, the angular integral is reported as an additional spatial diagnostic. Scattering, energy dependence, nonnegative delivery constraints and more complete dose models are left outside
the present formulation.

We formulate a residual least-squares functional that encodes the full
KKT system for the transport-based control problem. The residual
functional \new{$\mathcal J_b$} is quadratic, non-negative and vanishes only
at the solution. We show that \new{$\mathcal J_b$ } provides a rigorous a
posteriori estimate at the continuum level: its value controls the
error in all components of the solution triple $(u, z, g)$. This
enables residual-based convergence analysis and error certification
independently of discretisation. \new{The control error is measured in the
transport boundary norm, with weight $|\omega\cdot n|$.
The distinction between this norm and the unweighted control penalty
is essential: the normal velocity enters the boundary optimality
condition and may vanish at grazing directions. We establish stability
without dividing by this velocity or excluding nearly grazing rays.}

The unknowns are approximated using smooth feedforward neural networks
and trained to minimise \new{a quadrature approximation of $\mathcal J_b$ using
gradient-based methods}. \new{The networks provide an approximation
class on position-angle space, and automatic differentiation
supplies the directional derivatives in the residual. The analysis
starts from the continuum optimality system, so the stability estimate
also applies to other approximations for which these residuals are
well-defined.} We analyse the resulting approximate
residual minimisers under both deterministic and Monte Carlo
quadrature and prove convergence under \new{explicit approximation, quadrature and
optimisation hypotheses. Numerical experiments examine the relation
between residual and error and } demonstrate the feasibility of the
approach in two and three dimensions.

\rev{Residual minimisation of full optimality systems is established in least-squares finite element methods \cite{revBochevGunzburger2006}.}
\rev{For distributed optimal control, least-squares norm equivalence and reliable error estimation is obtained in \cite{revFuhrerKarkulik2023}.} Adaptive regularisation for elliptic control is studied in
\cite{powerPryer2025adaptive}. \rev{For transport equations, stable Petrov--Galerkin and least-squares formulations relate suitable residual norms to approximation errors
\cite{dahmen2012adaptive,revLiuZhang2020}.} \rev{Flux-weighted trace spaces are established in the forward transport framework of \cite{dahmen2012adaptive}}.
\rev{For the present control problem, the state and adjoint estimates must be combined with the
boundary optimality equation. The relevant stability question therefore
concerns the coupled system, including the error in its unknown
inflow data.} \rev{Here the boundary optimality equation couples an unweighted
control penalty to flux-weighted transport traces. We therefore
distinguish residual estimates for the control in the transport norm
from those in the norm used by the objective.}

\new{Neural network approaches to PDE-constrained optimisation include
\cite{revMowlaviNabi2023,lu2021physics}. In particular, joint neural
approximation of the state, adjoint and control through optimality
conditions is already used in \cite{revBarryStraume2026}. The recent
KKT-based method of \cite{revZhang2026} combines adversarial sampling
with scaling to address sensitivity to the regularisation parameter.}
\new{Deep Uzawa methods treat PDE and boundary constraints through
multipliers \cite{makridakis2024deep,makridakis2024deepbcs}, while the related
stationary transport method \cite{makridakis2025kineticuzawa}
enforces prescribed inflow data. Here the inflow is itself the control,
and its optimality condition couples the boundary trace of the adjoint
to the regularisation term. This coupling must be retained when
relating the residual to errors in all three unknowns.}

The main contribution is a global residual-stability analysis for the
transport boundary-control KKT system. It accounts simultaneously for
the unweighted control penalty and the flux-weighted transport traces,
without excluding nearly grazing directions. For each fixed
$\alpha>0$, the continuum residual controls the state and adjoint
errors in $L^2$ and the control error in the transport boundary norm,
independently of the particular approximation method.
\rev{A counterexample supported near grazing directions shows that
convergence of this weighted residual to zero need not imply convergence of
the control in the unweighted norm appearing in the objective.
Thus a small weighted KKT residual need not guarantee that the computed control is nearly optimal for the original control problem, as explained in
Remark~\ref{rem:grazing-control}.}
Strengthening only the optimality residual yields control in the unweighted
norm appearing in the objective and hence a bound on reduced-objective
suboptimality\rev{, while retaining flux-weighted state and adjoint
boundary residuals}.

This stability estimate separates the numerical error into
approximation, quadrature and optimisation contributions. For Monte
Carlo quadrature we distinguish estimates for a fixed approximation
from uniform estimates over a bounded training class, since a network
selected using the sampled loss is itself sample-dependent.
Deterministic quadrature fits the same residual-error framework.
The numerical experiments test the coupled state--adjoint--control
approximation in two and three dimensions, compare target-control
examples with a characteristic reference, and separately verify the
fixed-field Monte Carlo quadrature rate.

\new{The rest of the paper is organized as follows.
Section~\ref{sec:problem} introduces the control problem and establishes
well-posedness of the transport equations and their coupled optimality
system. Section~\ref{sec:residual-stability} characterises the residual
minimisers and proves the global stability estimate.
Section~\ref{sec:NN} describes the neural network parameterisation, quadrature
loss and minimisation algorithm. Section~\ref{sec:quadrature} relates
the discrete loss to the solution error through Monte Carlo and
deterministic quadrature estimates. Approximation and convergence are
considered in Section~\ref{sec:approx-convergence}. The numerical
results are presented in Section~\ref{sec:numerics}, followed by
conclusions in Section~\ref{sec:conclusions}.}

\section{Problem Setup}\label{sec:problem}

We study an optimal control problem arising in photon transport for
radiotherapy treatment planning. The aim is to determine an incident
flux that delivers a prescribed dose to a target region within tissue,
modelled as a bounded, convex \new{Lipschitz domain $\Omega\subset\mathbb R^d$,
with $d=2$ } or $3$. The state of a photon is described by its spatial
position \new{$\vec x\in\Omega$ and direction $\vec\omega\in\mathbb S^{d-1}$.
The flux $u=u(\vec x,\vec\omega)$ represents the photon intensity at
the corresponding} phase-space point $(\vec x, \vec \omega)$.

The full radiative transport model incorporates absorption, scattering
and external sources, but to clarify the control structure we first
consider the simplified setting without scattering or internal
sources. In this case, the transport of photons is governed by a pure
attenuation equation:
  \begin{equation}\new{\label{eq:transport-pde}
 \vec\omega\cdot\nabla_{\vec x}u+\sigma(\vec x)u=0
 \quad\text{in }\Omega\times\mathbb S^{d-1},
}\end{equation}
with inflow boundary condition
  \begin{equation}\new{\label{eq:transport-bc}
 u|_{\Gamma^-}=g.
}\end{equation}
\new{Here $\sigma\in L^\infty(\Omega)$ is nonnegative, and the inflow and
outflow boundaries are
}\[
 \new{\Gamma^\pm=\{(x,\omega)\in\partial\Omega\times\mathbb S^{d-1}:
                      \pm\omega\cdot n(x)>0\}.
}\]
The control variable is the inflow flux $g$, prescribed on $\Gamma^-$,
and the forward field $u$ solves
\eqref{eq:transport-pde}\new{-}\eqref{eq:transport-bc}. For treatment
planning, we aim to deliver dose to a region    \new{$T\subset\Omega$.
In the present model, the angular integral
}\[
 \new{D(u)(x)=\int_{\mathbb S^{d-1}}u(x,\omega)\,d\omega
}\]
\new{provides a simplified dose surrogate. We may therefore match
either the angular flux itself or its angular integral to a prescribed
target. We write $\mathcal O u$ for the chosen quantity: $\mathcal O=I$
for angular-flux tracking and $\mathcal O=D$ for dose tracking, with
$Y=L^2(\Omega\times\mathbb S^{d-1})$ and $Y=L^2(\Omega)$,
respectively.
}

We \new{consider the constrained minimisation problem}
\begin{equation}\new{\label{eq:corrected-objective}
 \min_{(u,g)}\ \frac{1}{2}\|\mathcal O u-X_T\|_Y^2
       +\frac{\alpha}{2}\|g\|_{L^2(\Gamma^-)}^2
 \quad\text{subject to }\eqref{eq:transport-pde},\ \eqref{eq:transport-bc}.
}\end{equation}
This is a linear-quadratic optimal control problem with PDE constraint
and \emph{boundary control}, where \new{$\alpha>0$ balances fidelity to
the target and control cost. The target $X_T$ belongs to $Y$. For angular-flux tracking it may
depend on both $x$ and $\omega$. A spatial target is understood to be
constant in the angular variable. The analysis below applies to
any bounded linear $\mathcal O:L^2(\Omega\times\mathbb S^{d-1})\to Y$.
\rev{The minimisation is over controls without a sign constraint.}}
\new{Quadratic dose-tracking objectives for kinetic treatment planning
are also studied in \cite{revBarnard2012}.}

 \subsection{\new{Transport graph space and boundary trace}}\label{subsec:trace}
\new{We write $Q=\Omega\times\mathbb S^{d-1}$, $H=L^2(Q)$ and
$b(x,\omega)=|\omega\cdot n(x)|$. The control cost in
\eqref{eq:corrected-objective} uses the unweighted boundary space
$U=L^2(\Gamma^-,dS_x\,d\omega)$, whereas integration by parts for
transport uses
}\[
 \new{G_\pm=L^2(\Gamma^\pm,b\,dS_x\,d\omega).
}\]

\new{The space $U$ describes the control penalty, $G_\pm$ describe the
transport boundary traces, and the space $W$ below describes the state
and adjoint. These are different roles. In particular,
\[
 \|h\|_{G_-}^2=\int_{\Gamma^-}b|h|^2\,dS_x\,d\omega
 \le \int_{\Gamma^-}|h|^2\,dS_x\,d\omega=\|h\|_U^2.
\]
Thus $U\hookrightarrow G_-$ continuously. The converse bound fails
near grazing directions. In the control objective we penalise the
incident intensity with respect to boundary area and solid angle, thus
the flux factor $b$ instead enters the transport identity. This choice
of control penalty is retained throughout.}
All angular integrals use surface measure on $\mathbb S^{d-1}$.
If $v\in H$ and $\omega\cdot\nabla_xv\in H$, then along almost every
characteristic the restriction of $v$ has an $H^1$ representative and
therefore well-defined endpoint values. We denote these characteristic
endpoint values by $\gamma^-v$ and $\gamma^+v$ whenever the resulting
boundary functions belong to $G_-$ and $G_+$, respectively.

We define
\begin{equation}\label{eq:corrected-trace-space}
 W=\{v\in H:\omega\cdot\nabla_xv\in H,\
                   \gamma^-v\in G_-,\ \gamma^+v\in G_+\},
\end{equation}
with norm
\[
 \|v\|_W^2
 =\|v\|_H^2+\|\omega\cdot\nabla_xv\|_H^2
  +\|\gamma^-v\|_{G_-}^2+\|\gamma^+v\|_{G_+}^2.
\]
Thus, the existence of finite global transport traces is part of the
definition of $W$. The transport graph norm alone need not control both
boundary traces. The final two terms in $\|\cdot\|_W$ make this
requirement explicit. This distinction also occurs in the transport
framework of \cite{dahmen2012adaptive}.

\begin{lemma}[Green identity for the transport operator]
\label{lem:trace-green}
\new{The trace maps $\gamma^\pm:W\to G_\pm$ are linear and continuous,
with
}\[
 \new{\|\gamma^\pm v\|_{G_\pm}\le\|v\|_W\qquad(v\in W).
}\]
\new{Let $T=\omega\cdot\nabla_x+\sigma$ and
$T^*=-\omega\cdot\nabla_x+\sigma$. For $v,w\in W$,
}\begin{equation}\new{\label{eq:corrected-green}
 (Tv,w)_H-(v,T^*w)_H
 = (\gamma^+v,\gamma^+w)_{G_+}
      -(\gamma^-v,\gamma^-w)_{G_-}.
}\end{equation}
\end{lemma}
\begin{proof}
\new{The trace bounds follow from the definition of $\|\cdot\|_W$.
Along almost every characteristic, a function and its directional
derivative belong to $L^2$, so its restriction belongs to $H^1$ on
that interval. One-dimensional integration by parts therefore applies.
For each fixed $\omega$, orthogonal projection of the inflow and
outflow boundaries onto $\omega^\perp$ has Jacobian
$b=|\omega\cdot n|$. Integrating the one-dimensional identity over
these projected characteristics and over $\omega$ gives
\eqref{eq:corrected-green}. The volume and boundary products are
integrable by Cauchy--Schwarz and the definition of $W$.
}\end{proof}

\new{For completeness, we record the transport estimates used below.
In the present convex geometry they follow directly from the
characteristic formula and require only nonnegative attenuation.
}

\begin{theorem}[\new{Well-posedness of the transport equation}]
\label{thm:transport_wellposed}
\new{Let $L=\operatorname{diam}(\Omega)$. For $f\in H$ and $h\in G_-$,
the problem $Tv=f$, $\gamma^-v=h$, has a unique solution $v\in W$,
and
}\begin{equation}\new{\label{eq:corrected-operator-bounds}
 \|v\|_H\le L\|f\|_H+\sqrt L\|h\|_{G_-},\qquad
 \|\gamma^+v\|_{G_+}\le\sqrt L\|f\|_H+\|h\|_{G_-}.
}\end{equation}
\new{For $f\in H$ and $k\in G_+$, the adjoint problem
$T^*w=f$, $\gamma^+w=k$, has a unique solution $w\in W$, with
}\begin{equation}\new{\label{eq:adjoint-transport-bounds}
 \|w\|_H\le L\|f\|_H+\sqrt L\|k\|_{G_+},\qquad
 \|\gamma^-w\|_{G_-}\le\sqrt L\|f\|_H+\|k\|_{G_+}.
}\end{equation}
\end{theorem}
\begin{proof}
\new{Let $x_-$ be the inflow endpoint of a characteristic and let
$0<t<\ell(x_-,\omega)\le L$ denote distance along it. Define
}\begin{equation}\new{\label{eq:corrected-characteristics}
 }\begin{split}
 \new{v(x_-+t\omega,\omega)
 =}{}&\new{\exp\left(-\int_0^t\sigma(x_-+s\omega)\,ds\right)h(x_-,\omega)}\\
 &\new{+\int_0^t\exp\left(-\int_s^t\sigma(x_-+r\omega)\,dr\right)
                    f(x_-+s\omega,\omega)\,ds.
 }\end{split}
\new{}\end{equation}
\new{The characteristic change of variables is
}\[
 \new{\int_Q\varphi(x,\omega)\,dx\,d\omega
 =\int_{\Gamma^-}\int_0^{\ell(x_-,\omega)}
       \varphi(x_-+t\omega,\omega)\,dt\,
           b(x_-,\omega)\,dS_{x_-}\,d\omega.
}\]
\new{Indeed, for each direction the orthogonal projection used in
Lemma~\ref{lem:trace-green} parametrises the straight characteristics.
Convexity gives one interval on each ray. The same projection shows
that the endpoint map from inflow to outflow preserves
$b\,dS_x\,d\omega$. Since $\sigma\ge0$, the attenuation factors in
\eqref{eq:corrected-characteristics} are at most one. For the homogeneous
term, integration along intervals of length at most $L$ gives the
bounds $\sqrt L\|h\|_{G_-}$ in $H$ and $\|h\|_{G_-}$ at outflow.
Cauchy--Schwarz gives the corresponding bounds $L\|f\|_H$ and
$\sqrt L\|f\|_H$ for the source term. Thus
\eqref{eq:corrected-operator-bounds} follows.
The formula has inflow trace $h$, and
$\omega\cdot\nabla_xv=f-\sigma v\in H$. The outflow estimate then
gives $v\in W$. Uniqueness follows from the homogeneous ordinary
differential equation on almost every characteristic.
Reversing the direction gives the adjoint statement and
\eqref{eq:adjoint-transport-bounds}.
}\end{proof}

 \subsection{\new{Lagrangian formulation and optimality conditions}}
\new{We use the standard Lagrangian approach to PDE-constrained
optimisation \cite{revHinze2009}.}
\new{We introduce the Lagrangian associated with
\eqref{eq:corrected-objective} by incorporating the transport equation
through a dual variable $z\in W$:
}\begin{equation}\new{\label{eq:lagrangian}
 \mathcal L(u,g,z)=\frac{1}{2}\|\mathcal O u-X_T\|_Y^2
             +\frac{\alpha}{2}\|g\|_U^2-(Tu,z)_H.
}\end{equation}
\new{The inflow condition is retained in the admissible pairs $(u,g)$.
Thus their variations satisfy $\gamma^-\delta u=\delta g$.
Using \eqref{eq:corrected-green}, the derivative in these directions is
}\begin{align*}
 \new{D_{(u,g)}\mathcal L}[\new{\delta u,\delta g}]
 \new{=}{}&\new{(\mathcal O^*(\mathcal O u-X_T)-T^*z,\delta u)_H
       -(\gamma^+z,\gamma^+\delta u)_{G_+}}\\
    &\new{+(\alpha g+b\gamma^-z,\delta g)_U.
}\end{align*}
\new{Stationarity therefore yields the KKT system
}\begin{equation}\new{\label{eq:corrected-kkt}
 \begin{gathered}
 T^*z^*=\mathcal O^*(\mathcal O u^*-X_T),\qquad\gamma^+z^*=0,\\
 \alpha g^*+b\gamma^-z^*=0,\\
 Tu^*=0,\qquad\gamma^-u^*=g^*.
 \end{gathered}
}\end{equation}
\new{The boundary condition is equivalently
$\alpha g^*=(\omega\cdot n)\gamma^-z^*$ on $\Gamma^-$.
The factor $b$ arises because the control penalty is unweighted,
whereas the boundary term in Green's identity is flux-weighted.
For angular-flux tracking the adjoint source is $u^*-X_T$. For dose
tracking it is $D(u^*)-X_T$, regarded as a function constant in $\omega$.
Indeed, $D^*p(x,\omega)=p(x)$ and
$\|D\|\le |\mathbb S^{d-1}|^{1/2}$.
}

\begin{theorem}[\new{Well-posedness of the optimality system}]
\label{thm:corrected-wellposed}
\new{For every $X_T\in Y$ and $\alpha>0$, problem
\eqref{eq:corrected-objective} has a unique minimiser. Its state,
adjoint and control form the unique solution
$(u^*,z^*,g^*)\in W\times W\times U$ of
\eqref{eq:corrected-kkt}. Moreover,
}\begin{equation}\new{\label{eq:corrected-apriori}
 \|g^*\|_U\le\alpha^{-1/2}\|X_T\|_Y,\qquad
 \|u^*\|_H\le\sqrt{L/\alpha}\|X_T\|_Y,\qquad
 \|z^*\|_H\le L\|\mathcal O\|\|X_T\|_Y.
}\end{equation}
\end{theorem}
\begin{proof}
\new{Let $Sh$ denote the solution of $T(Sh)=0$ with $\gamma^-(Sh)=h$.
Theorem~\ref{thm:transport_wellposed} gives
$\|Sh\|_H\le\sqrt L\|h\|_{G_-}$.
Since $b\le1$, this estimate also holds with $\|h\|_U$ on the right.
The reduced problem is to minimise
}\[
 \new{j(g)=\frac{1}{2}\|\mathcal O Sg-X_T\|_Y^2+\frac{\alpha}{2}\|g\|_U^2.
}\]
\new{Its stationarity equation is
}\[
 \new{\alpha(g,h)_U+(\mathcal O Sg,\mathcal O Sh)_Y
                         =(X_T,\mathcal O Sh)_Y\qquad(h\in U).
}\]
\new{The bilinear form on the left is bounded and coercive on $U$ with
coercivity constant $\alpha$. Lax--Milgram gives a unique stationary
point $g^*$, which minimises $j$ by strict convexity.
Set $u^*=Sg^*$ and let $z^*$ solve the adjoint equation with zero
outflow data. Green's identity gives
}\[
 \new{(\mathcal O u^*-X_T,\mathcal O Sh)_Y
   =(h,\gamma^-z^*)_{G_-}=(h,b\gamma^-z^*)_U.
}\]
\new{Here $b\gamma^-z^*\in U$, since
$\|bp\|_U\le\|p\|_{G_-}$. The stationarity equation is therefore
the boundary optimality condition in \eqref{eq:corrected-kkt}.
Conversely, that system implies stationarity of $j$.
Finally, $j(g^*)\le j(0)$ implies
$\sqrt\alpha\|g^*\|_U\le\|X_T\|_Y$ and
$\|\mathcal O u^*-X_T\|_Y\le\|X_T\|_Y$.
The transport estimates give \eqref{eq:corrected-apriori}.
}\end{proof}

\section{\new{Residual Stability}}\label{sec:residual-stability}
\new{Following the least-squares principle \cite{revBochev2009}, we
measure all equations of the optimality system through their residuals.}
\new{The optimality system may also be characterised by the residual
functional
}\begin{equation}\new{\label{eq:corrected-residual}
 }\begin{split}
 \new{\mathcal J_b(u,z,g)=\frac{1}{2}\big(}&\new{\|Tu\|_H^2
   +\|T^*z-\mathcal O^*(\mathcal O u-X_T)\|_H^2
   +\|\alpha g+b\gamma^-z\|_{G_-}^2}\\
   &\new{+\|\gamma^-u-g\|_{G_-}^2+\|\gamma^+z\|_{G_+}^2\big).
 }\end{split}
\end{equation}
Although the original control problem is posed with $g\in U$, the
weighted residual $\mathcal J_b$ is well defined on the larger space
$W\times W\times G_-$. Since $U\hookrightarrow G_-$, this enlarges the
residual domain without changing the zero-residual solution: as shown
below, vanishing of the boundary optimality residual forces
$g=-\alpha^{-1}b\gamma^-z\in U$.

\new{Each term measures the violation of one interior or boundary equation.
This formulation is convenient for neural network approximation,
because the state, adjoint and control can be approximated together
without imposing the constraints strongly. The boundary residuals use
the transport measure, and the control error below is measured in the
same norm.
}

\begin{lemma}[\new{Characterisation of residual minimisers}]
\label{lem:residual-minimisers}
\new{The solution $(u^*,z^*,g^*)$ of \eqref{eq:corrected-kkt} is the
unique global minimiser of $\mathcal J_b$ over $W\times W\times G_-$,
and its minimum value is zero. In particular,
}\[
\new{\mathcal J_b(u,z,g)=0
\quad\Longleftrightarrow\quad
(u,z,g)=(u^*,z^*,g^*).
}\]
\end{lemma}
\begin{proof}
\new{The KKT solution makes all five residuals vanish. Conversely,
if $\mathcal J_b(u,z,g)=0$, each squared residual norm is zero.
The boundary optimality equation gives
$g=-\alpha^{-1}b\gamma^-z\in U$, since
$\|b\gamma^-z\|_U\le\|\gamma^-z\|_{G_-}$.
Thus $(u,z,g)\in W\times W\times U$ satisfies
\eqref{eq:corrected-kkt}, and uniqueness follows from
Theorem~\ref{thm:corrected-wellposed}. Nonnegativity of
$\mathcal J_b$ completes the proof.}
\end{proof}

\begin{theorem}[\new{Global residual stability}]
\label{thm:corrected-coercivity}
\new{Under the hypotheses of Theorem~\ref{thm:corrected-wellposed}, there
is a constant $C=C(\alpha,L,\|\mathcal O\|)>0$ such that every
$(u,z,g)\in W\times W\times G_-$ satisfies
}\begin{equation}\new{\label{eq:corrected-coercivity}
 \|u-u^*\|_H^2+\|z-z^*\|_H^2+\|g-g^*\|_{G_-}^2
                      \le C\mathcal J_b(u,z,g).
}\end{equation}
\new{An explicit choice is $C=C_{\rm stab}$ in
\eqref{eq:explicit-stability-constant} below.}
\end{theorem}
\begin{proof}
\new{Write $e_u=u-u^*$, $e_z=z-z^*$ and $e_g=g-g^*$, and denote the
five residuals in \eqref{eq:corrected-residual}, in order, by
$r_1,\ldots,r_5$. Then
}\[
 \new{Te_u=r_1,\qquad T^*e_z=\mathcal O^*\mathcal O e_u+r_2,
 \qquad\alpha e_g+b\gamma^-e_z=r_3,
 \qquad\gamma^-e_u=e_g+r_4,\quad\gamma^+e_z=r_5.
}\]
\new{Let $S$ denote the state lift used above, and let $y$ solve
$Ty=r_1$, $\gamma^-y=r_4$. Then $e_u=Se_g+y$.
Let $\eta$ solve
}\[
 \new{T^*\eta=\mathcal O^*\mathcal O y+r_2,
                         \qquad\gamma^+\eta=r_5.
}\]
The difference $w=e_z-\eta$ satisfies
\[
 T^*w=\mathcal O^*\mathcal OSe_g,\qquad \gamma^+w=0.
\]
The state lift $S:G_-\to H$ is bounded, with
$\|S\|\le\sqrt L$ by Theorem~\ref{thm:transport_wellposed}.
Hence
\[
 A:=S^*\mathcal O^*\mathcal O S:G_-\to G_-
\]
is bounded, self-adjoint and nonnegative, and
\[
 \|A\|\le L\|\mathcal O\|^2.
\]
For every $h\in G_-$, Green's identity gives
\[
 (h,\gamma^-w)_{G_-}
 =(Sh,\mathcal O^*\mathcal OSe_g)_H
 =(h,Ae_g)_{G_-},
\]
so that $\gamma^-w=Ae_g$.

Let $P:G_-\to G_-$ denote multiplication by $b$. Since
$0\le b\le1$, $P$ is a bounded, self-adjoint, nonnegative contraction,
and $P^{1/2}$ is multiplication by $\sqrt b$. The boundary equation
therefore becomes
\begin{equation}\label{eq:corrected-reduced-error}
 (\alpha I+PA)e_g=r_3-P\gamma^-\eta.
\end{equation}

Set
\[
 B=P^{1/2}AP^{1/2}.
\]
Since $A$ is self-adjoint and nonnegative, so is $B$. Thus
\[
 ((\alpha I+B)v,v)_{G_-}\ge
 \alpha\|v\|_{G_-}^2,
\]
and consequently $\alpha I+B$ is invertible with
\[
 \|(\alpha I+B)^{-1}\|\le\alpha^{-1}.
\]
Using the resolvent identity with
$X=P^{1/2}$ and $Y=P^{1/2}A$,
\[
 (\alpha I+XY)^{-1}
 =\alpha^{-1}\bigl[I-X(\alpha I+YX)^{-1}Y\bigr],
\]
we obtain
\begin{equation}\label{eq:corrected-inverse}
 (\alpha I+PA)^{-1}
 =\alpha^{-1}\left[
 I-P^{1/2}(\alpha I+P^{1/2}AP^{1/2})^{-1}P^{1/2}A
 \right].
\end{equation}
In particular, using $\|P^{1/2}\|\le1$,
\[
 \|(\alpha I+PA)^{-1}\|
 \le \alpha^{-1}\left(1+\frac{\|A\|}{\alpha}\right)
 \le \kappa:=
 \alpha^{-1}\left(1+\frac{L\|\mathcal O\|^2}{\alpha}\right).
 \]
 \new{Writing $R_i=\|r_i\|$ in its respective residual space, the
transport estimates imply
}\begin{align*}
 \new{\|y\|_H}&\new{\le LR_1+\sqrt L R_4,}\\
 \new{\|\gamma^-\eta\|_{G_-}
   }&\new{\le\sqrt L(\|\mathcal O\|^2\|y\|_H+R_2)+R_5,}\\
 \new{\|e_g\|_{G_-}}&\new{\le\kappa(R_3+\|\gamma^-\eta\|_{G_-}),}\\
 \new{\|e_u\|_H}&\new{\le\sqrt L\|e_g\|_{G_-}+\|y\|_H,}\\
 \new{\|e_z\|_H}&\new{\le L(\|\mathcal O\|^2\|e_u\|_H+R_2)+\sqrt L R_5.
}\end{align*}
\new{Combining these estimates and using
$\sum_iR_i^2=2\mathcal J_b(u,z,g)$ proves
\eqref{eq:corrected-coercivity}.
}\new{An explicit choice is $C_{\rm stab}=2\|\mathsf M\|_{\rm F}^2$,
where $\mathsf M$ collects these coefficients as specified in
Appendix~\ref{app:boundary-stability}. For angular-flux tracking,
the characteristic structure gives the sharper choice
$\kappa=\alpha^{-1}$, and the dependence on $\alpha$ is examined there.}
\end{proof}

\begin{remark}[\new{Boundary norms and loss weights}]

\new{\label{rem:boundary-norms}}
\new{If the three boundary residuals in \eqref{eq:corrected-residual}
are measured in unweighted $L^2$, denote the resulting functional
by $\mathcal J_0$. Then $\mathcal J_b\le\mathcal J_0$.
In addition, the control equation and the adjoint trace estimate give
}\[
 \new{\|e_g\|_U\le\alpha^{-1}\big[
  \|r_3\|_U+\sqrt L(\|\mathcal O\|^2\|e_u\|_H+\|r_2\|_H)
                         +\|r_5\|_{G_+}\big].
}\]
\new{Consequently, $\mathcal J_0$ also controls the unweighted control
error whenever these residuals are finite. Multiplying the five
squared residual norms by fixed positive weights $\lambda_i$
preserves the stability estimate, with its constant multiplied by
$1/\min_i\lambda_i$. This includes fixed normalisations by the
volume and boundary measures.
}
\end{remark}
\begin{remark}[Weighted residual convergence and the control norm]\label{app:grazing}
\label{rem:grazing-control}
{\color{black}
Theorem~\ref{thm:corrected-coercivity} does not bound the control
error in $U$. Write $d\mu=dS_x\,d\omega$ and $A_\varepsilon=\{(x,\omega)\in\Gamma^-:0<b<\varepsilon\}$. These sets have positive measure in the present geometries. Set
$h_\varepsilon=\mu(A_\varepsilon)^{-1/2}\mathbf1_{A_\varepsilon}$.
Then $\|h_\varepsilon\|_U=1$ and
$\|h_\varepsilon\|_{G_-}^2\le\varepsilon$.
For $(u,z,g)=(u^*,z^*,g^*+h_\varepsilon)$, the residuals are
\[
 (r_1,r_2,r_3,r_4,r_5)
 =(0,0,\alpha h_\varepsilon,-h_\varepsilon,0),
 \qquad
 \mathcal J_b\le\tfrac12(1+\alpha^2)\varepsilon.
\]
Thus $\mathcal J_b\to0$ need not imply $g\to g^*$ in $U$.
The corresponding reduced-objective gap satisfies
$j(g^*+h_\varepsilon)-j(g^*)\ge\alpha/2$ by
\eqref{eq:objective-gap}. A small weighted residual therefore does not guarantee that
$j(g)$ is close to its minimum.
This does not contradict Lemma~\ref{lem:residual-minimisers}:
the zero-residual solution is the original minimiser, but convergence
of approximate minimisers is certified in the weighted norm. The following corollary obtains this stronger conclusion by measuring the optimality residual in \(U\).}
\end{remark}
\new{\begin{corollary}[Control accuracy in the objective norm]
\label{cor:unweighted-control}\label{cor:objective-gap}
Under the hypotheses of Theorem~\ref{thm:corrected-wellposed},
let $(u,z,g)\in W\times W\times U$ and define
\begin{equation}\label{eq:U-residual}
 \mathcal J_U(u,z,g)=\frac12\bigl(
 \|r_1\|_H^2+\|r_2\|_H^2+\|r_3\|_U^2
 +\|r_4\|_{G_-}^2+\|r_5\|_{G_+}^2\bigr).
\end{equation}
There is an explicit $C_U=C_U(\alpha,L,\|\mathcal O\|)$, given in
\eqref{eq:U-stability-constant}, such that
\begin{equation}\label{eq:U-stability}
 \|u-u^*\|_H^2+\|z-z^*\|_H^2+\|g-g^*\|_U^2
 \le C_U\mathcal J_U(u,z,g).
\end{equation}
Moreover, writing $e_g=g-g^*$, the reduced objective satisfies
\begin{equation}\label{eq:objective-gap}
\begin{aligned}
 j(g)-j(g^*)&=\tfrac12\|\mathcal OS e_g\|_Y^2
                    +\tfrac\alpha2\|e_g\|_U^2\\
 &\le\tfrac{\alpha+L\|\mathcal O\|^2}{2}C_U\mathcal J_U(u,z,g).
\end{aligned}
\end{equation}
The constant $C_U$ is of order $\alpha^{-2}$ for fixed $L$ and
$\|\mathcal O\|$ as $\alpha\downarrow0$.
\end{corollary}
\begin{proof}
Multiplication by $b$ maps $G_-$ into $U$ with norm at most one,
so $r_3=\alpha g+b\gamma^-z\in U$. Use $e_u=Se_g+y$ and
$e_z=w+\eta$ from the proof of Theorem~\ref{thm:corrected-coercivity}.
Testing $\alpha e_g+b\gamma^-w=r_3-b\gamma^-\eta$ against $e_g$
in $U$, Green's identity gives
\[
 \alpha\|e_g\|_U^2+\|\mathcal OSe_g\|_Y^2
 =(r_3-b\gamma^-\eta,e_g)_U,
 \qquad
 \|e_g\|_U\le\alpha^{-1}
 (\|r_3\|_U+\|\gamma^-\eta\|_{G_-}).
\]
The same transport estimates for $y$, $\eta$, $e_u$ and $e_z$,
with $\|e_g\|_{G_-}\le\|e_g\|_U$, prove \eqref{eq:U-stability}.
Their coefficients give the explicit constant in the appendix.
For \eqref{eq:objective-gap}, expand $j(g)$ about $g^*$.
The cross term vanishes by stationarity, and
$\|Se_g\|_H\le\sqrt L\|e_g\|_U$ gives the upper bound.
\end{proof}

\begin{remark}[Sharper reduced-objective bound]
\label{rem:sharper-objective-gap}
The reduced-objective estimate in Corollary~\ref{cor:unweighted-control} admits a
sharper dependence on $\alpha$.
For $(u,z,g)\in W\times W\times U$, write
$R_1=\|r_1\|_H$, $R_2=\|r_2\|_H$, $R_3=\|r_3\|_U$,
$R_4=\|r_4\|_{G_-}$ and $R_5=\|r_5\|_{G_+}$, and set
\[
\mathcal R_{\mathrm{red}}
:= R_3
+\sqrt L\left[
    \|\mathcal O\|^2(LR_1+\sqrt L R_4)+R_2
\right]+R_5.
\]
Then
\begin{equation}\label{eq:sharper-objective-gap}
j(g)-j(g^*)
\le \frac{\mathcal R_{\mathrm{red}}^2}{2\alpha}
\le
\frac{(L^3+L^2)\|\mathcal O\|^4+L+2}{\alpha}
\,\mathcal J_U(u,z,g).
\end{equation}

Indeed, let $z_g$ be the adjoint associated with the feasible
state $Sg$, and write $y=u-Sg$ and $\eta=z-z_g$.
The residual equations give
\[
Ty=r_1,\quad \gamma^-y=r_4,\qquad
T^*\eta=\mathcal O^*\mathcal O y+r_2,\quad \gamma^+\eta=r_5.
\]
Consequently, the reduced gradient satisfies
\[
\nabla j(g)=\alpha g+b\gamma^-z_g
           =r_3-b\gamma^-\eta.
\]
The transport estimates of Theorem~\ref{thm:transport_wellposed} and
$\|bp\|_U\le\|p\|_{G_-}$ yield
$\|\nabla j(g)\|_U\le\mathcal R_{\mathrm{red}}$.
Since $j$ is $\alpha$-strongly convex,
\[
j(g)-j(g^*)
\le\frac{1}{2\alpha}\|\nabla j(g)\|_U^2.
\]
The stated bound follows by Cauchy--Schwarz and
$\sum_{i=1}^5R_i^2=2\mathcal J_U(u,z,g)$.
For fixed $L$ and $\|\mathcal O\|$, the objective-gap coefficient grows
at most as $\alpha^{-1}$. The constant $C_U$ in
\eqref{eq:U-stability}, which controls the combined solution error,
retains its $\alpha^{-2}$ dependence.
\end{remark}

Only the optimality residual needs the stronger measure. No additional
unweighted traces of $u$ or $z$ are required. Since $b\le1$,
\begin{equation}\label{eq:residual-inclusion}
 \mathcal J_b\le\mathcal J_U\le\mathcal J_0
\end{equation}
wherever $\mathcal J_0$ is finite. Thus the unweighted boundary
loss also controls the error in $U$. A field trained with
$\mathcal J_b$ may instead be assessed afterwards using
$\mathcal J_U$. Weighted residual convergence alone does not imply
convergence in $U$, as the grazing directions construction in Remark~\ref{rem:grazing-control} shows. The objective $j(g)$ uses the feasible state $Sg$. For fixed $\alpha>0$, its convergence to $j(g^*)$ is
equivalent to $g\to g^*$ in $U$ by \eqref{eq:objective-gap} and
$\|Se_g\|_H\le\sqrt L\|e_g\|_U$.}

\begin{remark}[\new{Characteristic reference for angular-flux tracking}]
\new{For $\mathcal O=I$, the characteristic formula also gives an explicit
reference solution. At an inflow point $(y,\omega)$, set
}\[
 \new{a_y(t)=\exp\left(-\int_0^t\sigma(y+s\omega)\,ds\right),
 \qquad 0<t<\ell(y,\omega).
}\]
\new{Then $u(y+t\omega,\omega)=a_y(t)g(y,\omega)$, and the reduced
objective is the integral, with respect to $dS_y\,d\omega$, of
}\[
 \new{\frac{b(y,\omega)}{2}\int_0^{\ell(y,\omega)}
       |a_y(t)g-X_T(y+t\omega,\omega)|^2\,dt+\frac{\alpha}{2}|g|^2.
}\]
\new{Differentiation with respect to $g$ gives
}\begin{equation}\new{\label{eq:rev-characteristic-control}
 g^*(y,\omega)=
 \frac{b(y,\omega)\displaystyle\int_0^{\ell(y,\omega)}
                    a_y(t)X_T(y+t\omega,\omega)\,dt}{\alpha+b(y,\omega)\displaystyle\int_0^{\ell(y,\omega)}a_y(t)^2\,dt}.
}\end{equation}
\new{The denominator is bounded below by $\alpha$, and the adjoint is
recovered from its transport equation with zero outflow data.
For $X_T\ge0$, the control in \eqref{eq:rev-characteristic-control}
is nonnegative. Dose tracking couples the angular directions through
$D(u)$ and does not admit this separate minimisation along each ray.
}\end{remark}

\section{Neural Network Approximation and Residual Minimisation}
\label{sec:NN}
\new{Neural network residual minimisation \cite{raissi2019physics} has been used
for radiative transfer in \cite{mishra2021physics}. Here the network
represents the spatial and angular dependence jointly. High-order
space--angle--energy finite element approximations provide another
approach to phase-space discretisation \cite{houston2024efficient}.}
To approximate the optimality system \new{\eqref{eq:corrected-kkt}, we
represent the triple $(u,z,g)$} using neural networks as smooth surrogate
functions. These approximations are inserted directly into the residual functional \new{$\mathcal J_b$ defined in \eqref{eq:corrected-residual}, and trained by numerical minimisation of its quadrature approximation. This constitutes a neural network-based residual minimisation scheme for the
constrained transport control problem.}
Smooth feedforward networks provide a nonlinear approximation class on phase space that extends naturally to high-dimensional domains, and satisfy the required regularity conditions needed for the transport operator and boundary traces.

 \subsection{Neural network parameterisation of the KKT System}

Let \new{$V_\theta=(u_\theta,z_\theta,g_\theta)$ denote a feedforward
network with three outputs, parametrised by weights $\theta\in\mathbb R^P$.
The state and adjoint outputs are restricted to $\overline\Omega\times\mathbb S^{d-1}$, and the control output to
} $\Gamma^-$\new{. Three separate scalar networks may also be used. We use
$C^\infty$ activation functions, such as $\tanh$, and write
$\mathcal A_N$ for the resulting admissible class at resolution $N$.}

By construction, the network functions are smooth in both spatial and angular variables and hence belong to the regularity class of the residual functional:
\begin{displaymath}
\new{u_\theta,z_\theta\in C^\infty(\Omega\times\mathbb S^{d-1}),
\qquad (u_\theta,z_\theta,g_\theta)\in W\times W\times G_-.}
\end{displaymath}
\new{The control $g_\theta$ is the restriction of a smooth ambient
network to $\Gamma^-$. The membership in $W\times W\times G_-$
follows because the networks and their spatial derivatives are bounded
on the compact set $\overline\Omega\times\mathbb S^{d-1}$.
The directional derivatives $\omega\cdot\nabla_xu_\theta$ and
$\omega\cdot\nabla_xz_\theta$ are therefore classically defined
and computed by automatic differentiation. In two dimensions, using
$(\cos\vartheta,\sin\vartheta)$ as the angular input imposes angular
periodicity directly.}

\new{In fact, the boundedness of the control network and the finite
unweighted measure of $\Gamma^-$ give $g_\theta\in U$. Thus
$V_\theta\in W\times W\times U$, and the stronger control estimate in
Corollary~\ref{cor:unweighted-control} also applies when the optimality
residual is measured in $U$.}

\subsection{Loss Functional and Quadrature Discretisation}
\new{Denote the five residuals in \eqref{eq:corrected-residual} by
$r_1,\ldots,r_5$. Let $\xi_j^\Omega\in Q$ and
$\xi_j^\pm\in\Gamma^\pm$ be interior and boundary quadrature points,
with positive weights $w_j^\Omega$ and $w_j^\pm$. Each $\xi$ denotes
a position--direction pair. The discrete loss is
}\begin{equation}\label{eq:J-discrete}
\new{\begin{aligned}
\mathcal J_{b,h}(V)
={}&\frac12\sum_{j=1}^{M_\Omega}w_j^\Omega
          |r_1(V,\xi_j^\Omega)|^2
   +\frac12\sum_{j=1}^{M_\Omega}w_j^\Omega
          |r_2(V,\xi_j^\Omega)|^2\\
 &+\frac12\sum_{j=1}^{M_-}w_j^-
          |r_3(V,\xi_j^-)|^2
   +\frac12\sum_{j=1}^{M_-}w_j^-
          |r_4(V,\xi_j^-)|^2
   +\frac12\sum_{j=1}^{M_+}w_j^+
          |r_5(V,\xi_j^+)|^2.
\end{aligned}}
\end{equation}
\new{The weights approximate $dx\,d\omega$ in the volume and
$b\,dS_x\,d\omega$ on the boundary. For a product rule, the interior
weights have the form $w_xw_\omega$, as in spatial and angular
Gauss quadrature. We use $Q_{\Omega,h}$ and $Q_{\pm,h}$ below as
shorthand for these weighted sums.}
\new{For $\mathcal J_{U,h}$ and $\mathcal J_{0,h}$, the boundary
quadrature weights are changed according to the corresponding residual measures.}

\subsubsection{\new{Monte Carlo loss}}
\new{Alternatively, the integrals can be approximated by sample means.
Write
}\[
\new{m_\Omega=|\Omega|\,|\mathbb S^{d-1}|,\qquad
m_\pm=\int_{\Gamma^\pm}b\,dS_x\,d\omega.}
\]
\new{Draw $M_\Omega$ independent uniform samples in $Q$ and,
on each $\Gamma^\pm$, draw $M_\pm$ independent samples with
probability measure $m_\pm^{-1}b\,dS_x\,d\omega$.
Setting $w_j^\Omega=m_\Omega/M_\Omega$ and
$w_j^\pm=m_\pm/M_\pm$ in \eqref{eq:J-discrete} gives
}\begin{equation}\label{eq:J-Monte-Carlo}
\new{\begin{aligned}
\mathcal J_{b,M}(V)
={}&\frac{m_\Omega}{2M_\Omega}\sum_{j=1}^{M_\Omega}
       \bigl(|r_1(V,\xi_j^\Omega)|^2+|r_2(V,\xi_j^\Omega)|^2\bigr)\\
 &+\frac{m_-}{2M_-}\sum_{j=1}^{M_-}
       \bigl(|r_3(V,\xi_j^-)|^2+|r_4(V,\xi_j^-)|^2\bigr)
   +\frac{m_+}{2M_+}\sum_{j=1}^{M_+}|r_5(V,\xi_j^+)|^2,
\end{aligned}}
\end{equation}
\new{where $M=(M_\Omega,M_-,M_+)$}. \new{The measure factors in \eqref{eq:J-Monte-Carlo}
ensure that, for every fixed admissible $V$ independent of the samples,
\[
\mathbb{E}[\mathcal J_{b,M}(V)]=\mathcal J_b(V).
\]
The same analysis applies to fixed positive rescalings of the residual
terms, with the corresponding continuum functional and a modified
stability constant.}

\subsection{Training Algorithm}
The state, adjoint and control are trained jointly by minimising
$\mathcal J_h(V_\theta)$. We collect all trainable weights in
$\theta$. These may belong to one network with three outputs or to
three separate networks. For Monte Carlo quadrature, the samples are
drawn before training and held fixed during minimisation of $\mathcal
J_{b,M}$.  Deterministic rules are treated in the same way. We first
use Adam {\cite{revKingmaBa2015}} and, where specified, continue with
L-BFGS {\cite{revLiuNocedal1989}}. The procedure is summarised below.

\new{Here $\mathcal J_h$ denotes the quadrature approximation of the
chosen residual $\mathcal J_b$, $\mathcal J_U$ or $\mathcal J_0$.
Algorithm~\ref{alg:pinns} applies to all three choices by using the
corresponding boundary weights described above. The optimisation does
not project the parameters onto a prescribed bounded set. The relation
to the uniform sampling estimate is discussed after
Proposition~\ref{prop:uniform-quadrature}.}

\begin{algorithm}[htbp]
\caption{NN-based residual \new{minimisation}}\label{alg:pinns}
\begin{algorithmic}[1]
\Require \new{Initial parameters $\theta^0$, quadrature points and weights
(sampled once for Monte Carlo), Adam learning rate $\eta$, and budgets
$N_{\mathrm A}$ and $N_{\mathrm B}$ for Adam and L-BFGS}

\State Define $\ell(\theta)=\mathcal J_h(V_\theta)$ by
\eqref{eq:J-discrete}, using \eqref{eq:J-Monte-Carlo} for Monte Carlo samples
\For{\new{$k=0,\ldots,N_{\mathrm A}+N_{\mathrm B}-1$}}
\State \new{Evaluate $u_\theta,z_\theta,g_\theta$ at the interior and boundary quadrature points}
\State \new{Compute $\omega\cdot\nabla_xu_\theta$ and $\omega\cdot\nabla_xz_\theta$ by automatic differentiation}
\State \new{Form the state, adjoint, optimality, inflow and outflow residuals $r_1,\ldots,r_5$}

\State \new{Assemble $\ell(\theta^k)$ by the weighted sums in \eqref{eq:J-discrete}}
\State Compute \new{gradient $\nabla_\theta\ell(\theta^k)$ by automatic differentiation through the residuals}
\State Update \new{all network parameters together:}
\[
\new{\theta^{k+1}\gets
\begin{cases}
\operatorname{Adam}_{\eta}(\theta^k,\nabla_\theta\ell(\theta^k)),
 & k<N_{\mathrm A},\\
\operatorname{L\!\text{-}\!BFGS}(\theta^k,\ell,\nabla_\theta\ell),
 & k\ge N_{\mathrm A}.
\end{cases}}
\]
\State \new{Record the evaluated loss and residual contributions, identifying the parameter iterate used}

\EndFor
\State \Return \new{$(u_{\widehat\theta},z_{\widehat\theta},g_{\widehat\theta})$, where
$\widehat\theta=\theta^{N_{\mathrm A}+N_{\mathrm B}}$, and the residual history}
\end{algorithmic}
\end{algorithm}

\new{The Adam update retains its moment estimates between iterations.
At the change of optimiser, the L-BFGS history is initialised. Each
L-BFGS outer step may contain several inner iterations and line-search
evaluations. The network values, spatial derivatives, residuals and
parameter gradients are recomputed at each trial parameter vector.
Thus a recorded L-BFGS step is not equivalent in cost to an Adam step.
Setting $N_{\mathrm B}=0$ gives training with Adam alone.}

\new{For dose tracking, the adjoint residual uses
$D_hu_\theta(x)=\sum_iw_i u_\theta(x,\omega_i)$ in place of the angular
integral. Its parameter derivative therefore couples all directions
at each spatial point. We also denote this implemented loss by
$\mathcal J_{b,h}$. Its discrepancy from $\mathcal J_b$ includes the
inner angular approximation, as detailed in Section~\ref{sec:quadrature}. If the calculation is divided into spatial
chunks, each chunk retains the complete angular rule and its weighted
gradient contributions are accumulated before updating the parameters.}

\new{All conditions, including the inflow data and the boundary optimality
law, enter through the residual functional. The continuum stability
estimate applies to every returned network. To infer an error bound
from the discrete loss, the quadrature error must also be controlled.
For Monte Carlo quadrature, this is a sampling error, which motivates
the probabilistic estimates in the next section}.

\section{\new{Quadrature Error and A Posteriori Control}}
\label{sec:error-analysis}\label{sec:quadrature}
\new{The computed loss is a quadrature approximation to the continuum
residual. We now determine how this approximation enters the control of
the state, adjoint and control errors. Residual stability first gives
a deterministic estimate in terms of the discrete loss and its
quadrature discrepancy. For Monte Carlo quadrature, we bound this
discrepancy in mean square and with high probability, distinguishing
fixed fields from fields selected during training. We then consider
deterministic rules and the inner angular quadrature for dose tracking.}

\subsection{\new{A posteriori residual control}}
\label{sec:aposteriori}

\new{For $V=(u,z,g)$ and the unique optimality-system solution $V^*$,
write $E=H\times H\times G_-$ and
}\[
\new{\|V-V^*\|_E^2=\|u-u^*\|_H^2+\|z-z^*\|_H^2+\|g-g^*\|_{G_-}^2.
}\]
\new{The global stability estimate gives an a posteriori bound for every
admissible approximation, independently of its construction.}

\begin{corollary}[\new{A posteriori residual control}]
\label{thm:aposteriori}
\new{Every $V\in W\times W\times G_-$ satisfies
}\[
\new{\|V-V^*\|_E^2\le C_{\rm stab}\mathcal J_b(V),
}\]
\new{where $C_{\rm stab}$ depends only on $\alpha$, $L$ and
$\|\mathcal O\|$. In particular, $\mathcal J_b(V_N)\to0$ implies
$V_N\to V^*$ in $E$. If
$|\mathcal J_b(V)-\mathcal J_{b,h}(V)|\le q_h(V)$, then
}\[
\new{\|V-V^*\|_E^2
\le C_{\rm stab}\bigl(\mathcal J_{b,h}(V)+q_h(V)\bigr).
}\]
\end{corollary}

\begin{proof}
\new{Theorem~\ref{thm:corrected-coercivity} gives
}\[
\new{\|V-V^*\|_E^2
\le C_{\rm stab}\mathcal J_b(V)
\le C_{\rm stab}\bigl[\mathcal J_{b,h}(V)
       +|\mathcal J_b(V)-\mathcal J_{b,h}(V)|\bigr].}
\]
\new{The assertions follow by applying the first inequality to $V_N$
and the second to the stated quadrature bound.}
\end{proof}

\new{Thus the continuum residual controls the solution error without
requiring the approximation to lie near the solution. A small discrete
loss has the same interpretation once its quadrature error is bounded.
The relation between stability, residual loss and quadrature error is
{also central to the general analyses of neural residual minimisation, e.g.~
in \cite{mishra2023generalization,revShin2023}.} In particular, for the Monte Carlo loss,
}\begin{equation}\label{eq:mc-aposteriori-decomp}
\new{\|V-V^*\|_E^2\le C_{\rm stab}\left[
\mathcal J_{b,M}(V)+
|\mathcal J_b(V)-\mathcal J_{b,M}(V)|\right].}
\end{equation}
\new{This inequality holds for every realisation of the samples. It
separates the loss reduced during training from the sampling error
that must be controlled to infer accuracy of the solution.}

\new{For the control-cost norm, if $V\in W\times W\times U$ and
$|\mathcal J_U(V)-\mathcal J_{U,h}(V)|\le q_{U,h}(V)$,
Corollary~\ref{cor:unweighted-control} similarly gives
\begin{equation}\label{eq:U-discrete-certificate}
 \|u-u^*\|_H^2+\|z-z^*\|_H^2+\|g-g^*\|_U^2
 \le C_U\bigl[\mathcal J_{U,h}(V)+q_{U,h}(V)\bigr].
\end{equation}
Multiplication by $(\alpha+L\|\mathcal O\|^2)/2$ also bounds
$j(g)-j(g^*)$. By \eqref{eq:residual-inclusion}, the same argument
applies to $\mathcal J_0$ with its own quadrature discrepancy.
The constants in these reliability bounds can be conservative,
as discussed in Appendix~\ref{app:boundary-stability}. A loss evaluated
on training nodes alone is a diagnostic, not a continuum certificate.}

\subsection{\new{Monte Carlo quadrature}}
\label{subsec:mc-quadrature}

\new{We begin with the sample-mean approximation introduced in
\eqref{eq:J-Monte-Carlo}. For a fixed field its error is governed by
the moments of the residual densities. When the field is chosen by
minimising the sampled loss, uniform control over the neural network class
allows the same quadrature argument to be used.}

\new{The sample-mean loss \eqref{eq:J-Monte-Carlo} approximates five
residual integrals. To cover both boundary sampling choices, write
}\[
\new{\mathcal J_b(V)
 =\sum_{k=1}^{K}\int_{D_k}F_k(V,y)\,d\mu_k(y),\qquad K=5,
}\]
\new{where $F_k=|r_k|^2/2$, and $\mu_k$ is the volume or
flux-weighted boundary measure. Fixed positive loss weights may be
included in $F_k$. For a sampling probability measure $\nu_k$ with
$\mu_k\ll\nu_k$, set
}\[
\new{R_k(V,y)=F_k(V,y)\frac{d\mu_k}{d\nu_k}(y),\qquad
\mathcal J_{b,M}(V)
 =\sum_{k=1}^{K}\frac1{M_k}\sum_{j=1}^{M_k}R_k(V,Y_{k,j}).}
\]
\new{For each $k$, the samples $Y_{k,j}$ are independent with law
$\nu_k$. In \eqref{eq:J-Monte-Carlo}, $M_1=M_2=M_\Omega$,
$M_3=M_4=M_-$ and $M_5=M_+$. The two volume terms share their
samples, as do the two inflow terms. The arguments below allow this
sharing between residual terms.}

\rev{The proofs of the Monte Carlo estimates are collected in
Appendix~\ref{app:quadrature-proofs}.}

\begin{lemma}[\new{Mean-square quadrature error}]
\label{lem:mc-rms}
\new{Fix $V$ independently of the samples and suppose
$R_k(V,\cdot)\in L^2(\nu_k)$ for $k=1,\ldots,K$. Then
$\mathbb{E}\new{[}\mathcal J_{b,M}(V)\new{]}=\mathcal J_b(V)$ and
}\begin{equation}\label{eq:mc-rms-bound}
\new{\left(\mathbb{E}\new{\bigl[}\left|
\mathcal J_{b,M}(V)-\mathcal J_b(V)\right|^2\new{\bigr]}\right)^{1/2}
\le\sum_{k=1}^{K}
\left(\frac{\operatorname{Var}_{\nu_k}(R_k(V,\cdot))}{M_k}\right)^{1/2}.}
\end{equation}
\end{lemma}

\new{This is the classical variance argument for Monte Carlo
integration \cite{revCaflisch1998}, applied separately to the
volume and boundary residuals.}

\new{With a common sample count $M_k=M$, the root-mean-square
quadrature error is therefore $O(M^{-1/2})$ for fixed fields. This
requires finite second moments of the squared residual densities.
Well-posedness in $L^2$ alone does not supply them. Bounded densities
give the following high-probability estimate.}

\begin{proposition}[\new{Monte Carlo quadrature control}]
\label{prop:corrected-validation}
\new{Fix $V\in W\times W\times G_-$ independently of the samples
and let $0<\delta<1$. Suppose finite bounds $B_k$ satisfy
$0\le R_k(V,y)\le B_k$ for $\nu_k$-almost every $y$. Then, with
probability at least $1-\delta$,
}\begin{equation}\label{eq:mc-concentration}
\new{\left|\mathcal J_{b,M}(V)-\mathcal J_b(V)\right|
\le\sum_{k=1}^{K}B_k
       \sqrt{\frac{\log(2K/\delta)}{2M_k}}.}
\end{equation}
\new{On the same event,
}\begin{equation}\label{eq:corrected-validation}
\new{\|V-V^*\|_E^2
\le C_{\rm stab}\left[\mathcal J_{b,M}(V)
 +\sum_{k=1}^{K}B_k\sqrt{\frac{\log(2K/\delta)}{2M_k}}\right].}
\end{equation}
\end{proposition}

\new{Here $B_k$ bounds the range of the sampled density. A variance
bound suffices for \eqref{eq:mc-rms-bound}, but not for this exponential
concentration estimate. A field selected by
minimising $\mathcal J_{b,M}$ depends on the samples. We therefore
need a quadrature estimate that holds simultaneously throughout the
admissible neural network class.}

\subsubsection{\new{Uniform control on the training class}}

\begin{proposition}[\new{Uniform quadrature control}]
\label{prop:uniform-quadrature}
\new{For a fixed architecture, let the admissible parameter set
$\Theta_N\subset\mathbb R^{P_N}$ be nonempty and compact, and write
$R_k(\theta,y)=R_k(V_\theta,y)$. Suppose, on a common
set of full $\nu_k$-measure, that for every
$\theta,\vartheta\in\Theta_N$,
}\[
\new{0\le R_k(\theta,y)\le B_k,\qquad
|R_k(\theta,y)-R_k(\vartheta,y)|
 \le L_k\|\theta-\vartheta\|.
}\]
\new{Define the uniform discrepancy by
}\[
\new{\varepsilon_{\text{quad}}(N,M)
 =\sup_{\theta\in\Theta_N}
 |\mathcal J_b(V_\theta)-\mathcal J_{b,M}(V_\theta)|.
}\]
\new{For $\eta>0$, let $\mathcal N_N(\eta)$ be the cardinality of a
finite $\eta$-net of $\Theta_N$, with its centres in $\Theta_N$.
For $0<\delta<1$, with probability at least $1-\delta$,
}\begin{equation}\new{\label{eq:corrected-uniform-MC}
 \varepsilon_{\text{quad}}(N,M)\le
 2\eta\sum_{k=1}^{K}L_k
 +\sum_{k=1}^{K} B_k
  \sqrt{\frac{\log\!\bigl(2K\mathcal N_N(\eta)/\delta\bigr)}{2M_k}}.
}\end{equation}
\end{proposition}

\new{The event is uniform over $\Theta_N$, so it applies to an
approximation selected using those same samples, provided its
parameters belong to $\Theta_N$. The constants $B_k$, $L_k$ and the
covering number may grow with the architecture, so sample sizes must
increase accordingly. This is the distinction between quadrature
control for a fixed approximation and quadrature control throughout
training.}

\new{In particular, for a parameter $\widehat\theta\in\Theta_N$
returned by minimising this sampled loss,
}\[
\new{\|V_{\widehat\theta}-V^*\|_E^2
\le C_{\rm stab}\bigl[\mathcal J_{b,M}(V_{\widehat\theta})
                       +\varepsilon_{\text{quad}}(N,M)\bigr].}
\]
\new{Together with \eqref{eq:corrected-uniform-MC}, this gives the
probabilistic counterpart of the discrete residual estimate for the
trained approximation. It is this uniform quadrature control that
enters the convergence theorem below.}

Proposition~\ref{prop:uniform-quadrature} is a sufficient conditional
result for bounded parameter classes. The implemented Adam/L-BFGS
iterations are not projected onto a prescribed compact set
$\Theta_N$, so the hypotheses of the proposition are not asserted for
the optimisation trajectories used in Section~\ref{sec:numerics}.
Accordingly, Proposition~\ref{prop:uniform-quadrature} is used here as
a convergence framework, rather than as a probabilistic certificate
for the reported computations. For a fixed trained network, an
independent-sample certificate may instead be obtained under the
conditions described next.

\subsubsection{\new{Independent validation}}
\new{A second use of Proposition~\ref{prop:corrected-validation} is to
evaluate a trained approximation on fresh samples. Freeze
$\widehat V$ before drawing these samples. Conditional on training,
its residual densities are fixed, so \eqref{eq:corrected-validation}
holds with $V=\widehat V$ and with the loss evaluated on the validation
samples, provided the stated bounds $B_k$ hold conditionally.
The probability bound remains valid after averaging over training.
This gives an a posteriori estimate without requiring uniform control
over every network that training could have selected.}

\new{This is an available route to certification, rather than a
validation procedure performed for the reported trained networks.
It requires fresh independent samples and justified conditional bounds
on the residual densities. Reevaluation on a finer deterministic
quadrature rule alone does not supply this probabilistic certificate.}

\new{The independent-sample estimates above concern independent
position--direction pairs within each residual term. Reusing a direction
at many spatial points, or combining shared spatial and angular sample
sets, introduces dependence between the resulting pairs. Their product
count is therefore not the number of independent samples in these
estimates. The three-dimensional experiment uses such a sampling
structure. Its recorded training errors are not assigned a confidence
bound by substituting that product count into
\eqref{eq:mc-concentration}.}

\new{For $\mathcal J_U$, the same measure-based analysis uses
$\mu_3=dS_x\,d\omega$ for the optimality term. The corresponding
sampling weights and moment conditions are set out in
Appendix~\ref{app:U-sampling}.}

\subsection{\new{Deterministic quadrature}}\label{subsec:deterministic-quadrature}

\begin{lemma}[\new{Polynomial quadrature consistency}]
\label{lem:deterministic-quadrature}
\new{Let $Q_h$ be a positive quadrature rule with nodes in $D$ for a
finite measure $\mu$, exact on a polynomial space $\mathcal P_h$
containing constants. For a bounded integrand with specified pointwise
values, write $\|F\|_{\infty,D}=\sup_{y\in D}|F(y)|$. Then
}\begin{equation}\new{\label{eq:polynomial-quadrature-bound}
 \left|\int_D F\,d\mu-Q_hF\right|
 \le 2\mu(D)\inf_{p\in\mathcal P_h}\|F-p\|_{\infty,D}.
}\end{equation}
\end{lemma}
\begin{proof}
\new{Write $Q_hF=\sum_j w_jF(y_j)$. Positivity and exactness on
constants give $w_j\ge0$ and $\sum_jw_j=\mu(D)$. For every
$p\in\mathcal P_h$, polynomial exactness therefore yields
}\begin{align*}
\new{\left|\int_DF\,d\mu-Q_hF\right|}
&\new{=\left|\int_D(F-p)\,d\mu-\sum_jw_j(F-p)(y_j)\right|}\\
&\new{\le\int_D|F-p|\,d\mu+\sum_jw_j|F-p|(y_j)}\\
&\new{\le2\mu(D)\|F-p\|_{\infty,D}.}
\end{align*}
\new{Taking the infimum over $p$ proves the estimate.}
\end{proof}

\new{The estimate applies to each residual integral separately.
For a composite rule, applying it on each subinterval or integration
cell $D_\ell$, with local exactness space $\mathcal P_{h,\ell}$, gives
}\[
\new{\left|\int_DF\,d\mu-Q_hF\right|
\le2\sum_\ell\mu(D_\ell)
       \inf_{p\in\mathcal P_{h,\ell}}\|F-p\|_{\infty,D_\ell}.}
\]

\rev{For Gauss rules in reference coordinates, the flux factor,
geometric Jacobians and inflow or outflow indicator are included
in the integrand $F$. Polynomial exactness is therefore applied
to the reference measure.}

\new{A $q$-point Gauss--Legendre rule is exact through degree $2q-1$.
Its tensor product has this property in each coordinate. A spherical
$t$-design is exact for spherical polynomials of degree at most $t$.
Consequently, quadrature convergence follows from approximation of the
complete residual densities, including the coefficients, target data
and any geometric Jacobians. For a bound uniform over a neural network class,
the corresponding approximation estimates must also be uniform over
that class. Analytic extension to a common complex neighbourhood with
uniform bounds gives geometric polynomial-approximation estimates.
Finite differentiability gives the algebraic rate permitted by the
corresponding approximation estimate (see \cite[Chapters~7--8 and~19]{revTrefethen2019}) for these classical
polynomial-approximation and quadrature arguments. Smooth activation functions alone do not ensure exponential quadrature convergence. Discontinuous targets or coefficient interfaces require a separate piecewise analysis.}

\subsection{\new{Angular quadrature for dose tracking}}\label{subsec:angular-quadrature}
\new{For dose tracking, quadrature also enters inside the adjoint
residual through $D_hu$. Let $r_{2,h}$ denote the residual obtained by
replacing $Du$ by $D_hu$. Then
}\[
\new{r_2-r_{2,h}=D_hu-Du,\qquad
\|r_2-r_{2,h}\|_H
=|\mathbb S^{d-1}|^{1/2}\|Du-D_hu\|_{L^2(\Omega)},}
\]
\new{because the difference is constant in the angular variable.
Polarisation and the Cauchy--Schwarz inequality give
}\begin{align*}
\new{\frac12\left|\|r_2\|_H^2-\|r_{2,h}\|_H^2\right|}
&\new{=\frac12|(r_2-r_{2,h},r_2+r_{2,h})_H|}\\
&\new{\le\frac{|\mathbb S^{d-1}|^{1/2}}2
(\|r_2\|_H+\|r_{2,h}\|_H)\|Du-D_hu\|_{L^2(\Omega)}.}
\end{align*}
\new{This consistency term is added to the error in the outer
quadrature. Accordingly, the discrepancy
$|\mathcal J_b(V)-\mathcal J_{b,h}(V)|$ includes both approximations.}

\new{If the outer quadrature is Monte Carlo, the estimates above apply
to the modified residual with a fixed inner angular rule independent
of the outer samples. The displayed consistency term is then added to
control the original residual. For a uniform training estimate it must
also be bounded uniformly over the neural network class. In particular, an
unbiased estimate of $Du$ does not by itself give an unbiased estimate
of the squared adjoint residual.}

\section{\new{Neural Network Approximation and Convergence}}
\label{sec:approx-convergence}
\new{We now combine approximation in the transport space with the
quadrature estimates. The first result gives sufficient conditions
for the continuum residual to be approximated arbitrarily close
by neural networks. The convergence estimate then separates this
approximation error from quadrature and optimisation.}

\subsection{\new{Approximation in the graph and trace norm}}
\label{sec:approximation}

\new{We next consider approximation of the KKT solution by neural
networks. Since the residual contains transport derivatives and
boundary traces, approximation in $L^2$ alone is not sufficient. We use
the space
}\[
\new{\mathbb X=W\times W\times G_-,\qquad
\|V\|_{\mathbb X}^2=\|u\|_W^2+\|z\|_W^2+\|g\|_{G_-}^2,}
\]
\new{where $W$ carries the graph and trace norm introduced in
Subsection~\ref{subsec:trace}. Simultaneous approximation of a
function and its derivatives by neural networks was established in \cite{revHornik1990}. For the tanh
networks used here, we use the constructive approximation result from \cite{deryck2021tanh}.}

\begin{proposition}[\new{Neural network approximation}]
\label{prop:corrected-density}
\new{Suppose $V^*$ belongs to the $\mathbb X$-closure of triples obtained
by restricting smooth ambient functions to the phase domain and its
inflow boundary. Let the nested neural network classes $\mathcal A_N$ contain,
eventually, every three-output tanh network with two hidden layers
and finite width and parameters. Then
}\[
\new{\inf_{V\in\mathcal A_N}\|V-V^*\|_{\mathbb X}\longrightarrow0,
\qquad
\inf_{V\in\mathcal A_N}\mathcal J_b(V)
\le C\inf_{V\in\mathcal A_N}\|V-V^*\|_{\mathbb X}^2
\longrightarrow0.}
\]
\new{A sufficient condition for the closure hypothesis is
}\[
\new{u^*,z^*\in L^2(\mathbb S^{d-1},H^1(\Omega)),\qquad g^*\in G_-.}
\]
\new{Thus no angular derivatives are required for this sufficient
condition. It is an additional regularity assumption, rather than a
consequence of well-posedness for arbitrary $L^2$ data.}
\end{proposition}

\begin{proof}
\new{Fix $\varepsilon>0$. By the closure hypothesis, choose a
smooth ambient triple $V_\varepsilon$ such that
$\|V_\varepsilon-V^*\|_{\mathbb X}<\varepsilon/2$.
After multiplication by a smooth cutoff, its components may be
taken smooth on a compact box $B$ containing
$\overline\Omega\times\mathbb S^{d-1}$. For any smooth ambient
triple $V=(v_1,v_2,v_3)$, finiteness of the volume and boundary
measures and $|\omega|=1$ imply
}\[
\new{\|V\|_{\mathbb X}
\le C_B\left(\sum_{i=1}^3\|v_i\|_{C^1(B)}^2\right)^{1/2}.}
\]
\new{Indeed, the right-hand side controls the two volume norms,
the spatial directional derivatives and every weighted trace norm. By \cite[Theorem~5.1]{deryck2021tanh}, each component of $V_\varepsilon$ admits a two-hidden-layer tanh approximation in
$C^1(B)$. Placing the three scalar networks in parallel gives a
three-output network $V_{\theta_\varepsilon}$ such that
}\[
\new{\left(\sum_{i=1}^3
\|(V_{\theta_\varepsilon}-V_\varepsilon)_i\|_{C^1(B)}^2\right)^{1/2}
<\frac{\varepsilon}{2C_B}.}
\]
\new{This finite network belongs to $\mathcal A_N$ for all sufficiently
large $N$. Hence
}\[
\new{\inf_{V\in\mathcal A_N}\|V-V^*\|_{\mathbb X}
\le\|V_{\theta_\varepsilon}-V_\varepsilon\|_{\mathbb X}
   +\|V_\varepsilon-V^*\|_{\mathbb X}<\varepsilon.}
\]

For completeness, consider the stated sufficient condition. Since
$\Omega$ is Lipschitz, there is a bounded extension operator from
$H^1(\Omega)$ into $H^1(\mathbb R^d)$, and extension followed by
mollification shows that restrictions of functions smooth on a
neighbourhood of $\overline\Omega$ are dense in $H^1(\Omega)$.
Together with the density of smooth functions in
$L^2(\mathbb S^{d-1})$ and the density of finite-rank tensors in the
Bochner space, it follows that finite sums
\[
 \sum_{j=1}^m a_j(\omega)v_j(x),
 \qquad
 a_j\in C^\infty(\mathbb S^{d-1}),
\]
with each $v_j$ smooth on a neighbourhood of $\overline\Omega$, are
dense in
$L^2(\mathbb S^{d-1},H^1(\Omega))$. Smooth functions on the sphere
admit smooth extensions to a neighbourhood of $\mathbb S^{d-1}$, so
these tensor sums are restrictions of smooth ambient functions.

For $w\in L^2(\mathbb S^{d-1},H^1(\Omega))$, the spatial trace theorem
and $b\le1$ give
\[
 \|\omega\cdot\nabla_xw\|_H\le\|\nabla_xw\|_H,
 \qquad
 \|\gamma^-w\|_{G_-}^2+\|\gamma^+w\|_{G_+}^2
 \le C\|w\|_{L^2(\mathbb S^{d-1},H^1(\Omega))}^2.
\]
Hence, if $w_m$ is such a smooth tensor approximation,
\[
 \|w_m-w\|_W
 \le C\|w_m-w\|_{L^2(\mathbb S^{d-1},H^1(\Omega))}
 \longrightarrow0.
\]

For the control, let
$K=\partial\Omega\times\mathbb S^{d-1}$ and extend $g^*$ by zero from
$\Gamma^-$ to $K$. The measure
$b\,dS_x\,d\omega$ is a finite Borel measure on the compact set $K$,
so continuous functions on $K$ are dense in the corresponding
$L^2$ space. By the Stone--Weierstrass theorem, restrictions to $K$
of polynomials in the ambient coordinates are uniformly dense in
$C(K)$. Since such polynomials are smooth ambient functions, their
restrictions to $\Gamma^-$ approximate $g^*$ in $G_-$. This proves
the closure hypothesis.

\new{Finally, write $e_u=u-u^*$, $e_z=z-z^*$ and $e_g=g-g^*$.
Subtracting the zero residuals at $V^*$ gives
}\begin{align*}
\new{\|r_1\|_H}
&\new{\le\|\omega\cdot\nabla_x e_u\|_H+\|\sigma\|_\infty\|e_u\|_H,}\\
\new{\|r_2\|_H}
&\new{\le\|\omega\cdot\nabla_x e_z\|_H+\|\sigma\|_\infty\|e_z\|_H
             +\|\mathcal O\|^2\|e_u\|_H,}\\
\new{\|r_3\|_{G_-}}
&\new{\le\alpha\|e_g\|_{G_-}+\|\gamma^-e_z\|_{G_-},}\\
\new{\|r_4\|_{G_-}}
&\new{\le\|\gamma^-e_u\|_{G_-}+\|e_g\|_{G_-},\qquad
\|r_5\|_{G_+}=\|\gamma^+e_z\|_{G_+}.}
\end{align*}
\new{Squaring and summing, and using the definition of $\mathbb X$,
we obtain
}\[
\new{2\mathcal J_b(V)
=\sum_{k=1}^5\|r_k\|^2
\le C(\alpha,\|\sigma\|_\infty,\|\mathcal O\|)
       \|V-V^*\|_{\mathbb X}^2,}
\]
\new{where each residual uses its corresponding norm. Taking the
infimum over $\mathcal A_N$ completes the proof.}
\end{proof}

\new{The same argument applies to any smooth activation for which
approximation in $C^1$ is available. The result concerns increasing
neural network classes, allowing the width and parameter bounds to grow.
It does not assert consistency of a single fixed architecture.
Together with residual stability, it implies convergence of any
sequence of continuum near-minimisers whose optimisation gaps tend
to zero.}

\subsection{\new{Convergence}}
\label{sec:convergence}

\new{We conclude by combining approximation, quadrature and
optimisation. For each neural resolution $N$, choose a quadrature
rule $h=h(N)$. For Monte Carlo quadrature, this amounts to choosing
sample counts $M_k=M_k(N)$, and $\mathcal J_{b,h}$ denotes
$\mathcal J_{b,M}$. For classes
$\mathcal A_N=\{V_\theta:\theta\in\Theta_N\}$ satisfying the preceding
compactness, boundedness and Lipschitz hypotheses, the uniform discrepancy
is controlled by \eqref{eq:corrected-uniform-MC}.
The three contributions retain the notation
$\varepsilon_{\text{approx}}(N)$, $\varepsilon_{\text{quad}}(N)$ and
$\varepsilon_{\text{opt}}(N)$.}

\begin{theorem}[\new{Convergence under approximation, quadrature and optimisation}]
\label{thm:corrected-total-error}
\new{Let $\mathcal A_N\subset\mathbb X$ be a nonempty admissible neural network
class and $\widehat V_N\in\mathcal A_N$ a computed approximation
satisfying
}\[
\new{\mathcal J_{b,h}(\widehat V_N)
\le\inf_{V\in\mathcal A_N}\mathcal J_{b,h}(V)
   +\varepsilon_{\text{opt}}(N),\qquad
\varepsilon_{\text{opt}}(N)\ge0.}
\]
\new{Define
}\[
\new{\begin{aligned}
\varepsilon_{\text{approx}}(N)
   &:=\inf_{V\in\mathcal A_N}\mathcal J_b(V),\\
\varepsilon_{\text{quad}}(N)
   &:=\sup_{V\in\mathcal A_N}
       |\mathcal J_b(V)-\mathcal J_{b,h}(V)|.
\end{aligned}}
\]
\new{If $\varepsilon_{\text{quad}}(N)$ is finite, then
}\begin{equation}\label{eq:corrected-total-error}
\new{\|\widehat V_N-V^*\|_E^2
\le C_{\rm stab}\bigl(
\varepsilon_{\text{approx}}(N)
+2\varepsilon_{\text{quad}}(N)
+\varepsilon_{\text{opt}}(N)\bigr).}
\end{equation}
\new{Consequently, if all three contributions tend to zero, the state,
adjoint and control converge to the KKT solution in the error norm
$\|\cdot\|_E$.}
\end{theorem}

\begin{proof}
\new{For $\varepsilon>0$, choose $V_{N,\varepsilon}\in\mathcal A_N$
such that
$\mathcal J_b(V_{N,\varepsilon})
\le\varepsilon_{\text{approx}}(N)+\varepsilon$. Then
}\begin{align*}
\new{\mathcal J_b(\widehat V_N)}
&\new{\le\mathcal J_{b,h}(\widehat V_N)
        +\varepsilon_{\text{quad}}(N)}\\
&\new{\le\mathcal J_{b,h}(V_{N,\varepsilon})
        +\varepsilon_{\text{opt}}(N)+\varepsilon_{\text{quad}}(N)}\\
&\new{\le\mathcal J_b(V_{N,\varepsilon})
        +\varepsilon_{\text{opt}}(N)+2\varepsilon_{\text{quad}}(N)}\\
&\new{\le\varepsilon_{\text{approx}}(N)+\varepsilon
        +\varepsilon_{\text{opt}}(N)+2\varepsilon_{\text{quad}}(N).}
\end{align*}
\new{Let $\varepsilon\downarrow0$ and apply
Theorem~\ref{thm:corrected-coercivity}. No minimiser of the continuum
or discrete loss over $\mathcal A_N$ needs to be attained.}
\end{proof}

\new{The corresponding conclusion in $U$ requires
$\|r_3(\widehat V_N)\|_U\to0$ in addition to the three vanishing
error contributions above, with $\widehat g_N\in U$. Indeed,
\[
 \mathcal J_U(\widehat V_N)\le\mathcal J_b(\widehat V_N)
                +\tfrac12\|r_3(\widehat V_N)\|_U^2\to0,
\]
so Corollary~\ref{cor:unweighted-control} gives
$\widehat g_N\to g^*$ in $U$ and
$j(\widehat g_N)\to j(g^*)$. The added hypothesis concerns the
continuum residual,~i.e. its quadrature discrepancy must also be controlled.
Smoothness of the network alone does not imply this convergence.}

\rev{Theorem~\ref{thm:corrected-total-error} is stated for $\mathcal J_b$.
The same proof applies to $\mathcal J_U$ or $\mathcal J_0$, with
$C_U$, convergence in $H\times H\times U$, and approximation and
quadrature errors defined using the chosen loss. In particular,
$\inf_{\mathcal A_N}\mathcal J_b\to0$ alone does not imply
$\inf_{\mathcal A_N}\mathcal J_0\to0$.
For $\mathcal J_U$, sufficient approximation is convergence of the
state and adjoint in $W$ and of the control in $U$, since
$\|b\gamma^-v\|_U\le\|\gamma^-v\|_{G_-}$.
For $\mathcal J_0$, convergence of $\gamma^-u$ and $\gamma^+z$
in unweighted boundary $L^2$ is also sufficient. The spatial
$H^1$ condition in Proposition~\ref{prop:corrected-density}, together
with $g^*\in U$, supplies these trace approximations. This additional
regularity has not been established for the indicator-target examples.}

\begin{remark}[\new{Approximation, quadrature and optimisation}]
\new{A related separation of errors is used in the analysis of deep
Uzawa methods \cite{makridakis2025kineticuzawa}. Here residual
stability connects these three errors to the state, adjoint and
control. Proposition~\ref{prop:corrected-density} gives sufficient
conditions for the approximation error to vanish, while
Section~\ref{sec:quadrature} supplies deterministic and probabilistic
quadrature bounds. The factor two in
\eqref{eq:corrected-total-error} reflects the use of quadrature at
both the computed network and the continuum comparison function.
The optimisation gap remains an assumption on the returned
approximation. Convexity of the residual in the fields does not
imply convexity in the network parameters or convergence of a
particular training algorithm. Its nonnegative empirical loss is
an upper bound for this gap.}
\end{remark}

\begin{corollary}[\new{Convergence for random quadrature}]
\label{cor:random-convergence}
\new{Under the hypotheses of Theorem~\ref{thm:corrected-total-error},
suppose deterministic sequences $a_N,o_N,q_N\ge0$ tend to zero,
with $\varepsilon_{\rm approx}(N)\le a_N$,
$\varepsilon_{\rm opt}(N)\le o_N$ almost surely, and
}\[
\new{\mathbb P\bigl(\varepsilon_{\rm quad}(N)\le q_N\bigr)
\ge1-\delta_N,\qquad 0<\delta_N<1.}
\]
\new{If $\delta_N\to0$, then $\widehat V_N\to V^*$ in probability
in $E$. If $\sum_N\delta_N<\infty$, the convergence is almost sure.}
\end{corollary}
\begin{proof}
\new{On $A_N=\{\varepsilon_{\rm quad}(N)\le q_N\}$, the total-error
estimate gives, almost surely,
}\[
\new{\|\widehat V_N-V^*\|_E^2
\le C_{\rm stab}(a_N+2q_N+o_N).}
\]
\new{For any $\varepsilon>0$, the right-hand side is at most
$\varepsilon^2$ for all sufficiently large $N$. Thus
}\[
\new{\mathbb P\bigl(\|\widehat V_N-V^*\|_E>\varepsilon\bigr)
\le\mathbb P(A_N^c)\le\delta_N\longrightarrow0.}
\]
\new{If $\sum_N\delta_N<\infty$, Borel--Cantelli gives
$\mathbb P(A_N^c\text{ infinitely often})=0$. The deterministic
bound then implies almost-sure convergence. Independence between
refinement levels is not needed.}
\end{proof}

\new{The estimate concerns the squared solution error. An additive
quadrature contribution of order $M^{-1/2}$ therefore gives an
$M^{-1/4}$ contribution to the solution norm, with the associated
confidence and class-complexity factors. Obtaining an accuracy--cost
bound would additionally require quantitative approximation rates
and a bound on the cost of reducing the optimisation gap.}

\section{Numerical Results}\label{sec:numerics}

\rev{The numerical experiments examine the approximation of the optimality
system and the behaviour of the residual during training. Circular field
and dose targets in the square demonstrate the numerical
use of the two control objectives. Manufactured solutions in the disc and the three-dimensional ball allow
comparison with the exact state, adjoint and control, supporting the
residual-stability analysis. Families of ellipsoidal and bilobed dose targets in the cube demonstrate joint approximation over varying shapes, using one network for each target family. The exact characteristic reference also permits a residual--error comparison for circular field tracking. Supplementary computations are collected in Appendix~\ref{app:supplementary-numerics}.

We distinguish the dose approximation error $\|D(u_\theta)-D(u^*)\|$
from the optimal target mismatch $\|D(u^*)-q\|$, which may be nonzero.
When a reference solution is not known, we report the target mismatch $\|D(u_\theta)-q\|$ to assess agreement with the prescribed dose.}
\subsection*{\rev{Computational setting description}}
\rev{Table~\ref{tab:experiment-settings} collects the settings of the performed experiments. Each trained approximation is a fully connected network with three outputs $(u,z,g)$ and inputs $(x,\omega)$,
using four scalar coordinates in two dimensions and six in three dimensions.
The parametric examples add the scalar shape parameter $\mu$ as a seventh input.
The following initialisation and training settings apply to all neural
computations except the ball computation in Example~\ref{ex:ball}.
These computations use Xavier-uniform initialisation with zero biases. Training consists of $500$ Adam updates at learning rate $10^{-4}$ followed by $750$ outer L-BFGS updates at learning rate $1$, with a strong Wolfe line search, history size $50$ and at most $20$ inner iterations per update.
The parametric ellipsoidal and bilobed experiments use $2\,000$ outer L-BFGS updates.}

\rev{The computations minimise the discrete residual functional
$\mathcal J_{b,h}$. Errors and final residuals are evaluated using
the finer quadrature rules listed in Table~\ref{tab:experiment-settings}.
For the three-dimensional field benchmark in Example~\ref{ex:ball},
the reported diagnostics use the training samples.}

\begin{table}[!htbp]
\centering
\caption{\rev{Computational settings for the main and supplementary experiments.
Network entries give hidden-layer count $\times$ width and activation.
Node triples are $(N_x,M,N_b)$, with $N_x$ total spatial nodes, $M$ directions
and $N_b$ total boundary spatial nodes before the inflow/outflow split.
Parametric counts include all five shapes.
Product rules give $N_xM$ volume pairs and $N_bM$ boundary pairs before
splitting.}}
\label{tab:experiment-settings}
{\color{black}\footnotesize
\setlength{\tabcolsep}{3pt}
\renewcommand{\arraystretch}{1.16}
\begin{tabular}{@{}p{.27\linewidth}lccc@{}}
\toprule
Experiment & Network & Training nodes & Evaluation nodes & Arithmetic\\
\midrule
Circular and triangular fields in the square
(\ref{ex:circle}, \ref{ex:triangle})
& $12\times64$ Tanh & $(10\,800,75,360)$ & $(43\,200,150,720)$ & 64 bit, CUDA\\
Regular disc (\ref{ex:regular-disc-revised})
& $8\times64$ Tanh & $(8\,100,75,900)$ & $(21\,600,150,1\,800)$ & 64 bit, CUDA\\
Manufactured dose in the disc (\ref{ex:manufactured-dose})
& $12\times64$ Tanh & $(8\,100,75,270)$ & $(21\,600,150,540)$ & 32 bit, MPS\\
Circular dose in the square (\ref{ex:circle-dose}), uniform
& $12\times64$ Tanh & $(10\,800,75,360)$ & $(43\,200,150,720)$ & 64 bit, CUDA\\
Same target
& $12\times64$ SiLU & $(10\,800,75,360)$ & $(43\,200,150,720)$ & 64 bit, CUDA\\
Field and dose in the disc (\ref{ex:localised})
& $12\times64$ Tanh & $(10\,800,75,270)$ & $(43\,200,150,540)$ & 64 bit, CUDA\\
Manufactured dose in the ball (\ref{ex:ball-dose})
& $8\times64$ Tanh & $(10\,240,288,800)$ & $(22\,400,800,1\,568)$ & 64 bit, CUDA\\
Parametric ellipsoidal dose targets (\ref{ex:cube-sphere-dose})
& $8\times64$ Tanh & $(10\,000,288,1\,920)$ & $(20\,000,800,3\,840)$ & 64 bit, CUDA\\
Parametric bilobed dose targets (\ref{ex:bilobed-dose-ongoing})
& $8\times64$ Tanh & $(10\,000,288,1\,920)$ & $(20\,000,800,3\,840)$ & 64 bit, CUDA\\
\midrule
Manufactured ball (\ref{ex:ball})
& $8\times64$ Tanh & \shortstack{$(10\,000,30,900)$\\$N_x$ per direction} & Training rule & 32 bit\\
Fixed-field Monte Carlo (\ref{sec:mc-residual-experiment})
& Analytic fields & \multicolumn{2}{c}{$2^8,2^{10},\ldots,2^{18}$ samples per domain} & No training\\
\bottomrule
\end{tabular}
}
\end{table}

\rev{Table~\ref{tab:quadrature-construction} gives the quadrature
construction details. Two-dimensional deterministic rules use composite three-point
Gauss--Legendre quadrature. The parametric cube examples use fixed uniform
Monte Carlo samples in space and on each boundary face, with independent
samples for evaluation. The ball rule uses Gauss--Legendre nodes in radius and polar cosine,
with periodic azimuthal nodes. The cube direction rule uses the same
polar-cosine and azimuthal construction. Spatial weights include cell
areas or volumes and the polar or spherical Jacobian where appropriate.
Angular integrals use the corresponding direction weights.}

\begin{table}[!htbp]
\centering
\caption{\rev{Construction of the quadrature and sampling rules.
Entries count cells, quadrature nodes or samples as indicated in the first column.
}}
\label{tab:quadrature-construction}
{\color{black}\small
\setlength{\tabcolsep}{8pt}
\renewcommand{\arraystretch}{1.18}
\begin{tabular}{@{}p{.54\linewidth}cc@{}}
\toprule
Rule component & Training & Evaluation\\
\midrule
Square, uniform spatial cells & $40\times30$ & $80\times60$\\
Disc, manufactured spatial nodes $(r,\varphi)$ & $90\times90$ & $180\times120$\\
Disc, target-problem spatial nodes $(r,\varphi)$ & $120\times90$ & $240\times180$\\
\midrule
Ball, spatial nodes $(r,\cos\vartheta,\varphi)$ & $20\times16\times32$ & $28\times20\times40$\\
Ball, direction nodes $(\cos\vartheta_\omega,\varphi_\omega)$ & $12\times24$ & $20\times40$\\
Ball, boundary nodes $(\cos\vartheta,\varphi)$ & $20\times40$ & $28\times56$\\
\midrule
Parametric cube, spatial samples per shape & $2\,000$ & $4\,000$\\
Parametric cube, samples per face and shape & $64$ & $128$\\
Cube, direction nodes $(\cos\vartheta_\omega,\varphi_\omega)$ & $12\times24$ & $20\times40$\\
\bottomrule
\end{tabular}
}
\end{table}

\subsection{\rev{Example: Field and dose tracking of a circular target in the square}}
\label{ex:circle-square}

\subsubsection{\rev{
\rev{Circular field target under directional irradiation}}}
\label{ex:circle}

Let $\Omega=(0,1)^2$, $T=B_{0.1}(0.5,0.5)$,
$\sigma=0.1$ and $\alpha=10^{-4}$. We minimise the field-tracking
objective
\begin{equation}\label{eq:revised-field-objective}
F_{\mathrm{field}}(u,g)
=\frac12\|u-\chi_T\|_{L^2(\Omega\times S^1)}^2
+\frac{\alpha}{2}\|g\|_{L^2(\Gamma^-,\,dS_x\,d\omega)}^2
\end{equation}
subject to $\omega\cdot\nabla u+\sigma u=0$ and $\gamma^-u=g$.
The associated adjoint and optimality equations are
\[
-\omega\cdot\nabla z+\sigma z=u-\chi_T,
\qquad \gamma^+z=0,
\qquad \alpha g+b\gamma^-z=0,
\]
where $b=|\omega\cdot n|$. The field target is unscaled.
We also evaluate the angular integral
$D(u)(x)=\int_{S^1}u(x,\omega)\,d\omega$.

Formula~\eqref{eq:rev-characteristic-control} gives the exact optimal
control and state from the analytic intersections of each ray with
$T$. For the distance $\tau_+(x,\omega)$ to the outflow boundary,
the exact adjoint is
\[
z^*(x,\omega)=\int_0^{\tau_+(x,\omega)}e^{-\sigma s}
\bigl(u^*(x+s\omega,\omega)-\chi_T(x+s\omega)\bigr)\,ds.
\]
Both terms are integrated analytically, splitting at the target
crossings.

Figure~\ref{fig:circle-field-optimum} compares the approximate and
exact directional fluxes at $64.8^\circ$ and $151.2^\circ$,
together with their angular integrals and spatial cross-sections.
The elongated profiles follow straight characteristics through the
circular target. Their location is determined by the projection of
$T$ onto the inflow boundary along each direction, while their
transverse variation reflects the different lengths of intersection
with the target. Formula~\eqref{eq:rev-characteristic-control}
makes this dependence explicit: the optimal inflow is zero on rays
that miss $T$, and its value on intersecting rays balances their
attenuation-weighted overlap with $T$ against the field energy and
control penalty. The approximate directional fields reproduce the
resulting beam structure.

Integrating these fields over the $150$ directions of the evaluation
rule produces a maximum near the centre of the target and a gradual
decrease away from it. A directional flux cannot be confined to $T$,
since it propagates from the inflow boundary and continues beyond
the target along the same characteristic. The surrounding contribution
in the angular integral is therefore part of the transport response.
The training loss decreases to $\mathcal J_{b,h}=4.931\times10^{-5}$,
with a slower decrease towards the end of training. The characteristic
comparison quantifies how accurately this computation recovers the
optimal field and its angular integral.

\begin{figure}[htbp]
\centering
\includegraphics[width=.98\linewidth]{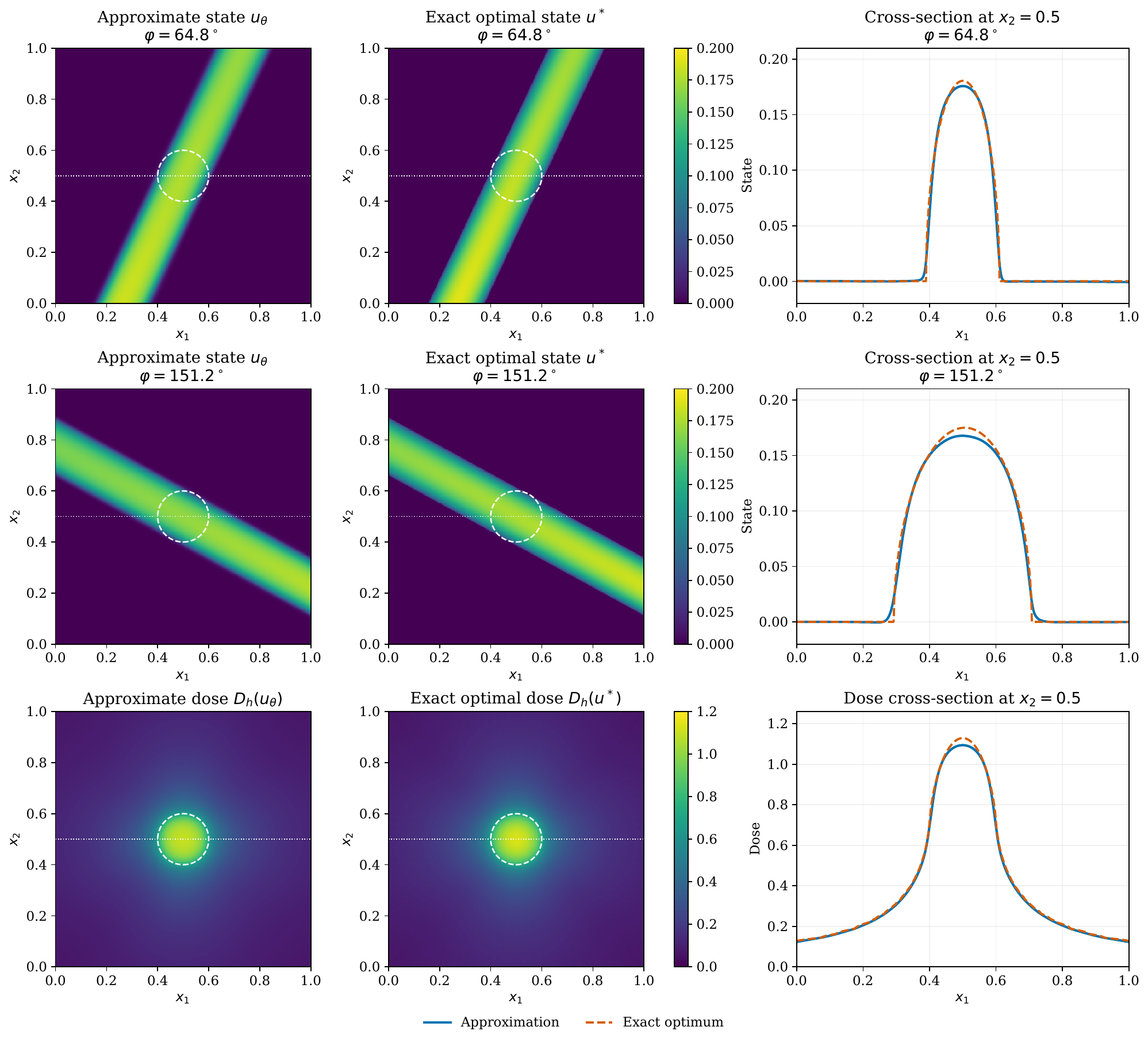}
\caption{\rev{Circular field target. Approximate and exact optimal
states at $64.8^\circ$ and $151.2^\circ$ occupy the first two rows.
Their angular integrals occupy the third row. The right column
gives cross-sections at $x_2=0.5$. Both doses use the same
$150$-direction Gauss rule. Dashed circles mark $\partial T$ and
each pair of maps has a common colour scale.}}
\label{fig:circle-field-optimum}
\end{figure}

\rev{Table~\ref{tab:circle-field-kkt-errors} reports errors against the
exact optimum on the separate evaluation rule. The state and adjoint
use the volume norm $H$, and the control is measured in both $G_-$
and $U$. The dose error compares angular integrals evaluated with the
same rule. At the target centre, the approximate and reference doses
are $1.0941$ and $1.1290$.}

\begin{table}[htbp]
\centering
\caption{\rev{Circular field target. Absolute and relative errors
against the characteristic optimum on the separate evaluation rule. Here $v_\theta$ and $v_{\rm ref}$ are the approximate and reference quantities in each row, and $X$ is the indicated norm.
Checkpoint diagnostics are reported in Table~\ref{tab:circle-field-stages}.}}
\label{tab:circle-field-kkt-errors}
{\color{black}\small
\begin{tabular}{lrr}
\toprule
Quantity and norm & \shortstack{$\|v_\theta-v_{\rm ref}\|_X$}
& \shortstack{$\dfrac{\|v_\theta-v_{\rm ref}\|_X}{\|v_{\rm ref}\|_X}$}\\
\midrule
State, $L^2(\Omega\times S^1)$ & $8.3760\times10^{-3}$ & $4.8223\times10^{-2}$\\
Adjoint, $L^2(\Omega\times S^1)$ & $3.0118\times10^{-3}$ & $6.3667\times10^{-2}$\\
Control, $G_-$ & $8.4473\times10^{-3}$ & $4.8298\times10^{-2}$\\
Control, $U$ & $8.9773\times10^{-3}$ & $4.8829\times10^{-2}$\\
Dose, $L^2(\Omega)$ & $4.6927\times10^{-3}$ & $1.8093\times10^{-2}$\\
\bottomrule
\end{tabular}}
\end{table}

Figure~\ref{fig:circle-field-adjoint} completes the comparison with
the exact adjoint. The cross-sections resolve the changes of slope
where the rays enter and leave the target. Together with the control
errors, this checks recovery of the coupled optimality system.

\begin{figure}[!htbp]
\centering
\includegraphics[width=.98\linewidth]{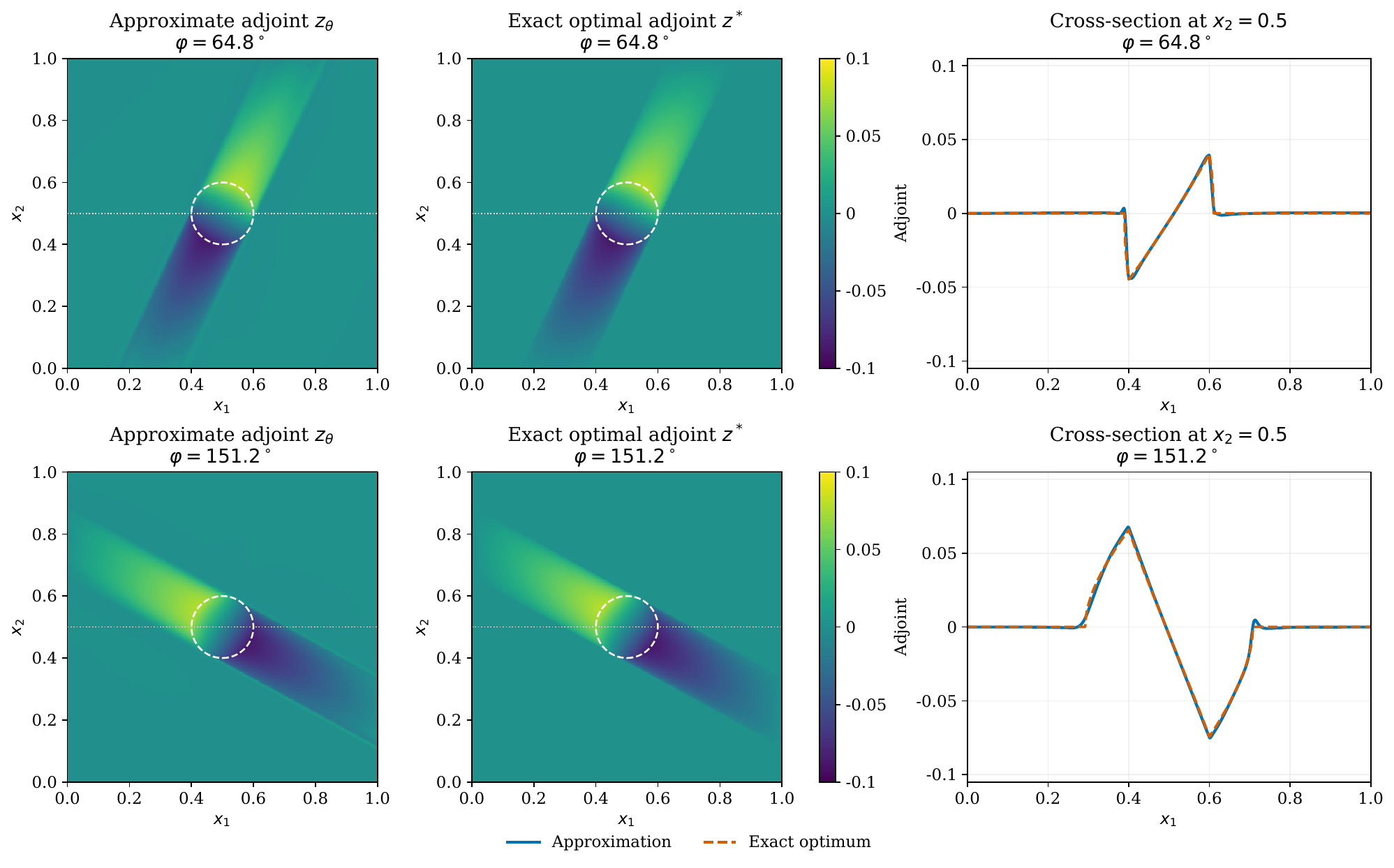}
\caption{\rev{Circular field target. Approximate and exact adjoints
at $64.8^\circ$ and $151.2^\circ$, with cross-sections at $x_2=0.5$.
Each approximate/exact pair shares the same colour scale. The
dashed circle marks $\partial T$.}}
\label{fig:circle-field-adjoint}
\end{figure}

\rev{Errors and residuals are evaluated together on the same separate rule, at different checkpoints, and reported in Table~\ref{tab:circle-field-stages}. The combined error decreases throughout the recorded trajectory, while $E_b^2/\mathcal J_{b,h}$ ranges from
$0.15636$ to $1.34704$ and $E_U^2/\mathcal J_{U,h}$ from $0.16581$ to $1.56912$.}

\begin{table}[htbp]
\centering
\caption{\rev{Circular field target. Combined squared errors and
integral residuals at selected checkpoints, evaluated on the common
separate rule. Here
$E_b^2=\|u_\theta-u^*\|_H^2+\|z_\theta-z^*\|_H^2+
\|g_\theta-g^*\|_{G_-}^2$ and
$E_U^2=\|u_\theta-u^*\|_H^2+\|z_\theta-z^*\|_H^2+
\|g_\theta-g^*\|_U^2$.}}
\label{tab:circle-field-stages}
{\color{black}\small
\begin{tabular}{rrrrr}
\toprule
Completed updates & $\mathcal J_{b,h}$ & $E_b^2$
& $E_b^2/\mathcal J_{b,h}$ & $E_U^2/\mathcal J_{U,h}$\\
\midrule
0 & $6.1860\times10^{-1}$ & $7.5738\times10^{-1}$ & $1.22434$ & $1.56912$\\
500 & $9.4751\times10^{-2}$ & $5.2949\times10^{-2}$ & $0.55882$ & $0.60424$\\
600 & $2.5388\times10^{-3}$ & $3.4198\times10^{-3}$ & $1.34704$ & $1.40727$\\
800 & $1.1423\times10^{-3}$ & $3.9042\times10^{-4}$ & $0.34179$ & $0.36067$\\
1000 & $9.9502\times10^{-4}$ & $2.3105\times10^{-4}$ & $0.23221$ & $0.24626$\\
1250 & $9.6308\times10^{-4}$ & $1.5059\times10^{-4}$ & $0.15636$ & $0.16581$\\
\bottomrule
\end{tabular}}
\end{table}

{\color{black}
\subsubsection{\rev{Direct dose tracking of a circular target}}\label{ex:circle-dose}
\rev{On the same geometry and with the same $\sigma$ and $\alpha$,
we now prescribe the integrated dose $q=\chi_T$.} Thus the observation is
$D(u)(x)=\int_{S^1}u(x,\omega)\,d\omega$, and the objective is
\begin{equation}\label{eq:circle-dose-objective}
 F_{\rm dose}(u,g)=\frac12\|D(u)-q\|_{L^2(\Omega)}^2
                  +\frac\alpha2\|g\|_U^2.
\end{equation}
The adjoint source is $D(u)-q$, and stationarity remains $\alpha g+b\gamma^-z=0$.

\rev{We compare Tanh and SiLU networks on the same uniform spatial
quadrature, with the settings in Table~\ref{tab:experiment-settings}.}

\begin{figure}[!htbp]
\centering
\includegraphics[width=\linewidth]{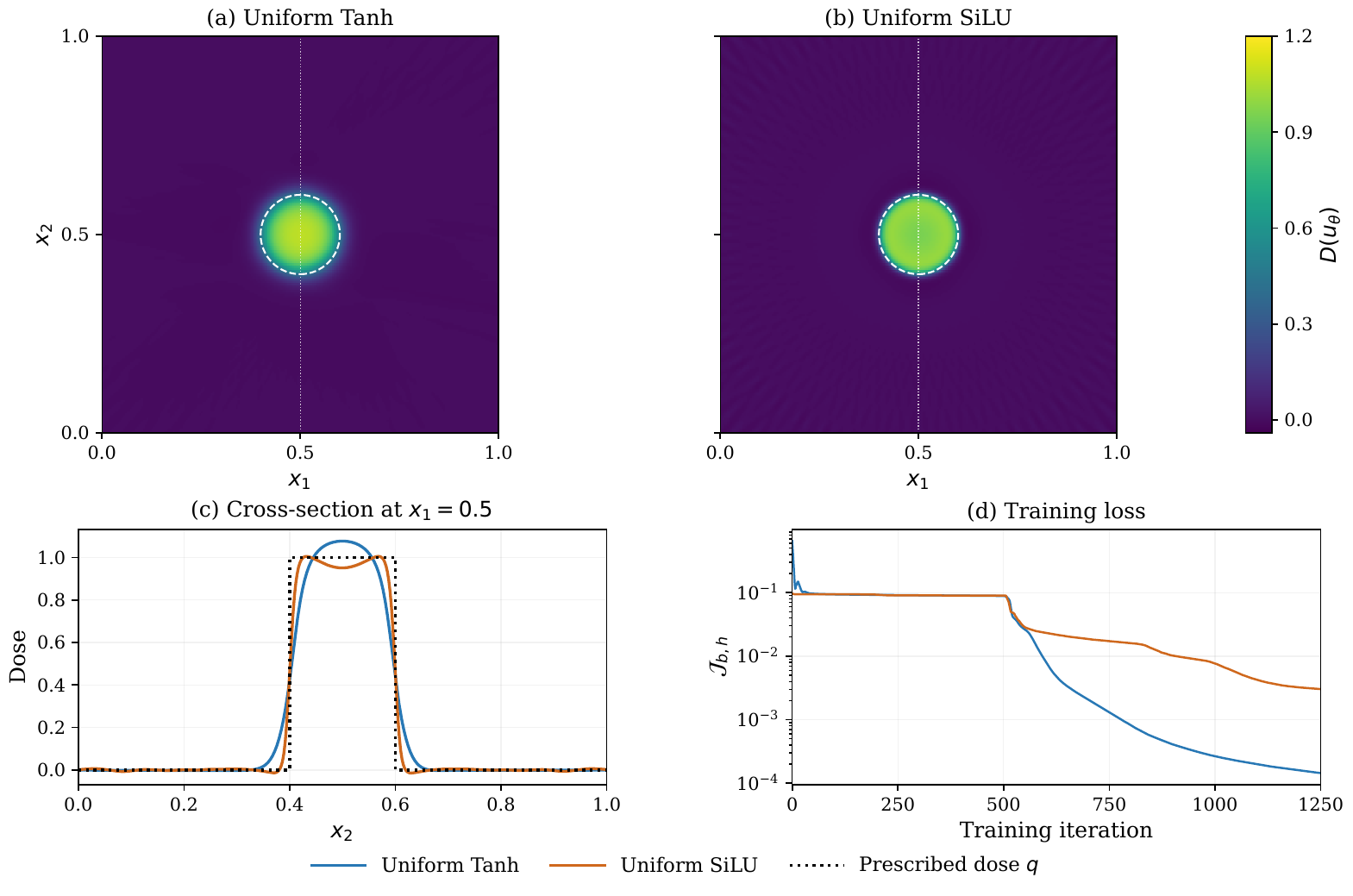}
\caption{\rev{Direct dose tracking of the circular target in the square.
(a)--(b) Doses for uniform Tanh and SiLU on a common colour scale.
Dashed circles mark the target, and dotted vertical lines mark $x_1=0.5$.
(c) Cross-sections and prescribed dose. (d) Training losses on the common
uniform quadrature. Dose maps and cross-sections use the same
$150$-direction evaluation rule.}}
\label{fig:circle-dose-uniform}
\end{figure}

\rev{Figure~\ref{fig:circle-dose-uniform} shows that SiLU gives a sharper
transition and a closer target fit than Tanh. Table~\ref{tab:circle-dose-uniform}
reports the target mismatch and residuals on the common evaluation rule
in Table~\ref{tab:experiment-settings}. Both dose profiles retain a
smooth transition across the target boundary.}

\begin{table}[htbp]
\centering
\caption{\rev{Circular dose comparison at the final checkpoints.
The relative target mismatch is
$e_q=\|D_h(u_\theta)-q\|_{L^2(\Omega)}/\|q\|_{L^2(\Omega)}$,
evaluated on the common separate rule.
The last two columns give $\mathcal J_b$ evaluated by the training
and common separate quadratures.}}
\label{tab:circle-dose-uniform}
{\color{black}\small
\begin{tabular}{lrrr}
\toprule
Network / rule & $e_q$ & Training residual & Evaluation residual\\
\midrule
Tanh / uniform & $0.31736$ & $1.447\,10^{-4}$ & $1.036\,10^{-3}$\\
SiLU / uniform & $0.20553$ & $3.034\,10^{-3}$ & $2.627\,10^{-3}$\\
\bottomrule
\end{tabular}}
\end{table}

}

\subsection{\rev{Example: Regular manufactured KKT solution in the unit disc and ball}}
\label{ex:regular-disc-revised}\label{ex:ball}
{\color{black}

\rev{We now consider a manufactured solution designed to satisfy the full KKT system while remaining regular near grazing directions. In the disc, we compare residuals with solution errors in $H$, $G_-$
and $U$,} while $\Omega=B_1(0)\subset\mathbb R^2$, $\mathcal O=I$ and $\sigma=\alpha=1$. For
$(x,\omega)\in\Omega\times S^1$, set
\[
d=x\cdot\omega,\qquad
\Delta=\sqrt{1-|x|^2+d^2},\qquad
\tau_+=\Delta-d,
\]
and define
\[
A(s)=e^{-1/s^2},\qquad
r(s)=-\frac{\alpha A(s)e^s}{s(1-e^{-2s})}
\quad(s>0),\qquad A(0)=r(0)=0.
\]
The exact solution and prescribed tracking field are
\begin{equation}\label{eq:regular-disc-revised-exact}
\begin{aligned}
u_{\rm ex}&=A(\Delta)e^{-d},&
z_{\rm ex}&=r(\Delta)(1-e^{-\tau_+}),\\
g_{\rm ex}&=A(\Delta)e^\Delta\quad\text{on }\Gamma^-,&
X_T&=u_{\rm ex}-r(\Delta).
\end{aligned}
\end{equation}
Indeed, $\omega\cdot\nabla d=1$,
$\omega\cdot\nabla\Delta=0$ and
$\omega\cdot\nabla\tau_+=-1$, so
\[
\omega\cdot\nabla u_{\rm ex}+u_{\rm ex}=0,\qquad
-\omega\cdot\nabla z_{\rm ex}+z_{\rm ex}=u_{\rm ex}-X_T.
\]
The traces satisfy $\gamma^-u_{\rm ex}=g_{\rm ex}$ and
$\gamma^+z_{\rm ex}=0$. On $\Gamma^-$, where $b=\Delta$ and
$\tau_+=2\Delta$, the optimality condition follows from
\[
\alpha g_{\rm ex}+b\gamma^-z_{\rm ex}
=\alpha A(\Delta)e^\Delta
+\Delta r(\Delta)(1-e^{-2\Delta})=0.
\]
The flatness of $A$ at zero removes the apparent grazing singularities.
Thus, \eqref{eq:regular-disc-revised-exact} solves the full field-tracking
KKT system with the unweighted control penalty in $U$.
\rev{The same characteristic construction applies in the unit ball in
three dimensions.}

The network used has $29\,635$ trainable parameters. Its computational
settings are listed in Table~\ref{tab:experiment-settings}. Final errors are reported in Table~\ref{tab:regular-disc-revised-errors}, and verify recovery of the coupled optimality system for this regular solution at $\alpha=1$. The dose $D(u)=\int_{S^1}u\,d\omega$ is a diagnostic of this field-tracking solution. Its reference is obtained by applying the same $150$-direction rule to $u_{\rm ex}$.

\begin{table}[!htbp]
\centering\small\color{black}
\begin{tabular}{lcc}
\toprule
Quantity and norm & \shortstack{$\|v_\theta-v_{\rm ref}\|_X$}
& \shortstack{$\dfrac{\|v_\theta-v_{\rm ref}\|_X}{\|v_{\rm ref}\|_X}$}\\
\midrule
State in $H$ & $1.5631\times10^{-3}$ & $9.5282\times10^{-4}$\\
Adjoint in $H$ & $1.2373\times10^{-3}$ & $5.1195\times10^{-4}$\\
Control in $G_-$ & $1.4724\times10^{-3}$ & $6.2718\times10^{-4}$\\
Control in $U$ & $2.5222\times10^{-3}$ & $1.0428\times10^{-3}$\\
Dose in $L^2(\Omega)$ & $8.5872\times10^{-4}$ & $2.5707\times10^{-4}$\\
\bottomrule
\end{tabular}
\caption{\rev{Regular manufactured disc benchmark at the final checkpoint.
All errors use the separate finer quadrature.
Here $v_\theta$ and $v_{\rm ref}$ are the approximate and reference
quantities in each row, and $X$ is the indicated norm.}}
\label{tab:regular-disc-revised-errors}
\end{table}

A comparison of the approximate fields, inflow control and dose with their exact counterparts is shown in Figure~\ref{fig:regular-disc-revised-solution} . The state attenuates along each characteristic, while the negative adjoint vanishes at outflow. \rev{The control is recovered accurately in both boundary norms, and
the approximate dose closely follows its exact profile.}

\begin{figure}[!htbp]
\centering
\includegraphics[width=.98\linewidth]{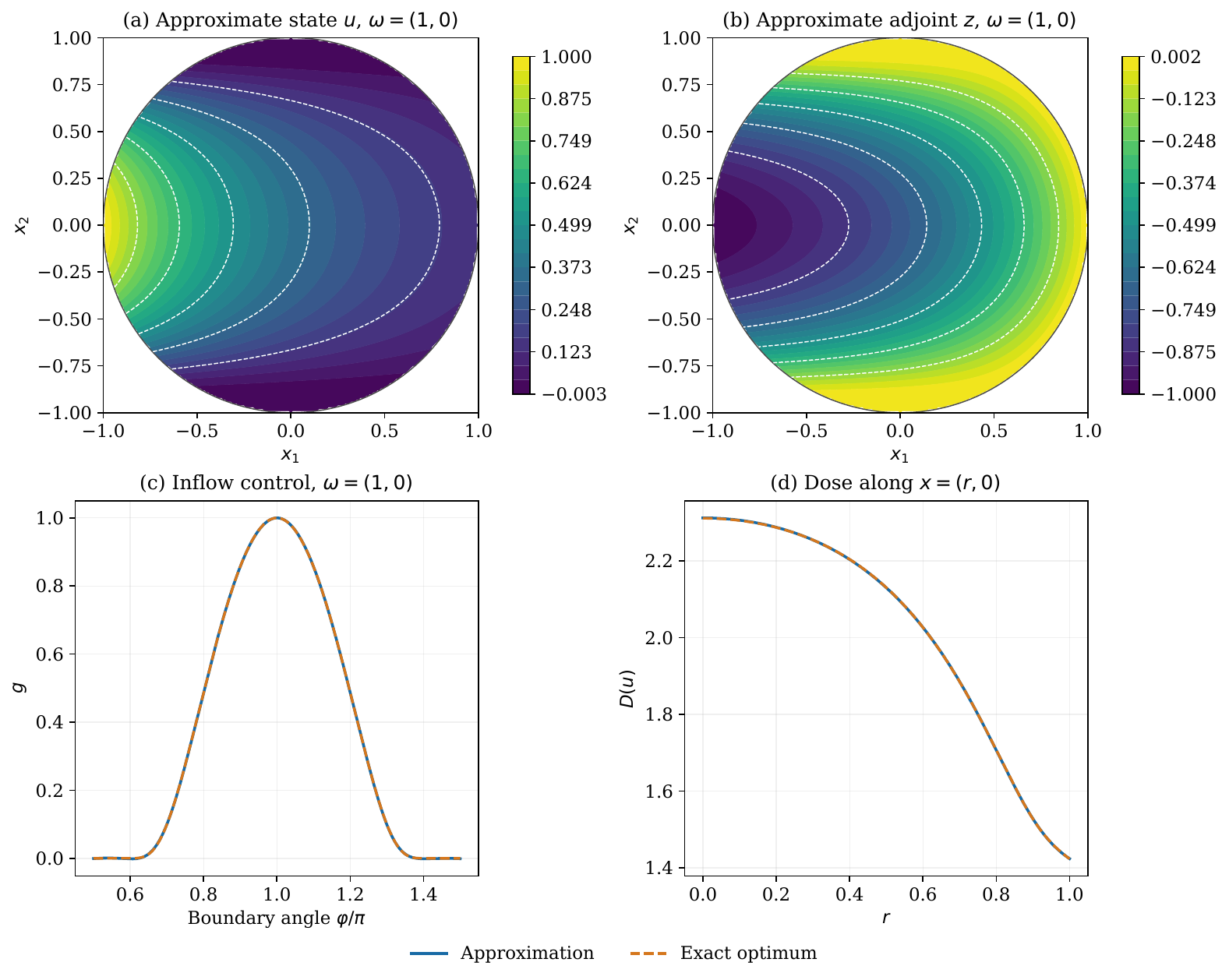}
\caption{Regular manufactured disc benchmark. The top row shows the approximate
state and adjoint at $\omega=(1,0)$, with exact level sets in white.
The bottom row compares the approximate and exact inflow control (left)
and dose along $x=(r,0)$ (right).}
\label{fig:regular-disc-revised-solution}
\end{figure}

For the residual comparison, define
\[
\begin{aligned}
E_b^2&=\|u_\theta-u_{\rm ex}\|_H^2
+\|z_\theta-z_{\rm ex}\|_H^2
+\|g_\theta-g_{\rm ex}\|_{G_-}^2,\\
E_U^2&=\|u_\theta-u_{\rm ex}\|_H^2
+\|z_\theta-z_{\rm ex}\|_H^2
+\|g_\theta-g_{\rm ex}\|_U^2.
\end{aligned}
\]
We evaluate these errors and $\mathcal J_{b,h}$ and
$\mathcal J_{U,h}$ with the finer rule at nine stages of the training. Figure~\ref{fig:regular-disc-revised-residual}(b) shows that the squared errors decrease together with their corresponding
residuals. The observed ratios satisfy
\[
0.3993 \le \frac{E_b^2}{\mathcal J_{b,h}} \le 0.9992,
\qquad
0.6583 \le \frac{E_U^2}{\mathcal J_{U,h}} \le 1.0342.
\]
These observations support the residual stability estimate for this benchmark without identifying its global reliability constant. \rev{For $\alpha=1$ and $L=2$, the directional choice in
Table~\ref{tab:stability-constants} gives $C_{\rm stab}=712$. Compared
with the observed final ratio $E_b^2/\mathcal J_{b,h}=0.5833$, this indicates a rather conservative global bound.}

\begin{figure}[!htbp]
\centering
\includegraphics[width=.98\linewidth]{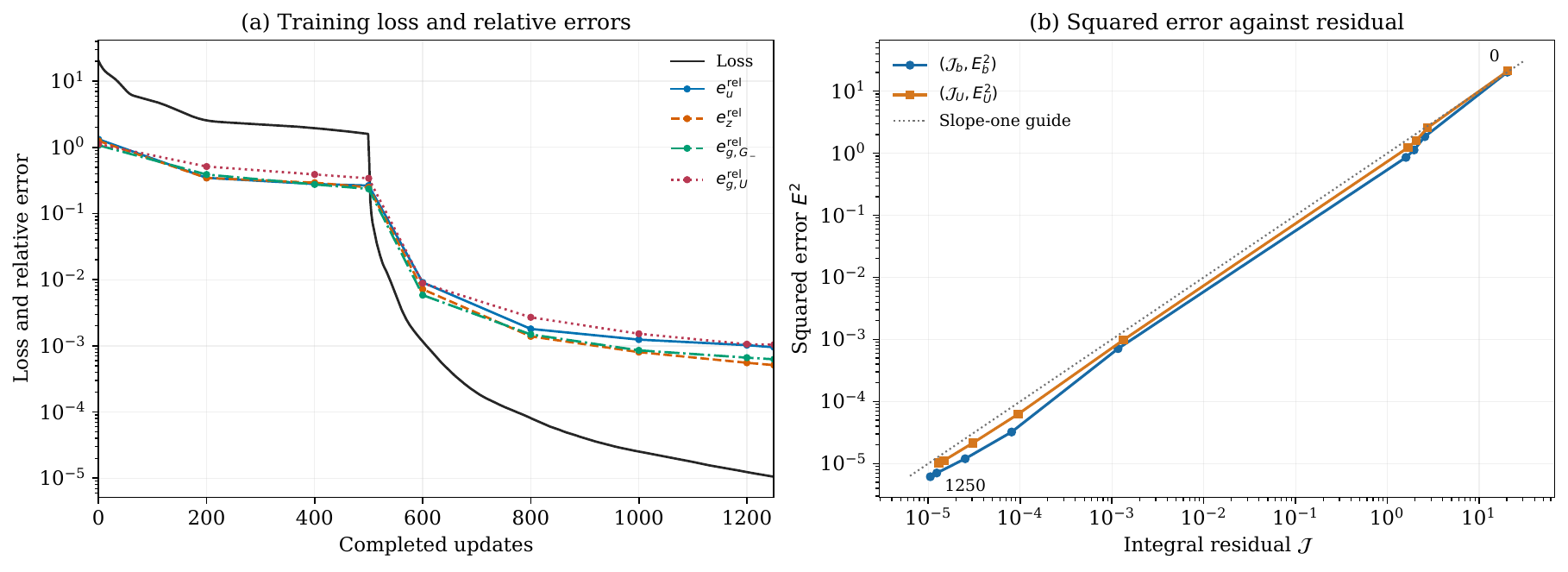}
\caption{\rev{Regular manufactured disc benchmark. Left: training loss
and relative state, adjoint and control errors, with the errors evaluated
on the finer rule at the saved checkpoints. Right: combined squared
absolute errors against their corresponding integral residuals,
evaluated on the finer rule. The dotted line is a slope-one guide.}}
\label{fig:regular-disc-revised-residual}
\end{figure}

\subsubsection*{\rev{Three-dimensional KKT solution.}}
\rev{The same construction is applied in $\Omega=B_1(0)\subset\mathbb R^3$,
with $\omega\in S^2$ and $\sigma=\alpha=1$. The exact formulas
\eqref{eq:regular-disc-revised-exact} give the state, adjoint and inflow
control in this five-dimensional phase space.} \rev{Interior points and directions are sampled uniformly and fixed throughout training. Spatial samples are drawn separately for each
direction, whereas the boundary points are reused, as specified in
Table~\ref{tab:experiment-settings}.} \rev{The training loss and the relative errors over $2\,500$ recorded training steps are reported in Table~\ref{tab:ball-field-training}.}

\rev{The reported control error is in $G_-$. \rev{The final relative errors are of order $10^{-3}$ on the training samples. The quadrature reuses directions across spatial points and
therefore differs from the independent phase-space sampling in
Section~\ref{sec:mc-residual-experiment}.}}

\begin{table}[htbp]
\centering\small\color{black}
\begin{tabular}{ccccc}
\toprule
Recorded step & Loss
& $\dfrac{\|u_\theta-u_{\rm ex}\|_H}{\|u_{\rm ex}\|_H}$
& $\dfrac{\|z_\theta-z_{\rm ex}\|_H}{\|z_{\rm ex}\|_H}$
& $\dfrac{\|g_\theta-g_{\rm ex}\|_{G_-}}{\|g_{\rm ex}\|_{G_-}}$\\[4pt]
\midrule
$1$ & $2.2436$ & $2.5907$ & $1.2388$ & $1.3713$\\
$500$ & $2.1056\times10^{-1}$ & $4.5218\times10^{-1}$ & $3.8011\times10^{-1}$ & $4.8808\times10^{-1}$\\
$600$ & $5.9999\times10^{-3}$ & $5.0647\times10^{-2}$ & $5.1280\times10^{-2}$ & $4.3514\times10^{-2}$\\
$1\,000$ & $7.0479\times10^{-5}$ & $4.9340\times10^{-3}$ & $6.3377\times10^{-3}$ & $2.9540\times10^{-3}$\\
$1\,500$ & $7.4395\times10^{-6}$ & $2.3741\times10^{-3}$ & $1.5970\times10^{-3}$ & $1.4301\times10^{-3}$\\
$1\,826$ & $3.9417\times10^{-6}$ & $1.6717\times10^{-3}$ & $1.1925\times10^{-3}$ & $1.2627\times10^{-3}$\\
$2\,500$ & $3.9417\times10^{-6}$ & $1.6717\times10^{-3}$ & $1.1925\times10^{-3}$ & $1.2627\times10^{-3}$\\
\bottomrule
\end{tabular}
\caption{\rev{Training loss and relative errors evaluated on the training samples for the manufactured field-tracking solution in the unit ball.}}
\label{tab:ball-field-training}
\end{table}

}

\subsection{\rev{Example: Manufactured-solution verification of three-dimensional dose control}}
\label{ex:ball-dose}
{
\color{black}
We consider $\Omega=B_1(0)\subset\mathbb R^3$ and $\omega\in S^2$, with
the unnormalised dose $D(u)(x)=\int_{S^2}u(x,\omega)\,d\omega$.
The objective is
\[
 F_{\rm dose}(u,g)=\frac12\|D(u)-q\|_{L^2(\Omega)}^2
                  +\frac\alpha2\|g\|_U^2,
 \qquad Tu=0,\quad\gamma^-u=g.
\]
We use a nonradial manufactured solution to check the residual--error relation for three-dimensional dose control, including the angular coupling in the adjoint equation.

For $\alpha,\sigma>0$, set
\[
 h(x)=1+0.3x_1+0.2x_2^2+0.1x_1x_3,
 \qquad q(x)=D(u^*)(x)+h(x).
\]
Write $d=x\cdot\omega$, $\Delta=(1-|x|^2+d^2)^{1/2}$,
$\tau_-=\Delta+d$ and $\tau_+=\Delta-d$.
For $(y,\omega)\in\Gamma^-$, let $\ell(y,\omega)=-2y\cdot\omega$ be the
characteristic length and $b(y,\omega)=-y\cdot\omega$. Define
\begin{equation}\label{eq:ball-dose-exact}
\begin{aligned}
 z^*(x,\omega)&=-\int_0^{\tau_+(x,\omega)}
            e^{-\sigma s}h(x+s\omega)\,ds,\\
 g^*(y,\omega)&=\frac{b(y,\omega)}\alpha
       \int_0^{\ell(y,\omega)}e^{-\sigma s}h(y+s\omega)\,ds,\\
 u^*(x,\omega)&=e^{-\sigma\tau_-(x,\omega)}
       g^*(x-\tau_-(x,\omega)\omega,\omega).
\end{aligned}
\end{equation}
The integrands are quadratic polynomials times an exponential, so the
state, adjoint and control are evaluated analytically. Differentiation
along characteristics gives
\[
 Tu^*=0,\qquad T^*z^*=-h=D(u^*)-q,
 \qquad\gamma^-u^*=g^*,\quad\gamma^+z^*=0,
 \quad\alpha g^*+b\gamma^-z^*=0.
\]
These bounded fields are well-defined, hence, by
Theorem~\ref{thm:corrected-wellposed}, give the unique optimum of the KKT system. The optimal target
mismatch is known explicitly,
\[
 \|D(u^*)-q\|_{L^2(\Omega)}=\|h\|_{L^2(\Omega)}
       =\sqrt{\frac{3856\pi}{2625}}\simeq2.148220.
\]

\rev{We take $\alpha=\sigma=1$. The network has $29\,763$ parameters,
with no imposed symmetry. The computational settings are listed
in Table~\ref{tab:experiment-settings}. The azimuthal nodes of the
finer rule are shifted relative to training.}

The reference dose is evaluated independently of the training
directions. Exact azimuthal averaging of the polynomial reduces
its angular integral to one dimension, where an order-doubling
check changes the reference values by less than $8\times10^{-15}$
on the finer spatial rule. On this rule, angular integration of
the exact state differs from the reference dose by
$2.534\times10^{-6}$ in $L^2(\Omega)$.

Table~\ref{tab:ball-dose-errors} reports errors on the finer rule.
The dose error includes angular integration of the approximate
state and compares it with $D(u^*)$, rather than with $q$.
The final residuals are $\mathcal J_{b,h}=5.2662\times10^{-4}$ on
the training rule and $5.2701\times10^{-4}$ on the finer rule.
Replacing the stationarity residual norm by $U$ gives
$\mathcal J_{U,h}=5.8803\times10^{-4}$ on the latter.

\begin{table}[htbp]
\centering
\caption{\rev{Manufactured three-dimensional dose control. Final errors
on the separate finer quadrature. Here $v_\theta$ and $v_{\rm ref}$ are the approximate and reference
quantities in each row, and $X$ is the indicated norm.}}
\label{tab:ball-dose-errors}
\rev{\begin{tabular}{lccc}
\toprule
Quantity & Norm & \shortstack{$\|v_\theta-v_{\rm ref}\|_X$}
& \shortstack{$\dfrac{\|v_\theta-v_{\rm ref}\|_X}{\|v_{\rm ref}\|_X}$}\\
\midrule
$u_\theta-u^*$ & $H$ & $6.962\times10^{-3}$ & $2.753\times10^{-3}$\\
$z_\theta-z^*$ & $H$ & $1.203\times10^{-2}$ & $3.024\times10^{-3}$\\
$g_\theta-g^*$ & $U$ & $1.533\times10^{-2}$ & $3.736\times10^{-3}$\\
$g_\theta-g^*$ & $G_-$ & $8.777\times10^{-3}$ & $2.399\times10^{-3}$\\
$D_h(u_\theta)-D(u^*)$ & $L^2(\Omega)$ & $4.440\times10^{-3}$ & $5.625\times10^{-4}$\\
\bottomrule
\end{tabular}}
\end{table}

Figure~\ref{fig:ball-dose-fields} compares the approximate and exact
state and adjoint on a central section for
$\omega=(1,2,3)/\sqrt{14}$. Figure~\ref{fig:ball-dose-profiles}
shows the corresponding dose profiles along the coordinate axes.
The computed profiles closely follow the exact optimal dose.
Their separation from the prescribed dose is the mismatch $h$,
which is present at the exact optimum.

\begin{figure}[htbp]
\centering
\includegraphics[width=.98\linewidth]{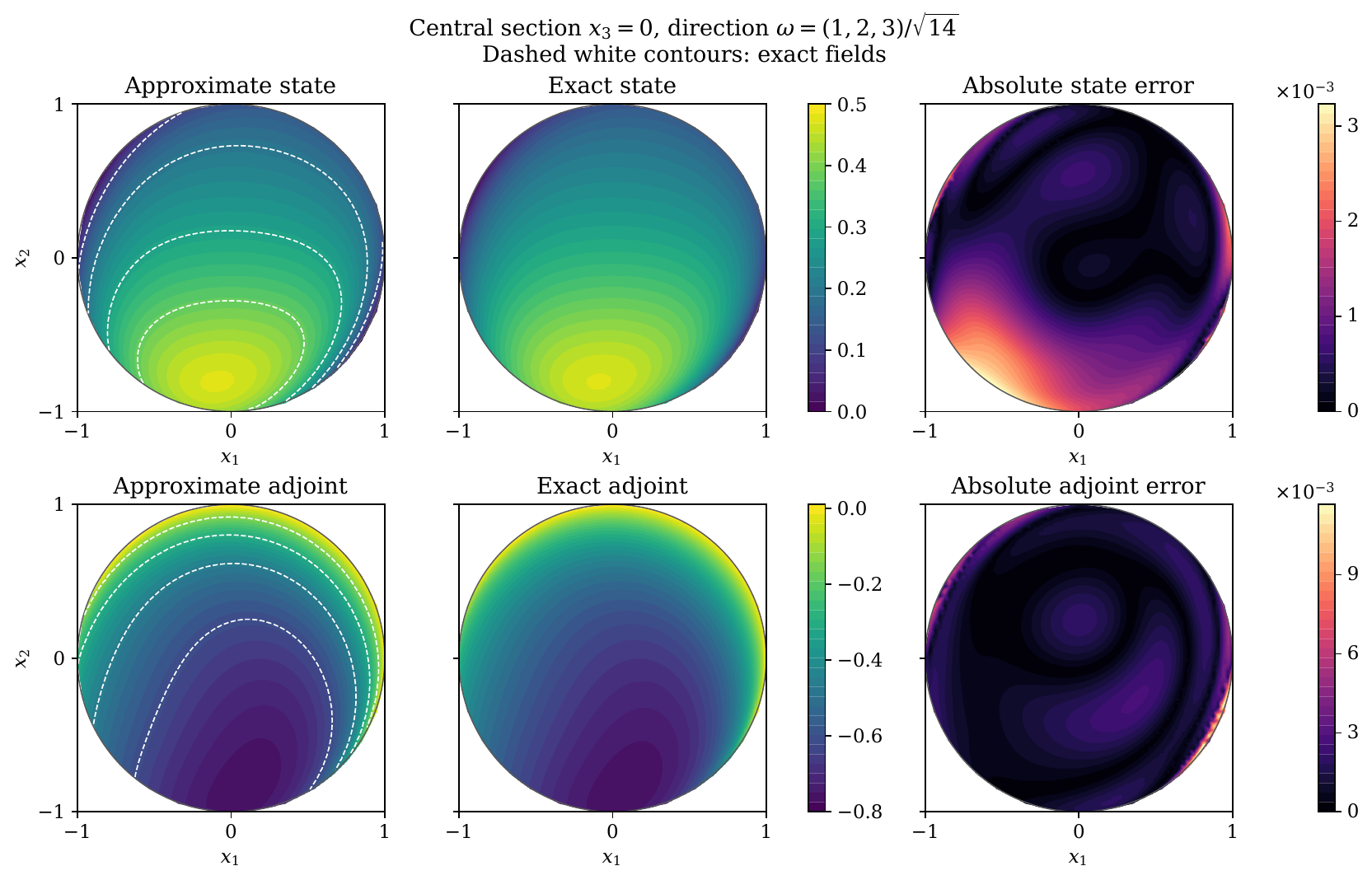}
\caption{\rev{Manufactured three-dimensional dose control. Approximate
and exact state (top) and adjoint (bottom) on $x_3=0$, for
$\omega=(1,2,3)/\sqrt{14}$, with their absolute errors. Approximate
and exact fields share the same colour scale in each row.
\rev{Dashed white curves show contours of the exact fields on the
approximate panels. The contours are interpolated linearly between
the saved section nodes.}}}
\label{fig:ball-dose-fields}
\end{figure}

\begin{figure}[htbp]
\centering
\includegraphics[width=.98\linewidth]{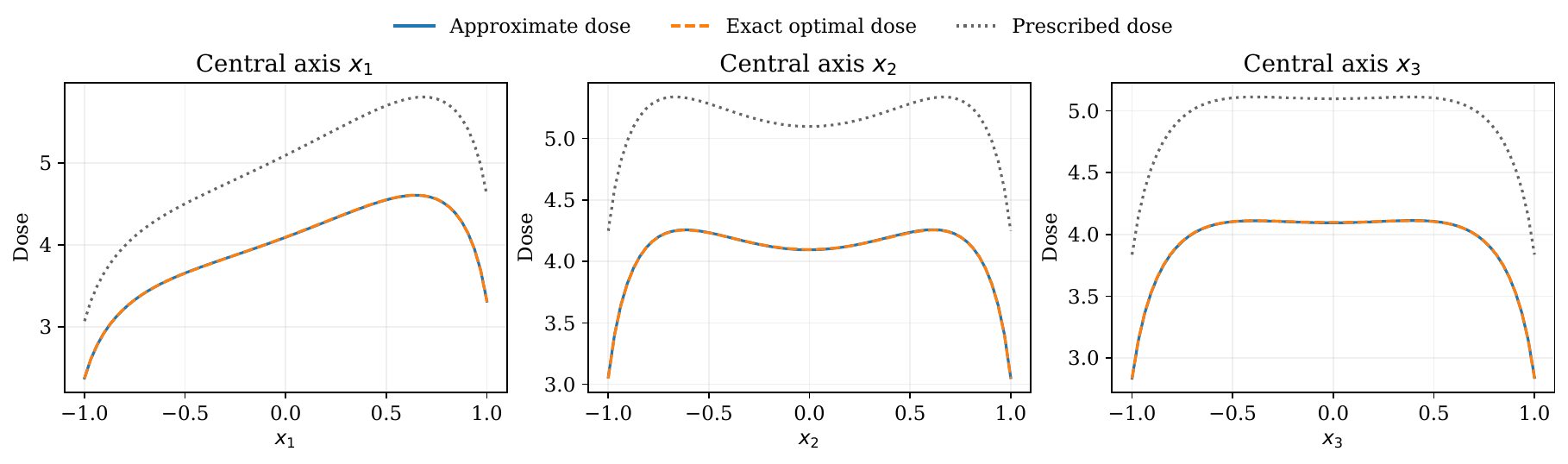}
\caption{\rev{Manufactured three-dimensional dose control. Approximate
and exact optimal doses along the three coordinate axes, together
with the prescribed dose $q=D(u^*)+h$.}}
\label{fig:ball-dose-profiles}
\end{figure}
\rev{Let $E_b^2$ and $E_U^2$ be the sums of the squared state,
adjoint and control errors, using $G_-$ and $U$ for the control,
respectively. During the saved L-BFGS stages, both errors decrease
with their corresponding residuals, and the observed ratios satisfy
\[
 0.4995\le\frac{E_b^2}{\mathcal J_{b,h}}\le0.7794,
 \qquad
 0.7281\le\frac{E_U^2}{\mathcal J_{U,h}}\le1.0850.
\]
The final ratios are $0.5126$ and $0.7281$, respectively.
Thus three-dimensional dose control exhibits the residual--error
behaviour observed in the regular disc benchmark.}

}

\subsection{\rev{Example: Parametric ellipsoidal dose targets in the unit cube}}
\label{ex:cube-sphere-dose}
{
\color{black}
As a further example, we use one neural network to approximate a family of dose-control problems with idealised fixed-volume ellipsoidal target shapes. This is an application that neural network approximation might be advantageous. In this case, we consider $\Omega=(0,1)^3$, with attenuation $\sigma=0.1$ and regularisation $\alpha=10^{-4}$. Let $c=(0.5,0.5,0.5)$, $r=0.2$ and $n=(1,2,3)/\sqrt{14}$. For $y=x-c$, define
\[
 T_\mu=\left\{x\in\Omega:
 e^{-\mu}(y\cdot n)^2+
 e^{\mu/2}\left|y-(y\cdot n)n\right|^2<r^2\right\},
 \qquad q_\mu=\chi_{T_\mu}.
\]
The semiaxes are $r e^{\mu/2}$ along $n$ and $r e^{-\mu/4}$ in
 the two transverse directions. Hence, the target changes from flattened
 to elongated at fixed volume $|T_\mu|=4\pi r^3/3$, and $\mu=0$
 gives the spherical target.

We approximate the dose-control problems at $\mu\in\{-2,-1,0,1,2\}$ by one network
$(u_\theta,z_\theta,g_\theta)(x,\omega,\mu)$, using
$D(u)=\int_{S^2}u\,d\omega$. Training minimises the arithmetic mean
of the five discrete residual functionals. The same shape values are
used for independent spatial and angular evaluation. The computational
settings are given in Table~\ref{tab:experiment-settings}.

The spherical member of this parametric approximation and the joint training history are shown in Figure~\ref{fig:cube-sphere-dose-opaque}, while Figure~\ref{fig:parametric-ellipsoid-spatial} displays the dose profiles for all five shapes. The profiles are evaluated along the oblique axis $n$. Between updates $1\,400$ and $2\,000$, the target transition sharpens while the central dose increases above the prescribed level. Table~\ref{tab:parametric-ellipsoid-mismatch} reports the corresponding target mismatches. At the final update $2\,500$, the joint target mismatch divided by the known target norm is $0.3638$. As in Example~\ref{ex:circle-dose}, these are target mismatches rather than errors against an exact optimal dose. The final training and independent-rule residuals are $5.4667\times10^{-4}$ and $1.0393\times10^{-2}$, respectively.
The adjoint equation gives the largest component on the independent rule.

\begin{figure}[!htbp]
\centering
\includegraphics[width=.95\linewidth]{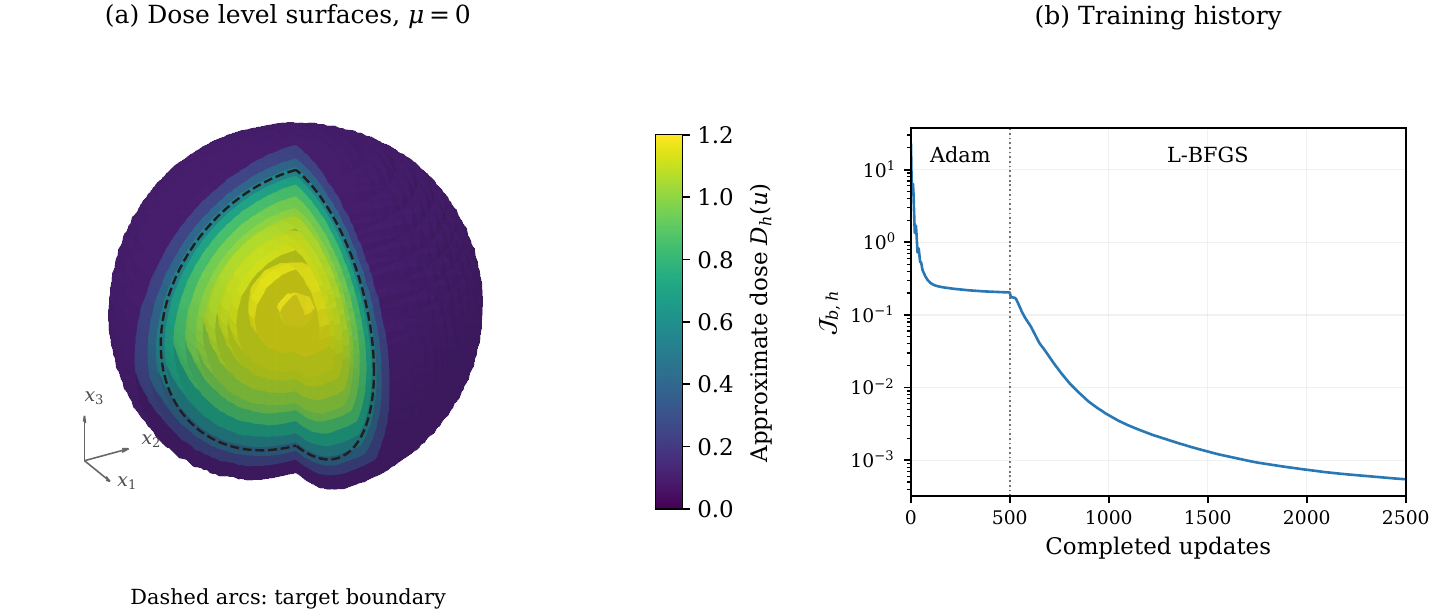}
\caption{\rev{Parametric ellipsoidal dose targets in the cube. (a) Dose level
surfaces for $\mu=0$, with a quarter removed. Dashed arcs mark the spherical
 target boundary on the exposed planes. The surfaces use direct evaluation
of the final parametric approximation with the $800$-direction rule.
(b) Joint training history for the five equally weighted shapes.}}
\label{fig:cube-sphere-dose-opaque}
\end{figure}

\rev{Figure~\ref{fig:parametric-ellipsoid-spatial} also shows the final
spatial dose sections and central profiles at updates
$1\,400$, $1\,600$ and $2\,000$ for the five shapes in the coordinates
$x=c+s n+t v$, where $v=(2,-1,0)/\sqrt{5}$.}
\begin{figure}[!htbp]
\centering
\includegraphics[width=.98\linewidth]{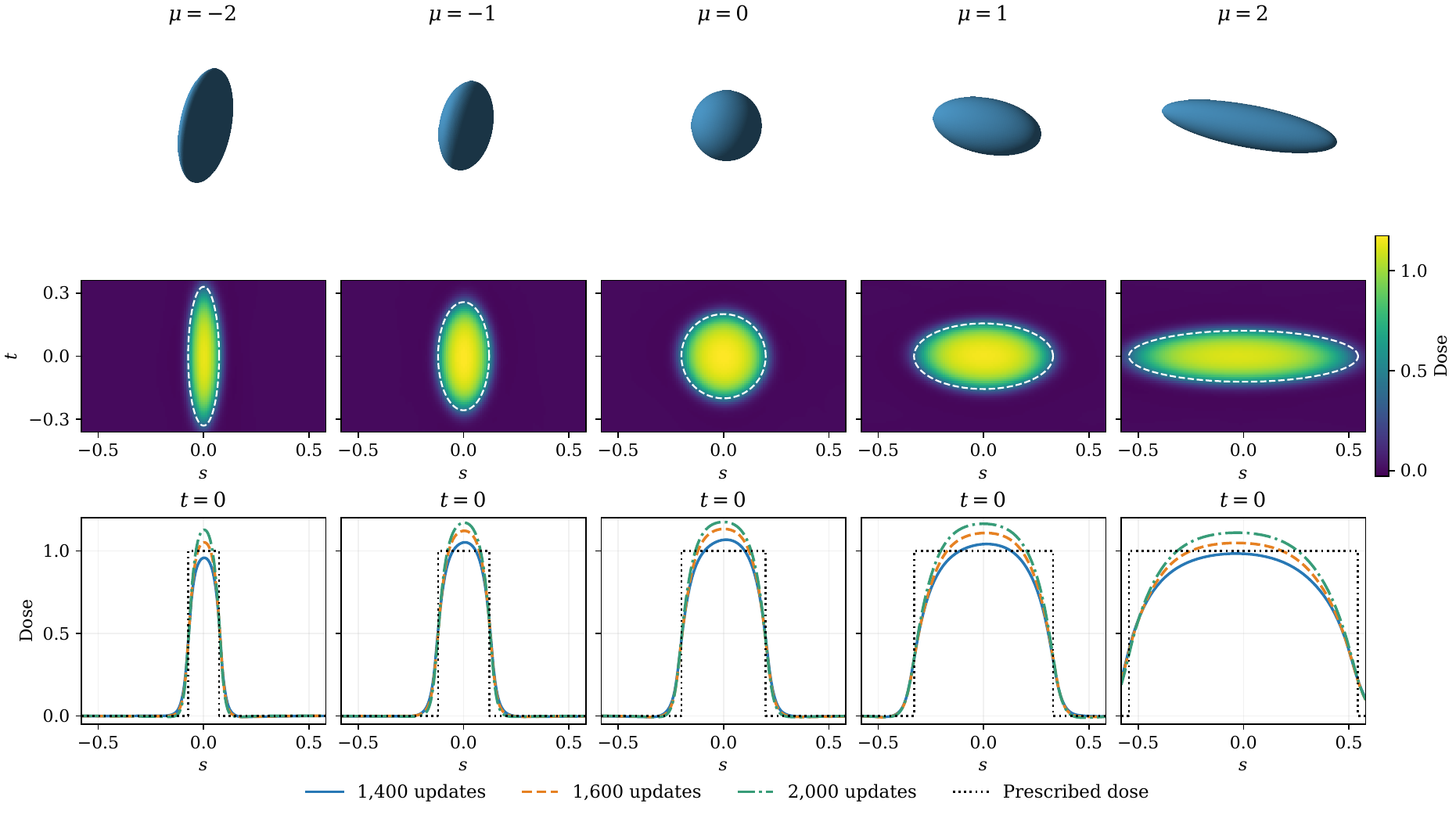}
\caption{\rev{Parametric ellipsoidal dose targets.
Top: the five target geometries in the fixed frame with first axis $n$.
Middle: final approximate dose at update $2\,500$ in the principal-axis
plane $x=c+s n+t v$,
with dashed target boundaries and a common colour scale.
Bottom: central dose profiles at $t=0$ and updates $1\,400$, $1\,600$
and $2\,000$, with the prescribed dose shown by dotted lines.
Each column uses the same parameter value throughout.
Update counts include $500$ Adam steps.}}
\label{fig:parametric-ellipsoid-spatial}
\end{figure}

\begin{table}[!htbp]
\centering
\caption{\rev{Parametric ellipsoidal dose targets. Target mismatches at update $2\,500$ on
independent spatial samples with $800$ directions. The denominator is
 the known target norm $\|q_\mu\|_{L^2(\Omega)}=\sqrt{4\pi r^3/3}$.}}
\label{tab:parametric-ellipsoid-mismatch}
{\color{black}\small
\renewcommand{\arraystretch}{1.18}
\begin{tabular}{@{}ccc@{}}
\toprule
$\mu$ & $\|D_h(u_\theta(\mu))-q_\mu\|_{L^2(\Omega)}$
& $\|D_h(u_\theta(\mu))-q_\mu\|_{L^2(\Omega)}/\|q_\mu\|_{L^2(\Omega)}$\\
\midrule
$-2$ & $0.06909$ & $0.3774$\\
$-1$ & $0.06763$ & $0.3695$\\
$0$ & $0.06558$ & $0.3582$\\
$1$ & $0.06422$ & $0.3508$\\
$2$ & $0.06637$ & $0.3626$\\
\bottomrule
\end{tabular}
}
\end{table}
}

\subsection{\rev{Example: Parametric bilobed dose targets}}
\label{ex:bilobed-dose-ongoing}
\rev{Here we use one neural network to approximate a family of dose-control problems with bilobed target shapes. These targets vary the separation of fixed-size lobes, introducing nonconvex geometries while remaining connected. 
In this case, we consider $\Omega=(0,1)^3$,
$\sigma=0.1$ and $\alpha=10^{-4}$. Let $c=(0.5,0.5,0.5)$,
$n=(1,2,3)/\sqrt{14}$, $v=(2,-1,0)/\sqrt{5}$ and
$F=(n,v,n\times v)$. The two ellipsoidal lobes have semiaxes
$a=(0.16,0.22,0.16)$ and centres
$c_\pm(\mu)=c\pm0.16(0.54+0.44\mu)n$. Define}
{\color{black}
\[
 \begin{aligned}
 \phi_\pm(x,\mu)&=\sum_{j=1}^3
 \left(\frac{[F^T(x-c_\pm(\mu))]_j}{a_j}\right)^2-1,
 \\
 T_\mu&=\left\{x\in\Omega:
 -\frac1{20}\log\!\left(e^{-20\phi_-}+e^{-20\phi_+}\right)<0\right\},
 \qquad q_\mu=\chi_{T_\mu}.
 \end{aligned}
\]
}
\rev{The separation increases with $\mu$, while the lobe sizes remain
fixed. The target volume varies with the separation. All targets remain connected and strictly inside the cube.
Figure~\ref{fig:bilobed-geometry-preview} shows the five shapes at
$\mu\in\{-0.5,0,0.5,0.75,1\}$. One network
$(u_\theta,z_\theta,g_\theta)(x,\omega,\mu)$ approximates the
corresponding dose-control solutions, using the settings in
Table~\ref{tab:experiment-settings}. Training minimises the arithmetic
mean of the five discrete residual functionals. We report the approximation
at update $2\,500$, after $500$ Adam and $2\,000$ L-BFGS updates.

The dose sections and profiles in Figure~\ref{fig:bilobed-geometry-preview}
follow the changing lobe separation. The boundary transitions remain
smooth, with an interior overshoot. For the most separated target,
the off-centre profile at $t=0.12$ resolves the central dip, although
a positive dose remains in the sectional gap. Figure~\ref{fig:bilobed-dose-cutaway}
shows a central cutaway for $\mu=-0.5$ and the joint training history.
The recorded training residual is  $7.6780\times10^{-4}$ and decreases by approximately $64\%$ between updates $1\,400$ and $2\,500$. Table~\ref{tab:bilobed-dose-mismatch} reports the target mismatches
on independent spatial samples with $800$ directions. The joint relative target mismatch is $0.3373$.}

\begin{figure}[!htbp]
\centering
\includegraphics[width=.98\linewidth]{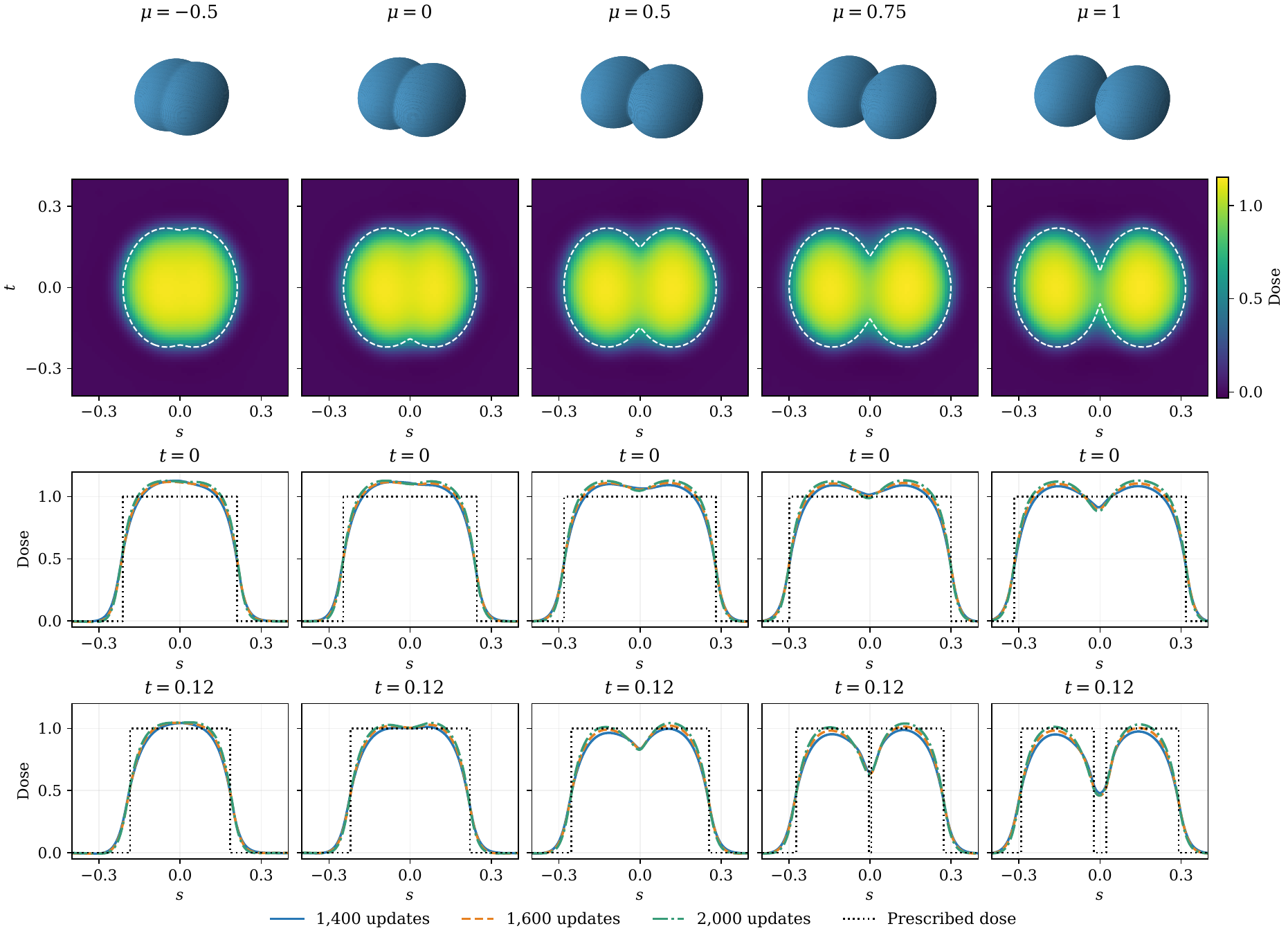}
\caption{\rev{Parametric bilobed dose targets.
Top: the three-dimensional target geometries.
Second row: final approximate doses at update $2\,500$ in the oblique
plane $x=c+s n+t v$, with dashed target boundaries and a common colour scale.
Below: dose profiles at $t=0$ and $t=0.12$ at updates $1\,400$,
$1\,600$ and $2\,000$, with the prescribed dose shown by dotted lines.
Each column uses the same parameter value throughout. Doses are evaluated
with $800$ directions. Update counts include $500$ Adam steps.}}
\label{fig:bilobed-geometry-preview}
\end{figure}

\begin{figure}[!htbp]
\centering
\includegraphics[width=.95\linewidth]{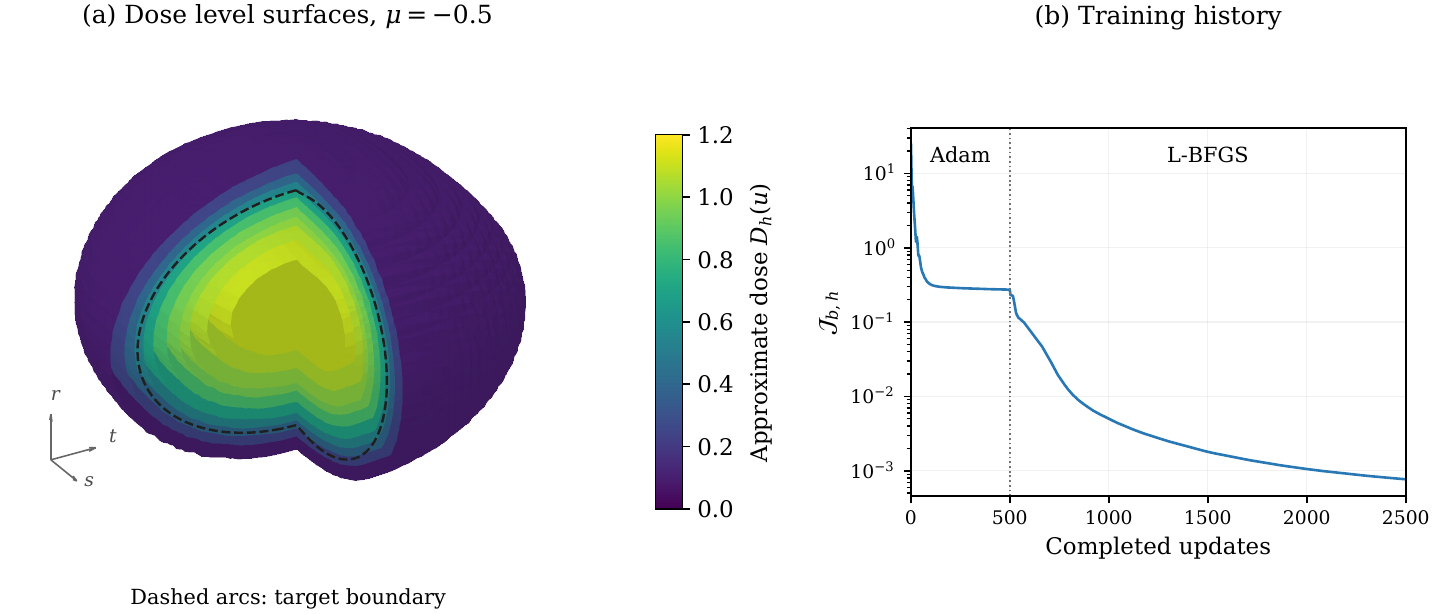}
\caption{\rev{Bilobed dose target for $\mu=-0.5$ at update $2\,500$.
(a) Dose level surfaces with a quarter removed through the target centre
in the oblique frame. Dashed arcs mark the target boundary on the
exposed planes. The surfaces use direct evaluation with $800$ directions.
(b) Recorded joint training history for the five equally weighted shapes.}}
\label{fig:bilobed-dose-cutaway}
\end{figure}

\begin{table}[!htbp]
\centering
\caption{\rev{Bilobed dose targets. Target mismatches at update $2\,500$
on independent spatial samples with $800$ directions. The denominator
uses the accurately evaluated target norm
$\|q_\mu\|_{L^2(\Omega)}=|T_\mu|^{1/2}$.}}
\label{tab:bilobed-dose-mismatch}
{\color{black}\small
\renewcommand{\arraystretch}{1.18}
\begin{tabular}{@{}ccc@{}}
\toprule
$\mu$ & $\|D_h(u_\theta(\mu))-q_\mu\|_{L^2(\Omega)}$
& $\|D_h(u_\theta(\mu))-q_\mu\|_{L^2(\Omega)}/\|q_\mu\|_{L^2(\Omega)}$\\
\midrule
$-0.5$ & $0.06363$ & $0.3421$\\
$0$ & $0.07224$ & $0.3573$\\
$0.5$ & $0.07018$ & $0.3296$\\
$0.75$ & $0.07082$ & $0.3280$\\
$1$ & $0.07220$ & $0.3324$\\
\bottomrule
\end{tabular}
}
\end{table}

\section{Conclusions}\label{sec:conclusions}
{\color{black}

We have established global residual stability for the KKT system
of transport boundary control. The continuum residual bounds the state
and adjoint errors in the volume norm and the control error in the
flux-weighted boundary norm. The proof reconciles this transport
measure with the unweighted control penalty and requires no positive
lower bound on $|\omega\cdot n|$. This stability result, followed by
the approximation, quadrature and optimisation error analysis, is the
main contribution of the paper. Joint neural network approximation of the
KKT variables provides its numerical realisation. The estimate holds
for fixed $\alpha>0$ and does not assert uniform stability as
$\alpha\to0$.

The distinction between the two boundary measures also determines
what is certified. Strengthening the optimality residual alone gives
an estimate in the unweighted control norm and hence a bound on
reduced-objective suboptimality. Weighted residual convergence by
itself does not imply these stronger conclusions.
\rev{The disc benchmark gives an unweighted relative control
error of $1.043\times10^{-3}$ on the separate finer quadrature.}

At the discrete level, the loss is evaluated by quadrature.
For Monte Carlo sampling, we have related the empirical residual to
the continuum functional through root-mean-square and
high-probability estimates. The fixed-field quadrature error has the
$M^{-1/2}$ rate under the stated moment assumptions. For fields
selected during training, uniform estimates over the neural network class
provide the required sampling control. Combining this discrepancy
with residual stability yields a probabilistic solution-error bound.
Combining it with approximation and optimisation gives the convergence
result. Deterministic quadrature and independent validation fit the
same residual-error argument. The sampling rate concerns the loss.
It does not by itself assert the same rate for the solution norm.

The numerical results show how these distinctions enter in
practice. The manufactured examples yield recorded relative errors
of order $10^{-3}$ in two dimensions and three dimensions. 
For the circular and triangular targets, the
characteristic reference reveals larger errors in the directional
field, while its angular integral is recovered more accurately.
These comparisons assess approximation of the optimal field.
They do not identify its angular integral with the prescribed target.
The fixed-field Monte Carlo study isolates quadrature error and
exhibits the expected root-mean-square rate.
\rev{The manufactured dose-tracking examples give relative
dose errors of $2.042\times10^{-4}$ in the disc and
$5.625\times10^{-4}$ in the ball against their exact optima,
independently of the nonzero target mismatch. The regular disc and
manufactured ball-dose examples also compare exact solution errors
with residuals at saved training stages. The parametric cube examples illustrate joint approximation of
dose-control solutions for ellipsoidal and bilobed target families,
using one network per family at five prescribed shape-parameter values.
These experiments assess approximation at the prescribed parameter
values rather than interpolation to unseen shapes.}

The explicit reliability constants expose another limitation: they
may be large at the small regularisation parameters used in target
tracking. For angular-flux tracking the characteristic structure
improves the bound, but an $\alpha^{-2}$ dependence is unavoidable
for the squared-error constant of the unscaled residual. Comparing
loss scalings at a fixed geometry, target and approximation class
would help assess their effect on training. The present computations
use different targets and configurations and do not constitute a
controlled numerical study of sensitivity to $\alpha$.

A direct next step towards treatment planning is to
\rev{extend the dose-tracking formulation to} nonnegative inflow
constraints. The latter leads to a variational inequality and requires
a corresponding change to the residual analysis. Scattering introduces
further angular coupling, while energy dependence is needed for more
complete transport models. For the present model, the characteristic
reference gives a useful benchmark for studying refinement, optimiser
accuracy and computational cost. The analysis developed here supplies
the stability estimate against which such extensions and numerical
comparisons can be assessed.
}

\subsubsection*{Acknowledgement} The authors acknowledge support by the Hellenic Foundation for Research and Innovation (H.F.R.I.) under the \enquote{2nd Call for H.F.R.I.~Research Projects to support Post-Doctoral Researchers} (project number: $01247$).

\subsubsection*{Data and code availability}
The source code and data supporting the numerical results will be made
publicly available upon publication of this article.

\appendix
\new{\section{Further estimates for boundary residuals}\label{app:boundary-stability}

\subsection{Explicit stability constants}
Using the residual norms from the proof of Theorem~\ref{thm:corrected-coercivity}, put $q=\|\mathcal O\|^2$ and
$R=(R_1,\ldots,R_5)^\top$. Define the row vectors
\begin{equation}\label{eq:stability-rows}
\begin{aligned}
 a_g&=\kappa(qL^{3/2},\sqrt L,1,qL,1),\\
 a_u&=(L,0,0,\sqrt L,0)+\sqrt L\,a_g,\\
 a_z&=(0,L,0,0,\sqrt L)+Lq\,a_u,
 \qquad \mathsf M=\begin{pmatrix}a_u\\a_z\\a_g\end{pmatrix}.
\end{aligned}
\end{equation}
The transport estimates in that proof give the componentwise inequality
\[
 (\|e_u\|_H,\|e_z\|_H,\|e_g\|_{G_-})^\top\le\mathsf M R.
\]
Consequently, with $\|\cdot\|_{\rm F}$ the Frobenius norm,
\begin{equation}\label{eq:explicit-stability-constant}
 C_{\rm stab}:=2\|\mathsf M\|_{\rm F}^2
 =2(\|a_u\|_2^2+\|a_z\|_2^2+\|a_g\|_2^2)
\end{equation}
is an admissible constant for \eqref{eq:corrected-coercivity}. The factor two accounts for the factor
$1/2$ in $\mathcal J_b$. The more detailed bound
$\|\mathsf M R\|_2$ also gives a reliable estimator from the individual
continuum residual norms.

For Corollary~\ref{cor:unweighted-control}, use the same matrix
with $\kappa=\alpha^{-1}$ and denote it by $\mathsf M_U$. Its proof
gives the componentwise estimate with control error in $U$ and
$R_U=(\|r_1\|_H,\|r_2\|_H,\|r_3\|_U,
\|r_4\|_{G_-},\|r_5\|_{G_+})^\top$. Hence,
\begin{equation}\label{eq:U-stability-constant}
 C_U=2\|\mathsf M_U\|_{\rm F}^2
\end{equation}
is sufficient for every bounded observation operator. Its growth is
of order $\alpha^{-2}$ for fixed $L$ and $\|\mathcal O\|$.
For $\mathcal O=I$, its values are given in the columns of the characteristic bound in
Table~\ref{tab:stability-constants}, applied now to $\mathcal J_U$ and the error in $U$.

\subsection{Dependence on the regularisation parameter}\label{rem:raywise-constant}
For $\mathcal O=I$, the constant can be sharpened. At an inflow
point $(y,\omega)$, let
\[
 m(y,\omega)=\int_0^{\ell(y,\omega)}
 \exp\left(-2\int_0^t\sigma(y+s\omega)\,ds\right)dt,
 \qquad 0\le m\le L.
\]
The characteristic change of variables gives
$(Sh,Sk)_H=(mh,k)_{G_-}$ for $h,k\in G_-$. Hence $A=S^*S$
is multiplication by $m$ and commutes with $P$. It follows that
\begin{equation}\label{eq:raywise-inverse}
 (\alpha I+PA)^{-1}h=\frac{h}{\alpha+bm},\qquad
 \|(\alpha I+PA)^{-1}\|\le\alpha^{-1}.
\end{equation}
Thus one may take $\kappa=\alpha^{-1}$ in
\eqref{eq:stability-rows} for all the angular-flux examples in Section~\ref{sec:numerics}.
For dose tracking, angular integration couples the directions and
the general inverse bound applies.

For angular-flux tracking, deterioration of the unscaled reliability constant as $\alpha\to0$ is unavoidable. Write $d\mu=dS_x\,d\omega$ and
$A_\varepsilon=\{(x,\omega)\in\Gamma^-:0<b<\varepsilon\}$.
These sets have positive measure in the present geometries. Choose
$h_\varepsilon=\mathbf1_{A_\varepsilon}/
\|\mathbf1_{A_\varepsilon}\|_{G_-}$, and set
$e_g=h_\varepsilon$, $e_u=Sh_\varepsilon$, with
$T^*e_z=e_u$ and $\gamma^+e_z=0$.
Only $r_3=(\alpha+bm)h_\varepsilon$ remains nonzero. Since
$\|e_g\|_{G_-}=1$,
\[
 \mathcal J_b\le\tfrac12(\alpha+\varepsilon L)^2,
 \qquad
 \|e_u\|_H^2+\|e_z\|_H^2+\|e_g\|_{G_-}^2\ge1.
\]
Letting $\varepsilon\to0$ shows that every global reliability
constant for this unscaled residual satisfies
$C_{\rm stab}\ge2/\alpha^2$. For fixed $L$, the sharpened
choice \eqref{eq:explicit-stability-constant} has this same order in
$\alpha$, although its numerical value need not be sharp.

\begin{table}[htbp]
\centering
\begin{tabular}{>{}c>{}c>{}c>{}c>{}c}
\toprule
&\multicolumn{2}{c}{\new{Square: $L=\sqrt2$}}
&\multicolumn{2}{c}{\new{Disc or ball: $L=2$}}\\
$\alpha$ & General bound & Characteristic bound & General bound & Characteristic bound\\
\midrule
$1$ & $7.176\times10^2$ & $1.910\times10^2$ & $4.104\times10^3$ & $7.120\times10^2$\\
$10^{-3}$ & $1.731\times10^{14}$ & $8.650\times10^7$ & $1.409\times10^{15}$ & $3.523\times10^8$\\
$10^{-4}$ & $1.729\times10^{18}$ & $8.643\times10^9$ & $1.408\times10^{19}$ & $3.520\times10^{10}$\\
\bottomrule
\end{tabular}
\caption{\new{Rounded values of the explicit choices of $C_{\rm stab}$ from
\eqref{eq:explicit-stability-constant} for $\mathcal O=I$.
The general column uses $\kappa=\alpha^{-1}(1+L/\alpha)$.
The characteristic bound uses $\kappa=\alpha^{-1}$. These are bounds for
the continuum integral residual and the flux-weighted control error,
not fitted constants or observed errors.}}
\label{tab:stability-constants}
\end{table}

\subsection{Residual weights and reliability}
If $\Lambda=\operatorname{diag}(\lambda_1,\ldots,\lambda_5)$
with $\lambda_i>0$ and $\mathcal J_\lambda=\frac12\sum_i\lambda_iR_i^2$, then
\[
 \|e_u\|_H^2+\|e_z\|_H^2+\|e_g\|_{G_-}^2
 \le 2\|\mathsf M\Lambda^{-1/2}\|_2^2\mathcal J_\lambda.
\]
For instance, replacing $r_3$ by $r_3/\alpha$ corresponds to
$\lambda_3=\alpha^{-2}$. This preserves the zero-residual solution
but changes both the loss and the reliability constant. A smaller
constant after rescaling does not by itself give a better error bound
or establish improved conditioning of the neural network optimisation.

For $\mathcal J_b$, the constant in Section~\ref{sec:quadrature} may be chosen as
\eqref{eq:explicit-stability-constant}, with the sharper inverse
bound of Appendix~\ref{rem:raywise-constant} for angular-flux tracking.
Table~\ref{tab:stability-constants} shows that this reliable bound
can be conservative for small $\alpha$. The observed ratio
$\|V-V^*\|_E^2/\mathcal J_b(V)$ does not determine a global constant.
For an estimator $\eta$ of $\|V-V^*\|_E$, the effectivity index is
$\eta/\|V-V^*\|_E$. The squared-error/residual ratio is a different
quantity. A reverse bound $\mathcal J_b(V)\le C\|V-V^*\|_E^2$
does not hold in general, since the residual contains directional
derivatives that the volume error norm does not control.

\subsection{Sampling the unweighted optimality residual}\label{app:U-sampling}
For $\mathcal J_U$, the measure-based Monte Carlo argument of Section~\ref{subsec:mc-quadrature}
applies with $\mu_3=dS_x\,d\omega$, leaving the other four measures
unchanged. Write $m_-^0=\int_{\Gamma^-}dS_x\,d\omega$.
For uniform samples $Y_j$ with law $(m_-^0)^{-1}dS_x\,d\omega$,
the optimality term is approximated by
\[
 \frac{m_-^0}{2M_3}\sum_{j=1}^{M_3}|r_3(V,Y_j)|^2.
\]
The same samples may approximate the inflow term by including
$b(Y_j)$ in its summand. Those estimates do not require
independence between different residual terms. If instead the
optimality term is sampled from $m_-^{-1}b\,dS_x\,d\omega$,
its sampled density is $m_-|r_3|^2/(2b)$. Even for bounded
$r_3$, the second moment of this density can diverge near grazing
directions. Finite-variance and bounded-density hypotheses must
therefore be verified for the chosen measure. Under the corresponding
moment and uniform-class hypotheses, the quadrature proofs in Section~\ref{subsec:mc-quadrature} apply to $\mathcal J_U$, and \eqref{eq:U-discrete-certificate}
replaces the weighted error bound.}

\section{\rev{Supplementary numerical experiments}}
\label{app:supplementary-numerics}

\rev{The supplementary experiments provide an additional manufactured dose
benchmark, target-control comparisons and a fixed-field test of Monte
Carlo quadrature. Their computational settings are included in
Table~\ref{tab:experiment-settings}.}

\subsection{\rev{
\rev{Triangular field target}}}\label{ex:triangle}

Let $\Omega=(0,1)^2$ and let $T$ be the triangle with vertices
$(0.75,0.75),\allowbreak (0.6,0.25),\allowbreak (0.15,0.5)$. We take $\sigma=0.1$ and $\alpha=10^{-4}$ and minimise
\[
F_{\mathrm{field}}(u,g)
=\frac12\|u-\chi_T\|_{L^2(\Omega\times S^1)}^2
+\frac{\alpha}{2}\|g\|_{L^2(\Gamma^-,\,dS_x\,d\omega)}^2
\]
subject to $\omega\cdot\nabla u+\sigma u=0$ and $\gamma^-u=g$.
The adjoint satisfies
\[
-\omega\cdot\nabla z+\sigma z=u-\chi_T,
\qquad \gamma^+z=0,
\qquad \alpha g+b\gamma^-z=0,
\]
where $b=|\omega\cdot n|$. The prescribed field is unscaled and
the control penalty uses $dS_x\,d\omega$. We report
$D(u)(x)=\int_{S^1}u(x,\omega)\,d\omega$ as an additional diagnostic.
The exact state and control follow from
\eqref{eq:rev-characteristic-control}, using analytic intersections
of each ray with the three target edges.

\begin{figure}[!htbp]
\centering
\includegraphics[width=.98\linewidth]{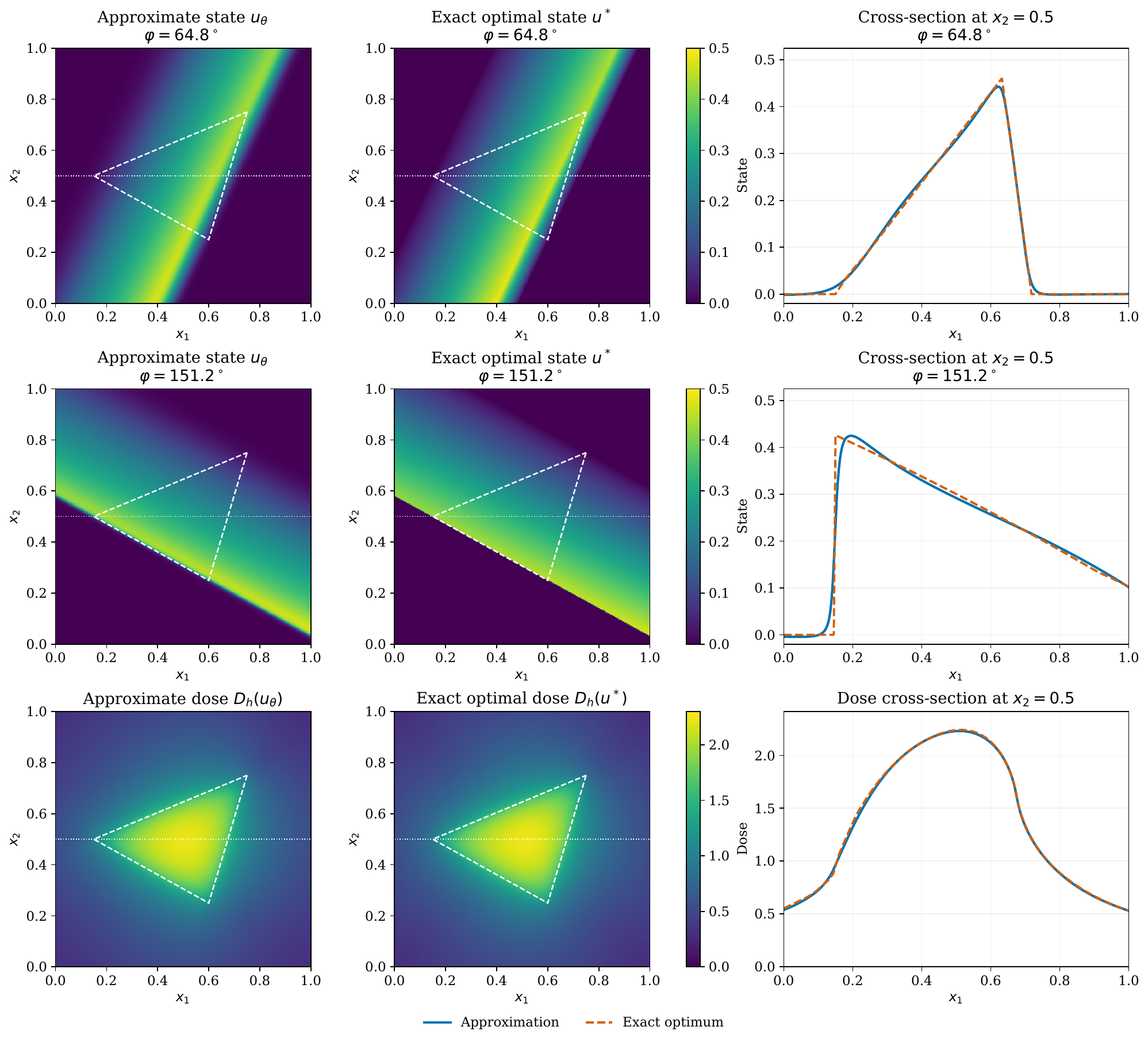}
\caption{\rev{Triangular field target. Approximate and exact optimal
states at $64.8^\circ$ and $151.2^\circ$ occupy the first two rows.
Their angular integrals appear in the third row. The right column
gives cross-sections at $x_2=0.5$. Dashed triangles mark $\partial T$.
Both doses use the same $150$-direction Gauss rule and each pair
of maps has a common colour scale.}}
\label{fig:triangle-field-optimum}
\end{figure}

\rev{The directional profiles in Figure~\ref{fig:triangle-field-optimum}
reflect the direction-dependent projections and chord lengths of the
triangular target. By \eqref{eq:rev-characteristic-control}, these
quantities determine the support and magnitude of the optimal inflow,
so changing the incidence direction produces a corresponding change
in the transported field. Changes of chord length at the target
vertices account for the changes of slope in the transverse profiles.
The angular integral retains the triangular geometry with the
smoothing caused by transport along complete characteristics.
The training loss decreases to $\mathcal J_{b,h}=1.499\times10^{-4}$.}

On the separate evaluation rule,
the relative state error in $L^2(\Omega\times S^1)$ is
$3.281\times10^{-2}$ and the relative control error in
$U=L^2(\Gamma^-,\,dS_x\,d\omega)$ is $3.601\times10^{-2}$.
The relative error between the two angular integrals, evaluated
with the same angular rule, is $8.581\times10^{-3}$.
Their approximate and reference maxima on the evaluation nodes
are $2.2323$ and $2.2429$. The dose peak above unity is
consistent with the optimum of the prescribed field objective.

\begin{table}[htbp]
\centering
\caption{\rev{Field targets at $\alpha=10^{-4}$. Relative errors
against the characteristic optimum on the separate evaluation rule.
The control norm uses $dS_x\,d\omega$ and both doses use the same
numerical angular integral.}}
\label{tab:revised-field-errors}
{\color{black}\small
\begin{tabular}{lccc}
\toprule
Target & $\dfrac{\|u_\theta-u^*\|_H}{\|u^*\|_H}$
& $\dfrac{\|g_\theta-g^*\|_U}{\|g^*\|_U}$
& $\dfrac{\|D_h(u_\theta)-D_h(u^*)\|_{L^2(\Omega)}}{\|D_h(u^*)\|_{L^2(\Omega)}}$\\[1.2ex]
\midrule
Circle & $4.822\times10^{-2}$ & $4.883\times10^{-2}$ & $1.809\times10^{-2}$\\
Triangle & $3.281\times10^{-2}$ & $3.601\times10^{-2}$ & $8.581\times10^{-3}$\\
\bottomrule
\end{tabular}}
\end{table}

\subsection{\new{\color{black}Example: Manufactured dose-tracking optimum}}
\label{ex:manufactured-dose}
\new{\color{black}
\rev{The field-tracking benchmarks do not test the angular coupling introduced
by $\mathcal O=D$. We therefore construct an exact optimum for the
dose-tracking problem. This permits measurement of the numerical dose
error independently of the mismatch between the optimal dose and the
prescribed target. This supplements the nonradial three-dimensional verification in
Example~\ref{ex:ball-dose}.}

Let $\Omega=B_1(0)\subset\mathbb R^2$ and let $\alpha,\sigma,c>0$ be
constants. With the characteristic coordinates of
Example~\ref{ex:regular-disc-revised},
\[
 d=x\cdot\omega,\qquad \Delta=\sqrt{1-|x|^2+d^2},\qquad
 \tau_-=\Delta+d,\quad\tau_+=\Delta-d,
\]
define
\begin{equation}\label{eq:manufactured-dose-fields}
\begin{aligned}
 u^*(x,\omega)&=\frac{c\Delta}{\alpha\sigma}
       (1-e^{-2\sigma\Delta})e^{-\sigma\tau_-},&
 z^*(x,\omega)&=-\frac c\sigma(1-e^{-\sigma\tau_+}),\\
 g^*(x,\omega)&=\frac{c\Delta}{\alpha\sigma}
       (1-e^{-2\sigma\Delta})\quad\text{on }\Gamma^-,&
 X_T(x)&=D(u^*)(x)+c.
\end{aligned}
\end{equation}
Here $D$ is the unnormalised angular integral used throughout the paper.
Since $\omega\cdot\nabla_x\Delta=0$ and
$\omega\cdot\nabla_x\tau_\pm=\mp1$,
\[
 Tu^*=0,\qquad T^*z^*=-c=D(u^*)-X_T.
\]
On $\Gamma^-$, $d=-\Delta$ and $b=\Delta$, so that
$\gamma^-u^*=g^*$ and $\alpha g^*+b\gamma^-z^*=0$.
On $\Gamma^+$, $\tau_+=0$ and $\gamma^+z^*=0$.
The fields and their transport derivatives are bounded, hence
$(u^*,z^*,g^*)\in W\times W\times U$.
Thus \eqref{eq:manufactured-dose-fields} satisfies the full
dose-tracking KKT system, without adding volume forcing to either
transport equation. By Theorem~\ref{thm:corrected-wellposed}, it is the
unique optimal solution for this target.

The exact dose is radial. Writing
$\Delta_r(\vartheta)=\sqrt{1-r^2\sin{\vartheta}^2}$, we have
\begin{equation}\label{eq:manufactured-dose-integral}
 D(u^*)(r)=\frac c{\alpha\sigma}\int_0^{2\pi}
 \Delta_r(\vartheta)(1-e^{-2\sigma\Delta_r(\vartheta)})
 e^{-\sigma(\Delta_r(\vartheta)+r\cos{\vartheta})}\,d\vartheta.
\end{equation}
In particular,
\begin{equation}\label{eq:manufactured-dose-mismatch}
 \|D(u^*)-X_T\|_{L^2(\Omega)}=c\sqrt\pi.
\end{equation}
This nonzero target mismatch is intrinsic to the optimum. It is not a
discretisation or optimisation error.
}

\subsubsection{\new{\color{black}Discretisation and reference evaluation}}
\new{\color{black}
We take $\alpha=\sigma=c=1$. No radial symmetry is imposed on the
network.

The target is evaluated from \eqref{eq:manufactured-dose-integral}
in double precision, independently of the training directions.
After symmetry reduction to $[0,\pi/2]$, Gauss rules of orders
$128$ and $256$ differ by less than $4\times10^{-15}$ at the
evaluation radii. This is an order-doubling check of the reference
evaluation. The target is therefore not defined using the training
approximation $D_h(u^*)$.
}

\subsubsection{\new{\color{black}Recovery of the optimal dose}}
\new{\color{black}
Table~\ref{tab:manufactured-dose-errors} gives the final errors on the
finer rule. The dose error compares $D_h(u_\theta)$ with the exact
angular integral \eqref{eq:manufactured-dose-integral}. It includes
the numerical angular integration of the learned field.
The residual is $2.502\times10^{-5}$ on the training rule and
$2.530\times10^{-5}$ on the finer rule. On the latter,
the unweighted squared optimality residual is
$8.006\times10^{-6}$, giving $\mathcal J_{U,h}=2.825\times10^{-5}$.
}

\rev{On this same evaluation rule, let $\widehat E_b^2$ be the sum
of the squared absolute state, adjoint and $G_-$ control errors, and
let $\widehat E_U^2$ use the $U$ control error instead. Then
\[
\begin{aligned}
\widehat E_b^2&=9.411\times10^{-6},&
\frac{\widehat E_b^2}{\mathcal J_{b,h}}&=0.3719,\\
\widehat E_U^2&=1.642\times10^{-5},&
\frac{\widehat E_U^2}{\mathcal J_{U,h}}&=0.5813.
\end{aligned}
\]
These are observed final-checkpoint ratios, not estimates of the
uniform reliability constants. The saved loss history does not
contain intermediate solution errors, so it does not establish
residual--error correspondence throughout optimisation.}

\begin{table}[!htbp]
\centering
\caption{\new{\color{black}Manufactured dose tracking: errors at the final
checkpoint, evaluated on the finer deterministic quadrature. Here $v_\theta$ and $v_{\rm ref}$ are the approximate and reference quantities in each row, and $X$ is the indicated norm.}}
\label{tab:manufactured-dose-errors}
\new{\color{black}
\begin{tabular}{lccc}
\toprule
Quantity & Norm & \shortstack{$\|v_\theta-v_{\rm ref}\|_X$}
& \shortstack{$\dfrac{\|v_\theta-v_{\rm ref}\|_X}{\|v_{\rm ref}\|_X}$}\\
\midrule
$u_\theta-u^*$ & $H$ & $1.390\times10^{-3}$ & $8.375\times10^{-4}$\\
$z_\theta-z^*$ & $H$ & $2.274\times10^{-3}$ & $9.059\times10^{-4}$\\
$g_\theta-g^*$ & $U$ & $3.053\times10^{-3}$ & $1.202\times10^{-3}$\\
$g_\theta-g^*$ & $G_-$ & $1.519\times10^{-3}$ & $6.370\times10^{-4}$\\
$D_h(u_\theta)-D(u^*)$ & $L^2(\Omega)$ & $7.438\times10^{-4}$ & $2.042\times10^{-4}$\\
\bottomrule
\end{tabular}}
\end{table}

\begin{figure}[!htbp]
\centering
\includegraphics[width=.98\linewidth]{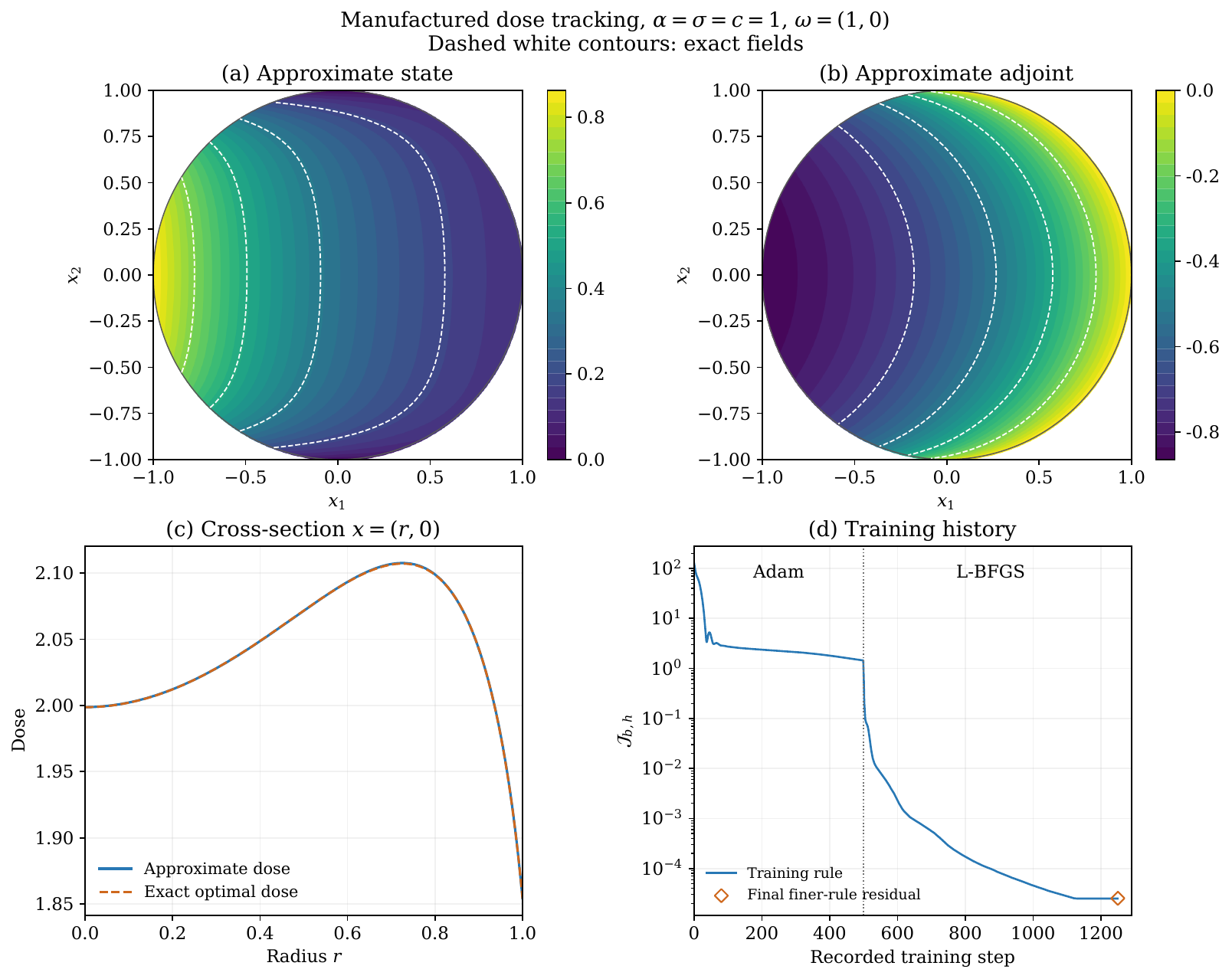}
\caption{\new{\color{black}Manufactured dose tracking with
$\alpha=\sigma=c=1$. Top: \rev{Approximate} state and adjoint for
$\omega=(1,0)$. Dashed white curves are contours of the exact fields.
Bottom left: computed and exact optimal doses along $x=(r,0)$. The computed dose uses $150$ angular
nodes and direct network evaluation, without radial averaging.
Bottom right: the full recorded training loss and the final residual
on the finer rule.}}
\label{fig:manufactured-dose-verification}
\end{figure}
\new{\color{black}
\rev{Figure~\ref{fig:manufactured-dose-verification} shows the directional
fields and the radial optimal dose. The target mismatch is $1.772440$,
close to $\sqrt\pi\simeq1.772454$, while the absolute dose error is
$7.438\times10^{-4}$. Together with Table~\ref{tab:manufactured-dose-errors},
this verifies recovery of the dose-tracking optimum at $\alpha=1$.}
}
{\color{black}
\subsection{\rev{Example: Field and dose tracking of a circular target}}
\label{ex:localised}
Let $\Omega=B_1(0)\subset\mathbb R^2$, $T=B_{0.3}(0)$ and
$q=\chi_T$. We take $\sigma=1$ and $\alpha=10^{-4}$ and compare
angular-field tracking with direct dose tracking on the same geometry.
The objectives are
\begin{equation}\label{eq:disc-matched-objectives}
\begin{aligned}
F_{\mathrm{field}}(u,g)
 &=\frac12\|u-q\|_{L^2(\Omega\times S^1)}^2
   +\frac\alpha2\|g\|_U^2,\\
F_{\mathrm{dose}}(u,g)
 &=\frac12\|D(u)-q\|_{L^2(\Omega)}^2
   +\frac\alpha2\|g\|_U^2,
\qquad D(u)(x)=\int_{S^1}u(x,\omega)\,d\omega.
\end{aligned}
\end{equation}
Both problems satisfy $\omega\cdot\nabla u+\sigma u=0$,
$\gamma^-u=g$, $\gamma^+z=0$ and $\alpha g+b\gamma^-z=0$,
where $b=|\omega\cdot n|$ and
$U=L^2(\Gamma^-,\,dS_x\,d\omega)$. The adjoint sources are
$u-q$ for field tracking and $D(u)-q$ for dose tracking.
The field target is unscaled. Thus the two objectives prescribe the
same indicator for different observations, and their optima need
not have the same angular integral.

\rev{The two computations have identical initial network parameters and
quadrature nodes. Their common settings are given in
Table~\ref{tab:experiment-settings}.}

\begin{figure}[htbp]
\centering
\includegraphics[width=.86\linewidth]{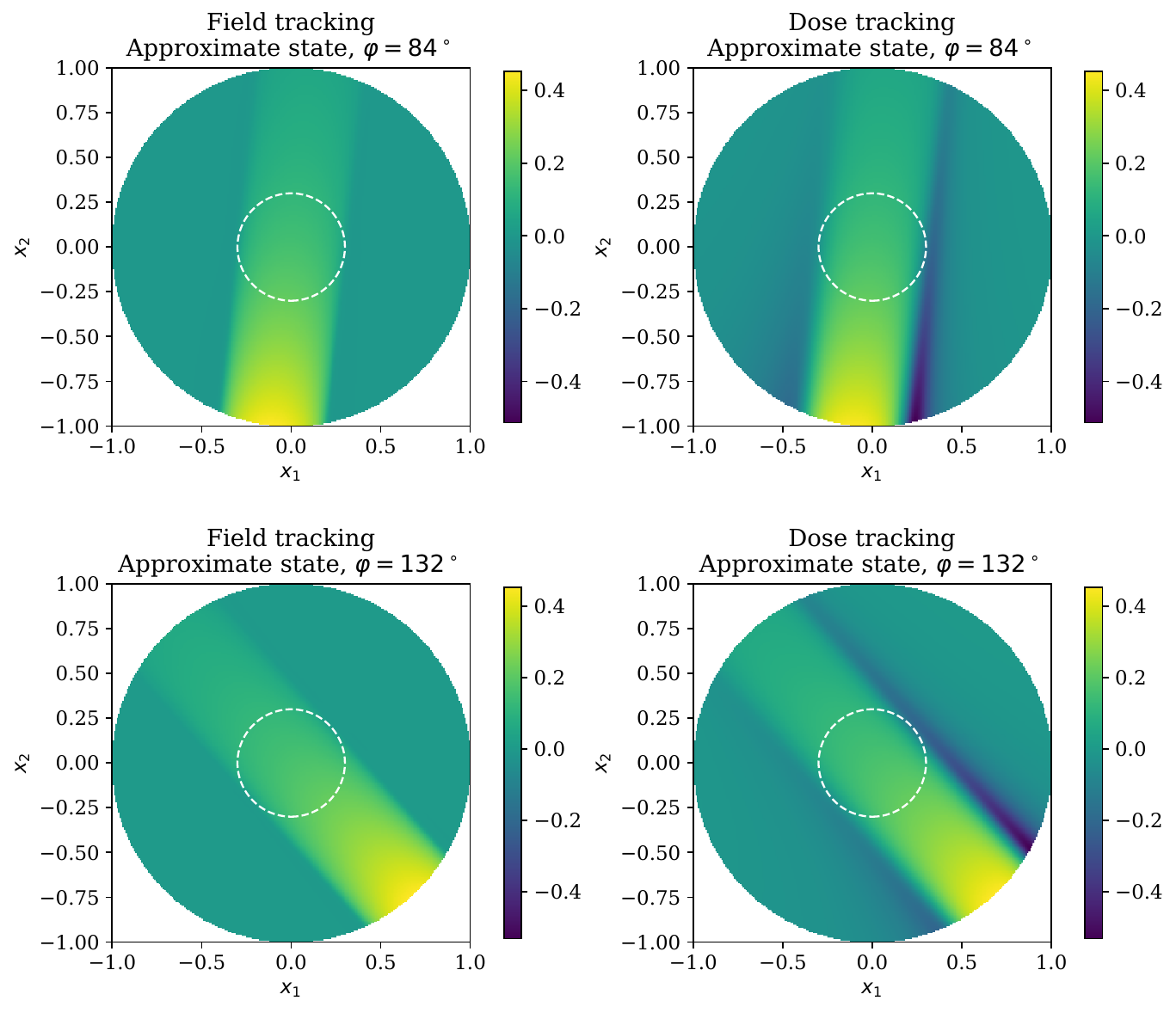}
\caption{\rev{Circular target in the unit disc. Approximate states
for field tracking and dose tracking at $84^\circ$ and $132^\circ$,
evaluated directly from the final checkpoints. Each row uses a common
colour scale. The dashed circle marks $\partial T$.}}
\label{fig:disc-matched-states}
\end{figure}

\rev{Figure~\ref{fig:disc-matched-states} shows the directional fields for
the two objectives. The field-tracking optimum follows from
\eqref{eq:rev-characteristic-control}, while dose tracking couples the directions through the adjoint source $D(u)-q$.}
\begin{figure}[htbp]
\centering
\includegraphics[width=.95\linewidth]{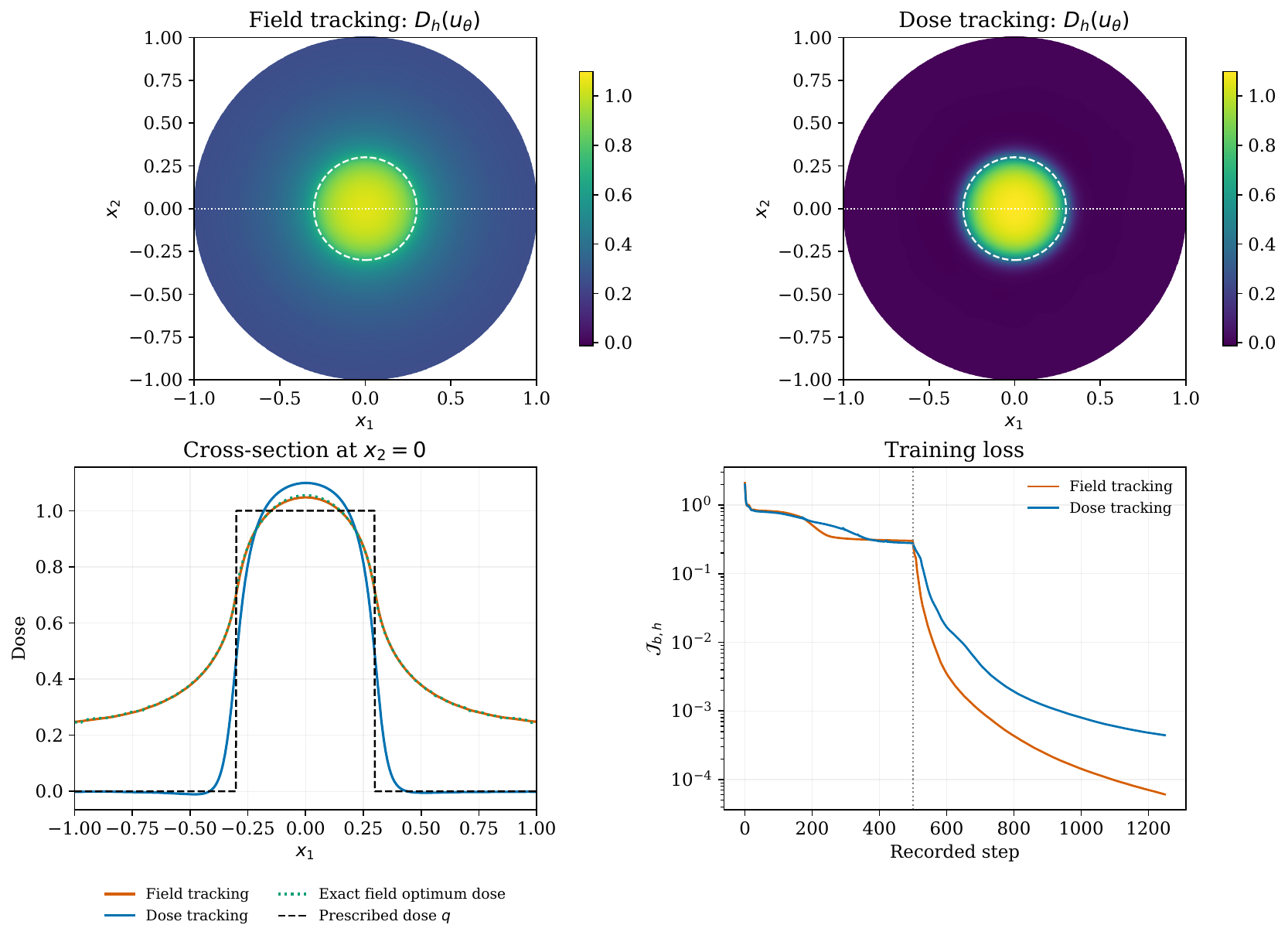}
\caption{\rev{Field and dose tracking on common uniform quadrature
in the unit disc. The upper panels show their angular integrals on a
common colour scale. The dashed circle marks $\partial T$.
The lower panels show direct checkpoint evaluations along $x_2=0$
and the training histories. The exact field-optimum dose is obtained
by integrating the characteristic state over the same $150$
evaluation directions. The indicator $q$ is the prescribed dose
for the dose-tracking problem.}}
\label{fig:sharp-dose}
\label{fig:disc-matched-doses}
\end{figure}
The relative target mismatch $e_q$ and the residuals for each optimality system are reported in Table~\ref{tab:disc-matched}.
\begin{table}[htbp]
\centering
\caption{\rev{Final diagnostics for the common-setting disc computations.
The relative target mismatch is
$e_q=\|D_h(u_\theta)-q\|_{L^2(\Omega)}/\|q\|_{L^2(\Omega)}$.
The control norm and target mismatch use the separate evaluation rule.
The residuals correspond to the respective field and dose KKT systems.}}
\label{tab:disc-matched}
{\color{black}\small
\begin{tabular}{lrrrr}
\toprule
Objective & $e_q$ & $\|g_\theta\|_U^2$
& $\mathcal J_{b,h}$ & Evaluation $\mathcal J_b$\\
\midrule
Field & $1.05939$ & $0.53327$ & $6.106\times10^{-5}$ & $4.292\times10^{-3}$\\
Dose & $0.27574$ & $0.82241$ & $4.439\times10^{-4}$ & $1.604\times10^{-2}$\\
\bottomrule
\end{tabular}
}
\end{table}

A broader angular integral for field tracking and a more concentrated dose for direct dose
tracking are shown in Figure~\ref{fig:disc-matched-doses} . The field-optimum curve distinguishes the broader transport response from the approximation error. In particular, $e_q$ is a target mismatch rather than an error to the field-tracking optimum. For dose tracking, the transition across $\partial T$
remains smooth and the relative target mismatch is $0.27574$. The larger residuals on the evaluation rule are mainly due to the adjoint component, whose source contains the discontinuous target. The two training histories concern different KKT systems and are interpreted with their respective objectives.

}

 \subsection{\rev{Residual approximation under Monte Carlo quadrature}}
\label{sec:mc-residual-experiment}
\new{\rev{We test Monte Carlo quadrature with fixed fields and an exactly known
residual integral. On the unit disc, let $(u^*,z^*,g^*)$ be the solution
of Example~\ref{ex:regular-disc-revised}, with $\alpha=\sigma=1$, and set}
}\begin{equation}\new{\label{eq:mc-perturbation}
 u_\varepsilon=u^*+\varepsilon x\cdot\omega,\qquad
 z_\varepsilon=z^*+\varepsilon|x|^2,\qquad
 g_\varepsilon=g^*+\varepsilon(1+x_1),
 \qquad \varepsilon=0.05.
}\end{equation}

\new{Here $\varepsilon$ is a fixed perturbation amplitude that scales the
residual functional by $\varepsilon^2$. The target is unchanged.
Writing $d=x\cdot\omega$, the two interior
residuals are $\varepsilon(1+d)$ and
$\varepsilon(|x|^2-3d)$. On $\Gamma^-$, the optimality and inflow
residuals are $\varepsilon(1+x_1+b)$ and
$-\varepsilon(1+x_1+b)$, respectively, while the outflow residual is
$\varepsilon$. Direct integration therefore gives
}\begin{equation}\new{\label{eq:mc-reference}
 \mathcal J_b(u_\varepsilon,z_\varepsilon,g_\varepsilon)
 =\frac{\varepsilon^2}{6}(35\pi^2+64\pi)
 \simeq 0.2277075349.
}\end{equation}

\new{For each $M$, we take $M$ independent uniform samples in
$\Omega\times S^1$ and, independently, $M$ samples from the probability
measure $(4\pi)^{-1}b\,dS\,d\omega$ on $\Gamma^-$. The sample means
are multiplied by $2\pi^2$ and $4\pi$, respectively. The constant
outflow contribution is integrated exactly. We use
$M=2^8,2^{10},\ldots,2^{18}$ and repeat the computation independently
$256$ times. Within each repetition, the smaller sample sets are
prefixes of the same unsorted independent sample pools.
}

\new{\rev{Figure~\ref{fig:mc-controlled} shows the root mean square of
$\mathcal J_{b,M}-\mathcal J_b$ over the repetitions. It decreases
from $6.89\times10^{-3}$ to $2.28\times10^{-4}$ over the tested
range, with a fitted exponent $0.492$. This agrees with the
$M^{-1/2}$ dependence for fixed fields in \eqref{eq:mc-rms-bound}.
Keeping the fields fixed makes this a direct test of
Lemma~\ref{lem:mc-rms}: the residual density and its variance remain
unchanged as $M$ increases. The repeated computations estimate the
mean-square sampling error, whose rate need not be visible in a
single realisation. The experiment thus isolates the quadrature
contribution in the error analysis. During network training the fields
depend on the samples, and the uniform estimate of
Proposition~\ref{prop:uniform-quadrature} is needed to control that
dependence.}
}\begin{figure}[htbp]
\centering
\new{\includegraphics[width=0.75\linewidth]{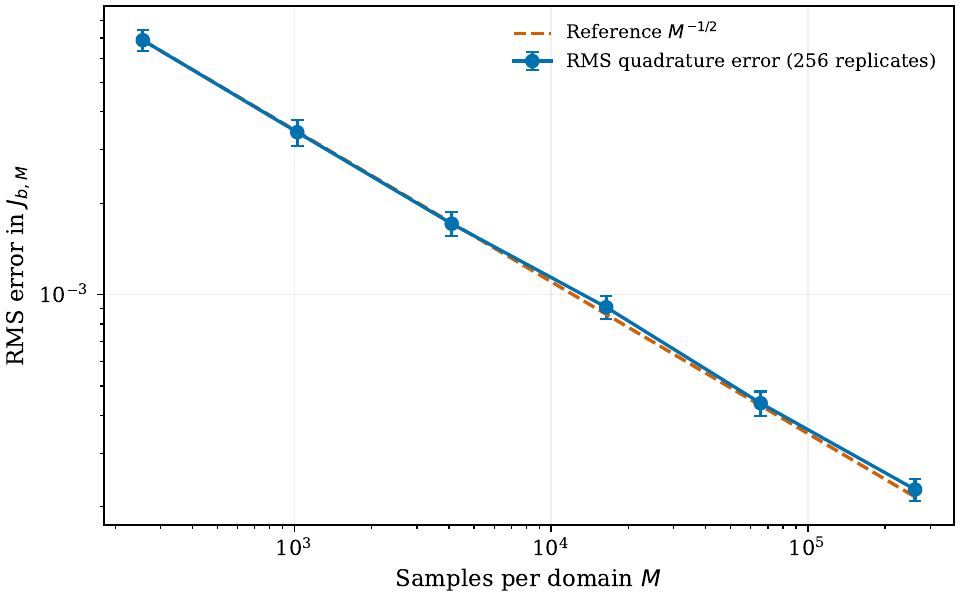}
}\caption{\new{Monte Carlo approximation of the residual for the fixed fields
\eqref{eq:mc-perturbation}. The error is computed against the exact
value \eqref{eq:mc-reference}. Each point uses $256$ independent
repetitions. The bars show $1.96$ estimated standard errors of the
reported root mean square, computed by the delta method. The dashed
line has slope $-1/2$. Here $M$ denotes the number of samples in each
of the interior and inflow domains.}}
\label{fig:mc-controlled}
\end{figure}

\color{black}

\section{\rev{Proofs of the Monte Carlo quadrature estimates}}
\label{app:quadrature-proofs}
\rev{We use the residual densities and sampling measures introduced
in Subsection~\ref{subsec:mc-quadrature}. The estimates allow samples
to be shared between different residual terms.}

{\color{black}
\begin{proof}[\rev{Proof of Lemma~\ref{lem:mc-rms}}]
\rev{Set $I_k=\int_{D_k}F_k(V,y)\,d\mu_k(y)$ and
$X_{k,j}=R_k(V,Y_{k,j})-I_k$. For each fixed $k$, independence and
centering give
}\[
\rev{\mathbb{E}\rev{[}X_{k,j}\rev{]}=0,\qquad
\mathbb{E}\rev{[}X_{k,j}X_{k,\ell}\rev{]}=0\quad(j\ne\ell),\qquad
\mathbb{E}\rev{[}X_{k,j}^2\rev{]}=\operatorname{Var}_{\nu_k}(R_k(V,\cdot)).}
\]
\rev{Consequently, with $\Delta_k=M_k^{-1}\sum_jX_{k,j}$,
}\[
\rev{\mathbb{E}\rev{[}\Delta_k\rev{]}=0,\qquad
\mathbb{E}\rev{[}|\Delta_k|^2\rev{]}
=\frac{1}{M_k^2}\sum_{j=1}^{M_k}\mathbb{E}\rev{[}X_{k,j}^2\rev{]}
=\frac{\operatorname{Var}_{\nu_k}(R_k(V,\cdot))}{M_k}.}
\]
\rev{Since $\mathcal J_{b,M}(V)-\mathcal J_b(V)=\sum_k\Delta_k$,
unbiasedness follows by linearity of expectation. Minkowski's
inequality yields
}\[
\rev{\left(\mathbb{E}\rev{\bigl[}\left|\sum_k\Delta_k\right|^2\rev{\bigr]}\right)^{1/2}
\le\sum_k(\mathbb{E}\rev{[}|\Delta_k|^2\rev{]})^{1/2},}
\]
\rev{which proves \eqref{eq:mc-rms-bound} without requiring
independence between different residual terms.}
\end{proof}
}

{\color{black}
\begin{proof}[\rev{Proof of Proposition~\ref{prop:corrected-validation}}]
Define
\[
\rev{\Delta_k=\frac1{M_k}\sum_{j=1}^{M_k}R_k(V,Y_{k,j})
-\int R_k(V,y)\,d\nu_k(y).}
\]
\rev{Hoeffding's inequality \cite{revHoeffding1963} gives, for
$B_k>0$ and $t>0$,
}\[
\rev{\mathbb P(|\Delta_k|>t)
\le 2\exp\left(-\frac{2M_kt^2}{B_k^2}\right).}
\]
\rev{Choose $t_k=B_k\sqrt{\log(2K/\delta)/(2M_k)}$.
If $B_k=0$, then $\Delta_k=t_k=0$ almost surely. Hence
}\[
\rev{\mathbb P\bigl(\exists k:\ |\Delta_k|>t_k\bigr)
\le\sum_{k:B_k>0}2\exp\left(-\frac{2M_kt_k^2}{B_k^2}\right)
\le\delta.}
\]
\rev{On the complementary event,
}\[
\rev{|\mathcal J_{b,M}(V)-\mathcal J_b(V)|
=\left|\sum_k\Delta_k\right|\le\sum_k t_k.}
\]
\rev{This proves \eqref{eq:mc-concentration}. Substitution into
\eqref{eq:mc-aposteriori-decomp} gives \eqref{eq:corrected-validation}.}
\end{proof}
}

{\color{black}
\begin{proof}[\rev{Proof of Proposition~\ref{prop:uniform-quadrature}}]
\rev{Let $\{\theta_\ell\}_{\ell=1}^{\mathcal N_N(\eta)}$ be the net
and define
}\[
\rev{\Delta_k(\theta)=\frac1{M_k}\sum_{j=1}^{M_k}R_k(\theta,Y_{k,j})
                    -\int R_k(\theta,y)\,d\nu_k(y).}
\]
\rev{For $\|\theta-\theta_\ell\|\le\eta$, the common Lipschitz bound
implies, almost surely,
}\[
\rev{|\Delta_k(\theta)-\Delta_k(\theta_\ell)|
\le\frac1{M_k}\sum_{j=1}^{M_k}L_k\eta+\int L_k\eta\,d\nu_k
=2L_k\eta.}
\]
\rev{It follows that
}\[
\rev{\sup_{\theta\in\Theta_N}\left|\sum_k\Delta_k(\theta)\right|
\le 2\eta\sum_kL_k+\sum_k\max_\ell|\Delta_k(\theta_\ell)|.}
\]
\rev{Set
$t_k=B_k\sqrt{\log(2K\mathcal N_N(\eta)/\delta)/(2M_k)}$.
Applying Hoeffding's inequality at the net points gives
}\[
\rev{\mathbb P\bigl(\exists k,\ell:\ |\Delta_k(\theta_\ell)|>t_k\bigr)
\le\sum_{k:B_k>0}\sum_{\ell=1}^{\mathcal N_N(\eta)}
2\exp\left(-\frac{2M_kt_k^2}{B_k^2}\right)\le\delta.}
\]
\rev{Terms with $B_k=0$ vanish. On the complementary event, inserting
these thresholds into the preceding supremum bound proves
\eqref{eq:corrected-uniform-MC}.}
\end{proof}
}

 \printbibliography

\end{document}